\documentclass{mcom-l}

\allowdisplaybreaks[4]
\usepackage{color}
\usepackage{url}

\newcommand{\Jac}{Df(x)}
\newcommand{\im}{\mathrm{im}}
\newtheorem{claim}{Claim}
\newtheorem{definition}{Definition}
\newtheorem{lem}{Lemma}
\newtheorem{thm}{Theorem}
\newtheorem{remark}{Remark}

\newtheorem{example}{Example}
\newcommand{\aaa}{\zeta}
\newcommand{\bbb}{\eta}
\newcommand{\Sigg}{\Sigma_{n-1}}

\newcommand{\Ax}{z}
\newcommand{\Bx}{y}
\newcommand{\Cx}{w}
\newcommand{\Dx}{x}
\newcommand{\Exi}{\xi}
\newcommand{\Fxi}{\zeta}
\newcommand{\Ff}{f}
\newcommand{\Sf}{g}

\newcommand{\Nf}{h}
\newcommand{\Nb}{u}
\newcommand{\coeA}{b_1}
\newcommand{\coeB}{b_2}
\usepackage{amssymb}
\newcommand{\coefhat}{l_1}
\newcommand{\coefn}{l_2}
\newcommand{\kappatwo}{\hat{\gamma}_{\mu}(\Ff,\Ax)\|\Exi-\Ax\|}
\newcommand{\coeftemp}{l_3}

\newcommand{\Sig}{D \hat{f}(x)}

\newtheorem{algorithm}{Algorithm}
\newcommand{\CC}{\mathbb{C}}
\newcommand{\xx}{x}
\newcommand{\yy}{y}

\newcommand{\mux}{\mu}

\newcommand{\ran}{\mathrm{span}}
\newcommand{\br}{\kappa}
\newcommand{\dps}{\rho}
\newcommand{\degree}{\mathrm{deg}}
\newcommand{\img}{\mathrm{im}}

\newcommand{\Amu}{\Delta_{\mu}(f_n)}
\newcommand{\Athree}{\Delta_{3}(f_n)}

\newcommand{\bigO}{O}

\allowdisplaybreaks[4]
\newcommand{\app}{z}
\newcommand{\acc}{\xi}

\newcommand{\minus}{z-\xi}
\newcommand{\minusa}{z_1-\xi_1}
\newcommand{\minusb}{\hat{y}-\hat{\xi}}
\newcommand{\Sigp}{D\hat{f}(z)}
\newcommand{\Sigc}{D\hat{f}(\xi)}

\newcommand{\Delp}{\frac{\partial^2f_n(\app)}{\partial X_1^2}}
\newcommand{\Delc}{\frac{\partial^2f_n(\acc)}{\partial X_1^2}}
\newcommand{\xxi}{\xx}

\newcommand{\newd}{\Gamma}
\newcommand{\atwo}{\Sigp^{-1}\frac{1}{2}\frac{\partial^2 \hat{f}(\app)}{\partial X_1^2}}

\newcommand{\simplex}{z}
\newcommand{\azhat}{\hat z}
\newcommand{\azone}{\app_1}
\newcommand{\bzhat}{\hat y}
\newcommand{\newl}{L}

\newcommand{\acoef}{a_2(\itb)}
\newcommand{\deltacoef}{a_3(\itb)}

\newcommand{\btwoone}{b_{2,1}(\itb)}
\newcommand{\btwotwo}{b_{2,2}(\itb)}
\newcommand{\btwothree}{b_{2,3}(\itb)}
\newcommand{\btwofour}{b_{2,4}(\itb)}
\newcommand{\bthreethree}{b_{3,3}(\itb)}
\newcommand{\bthreefour}{b_{3,4}(\itb)}

\newcommand{\itb}{u}
\usepackage{algorithm}
\usepackage{algorithmic}

\usepackage{enumerate}

\numberwithin{equation}{section}
\newcommand{\xidag}{\xi}
\begin{document}

 \title{Computing Simple Multiple  Zeros  of Polynomial Systems}



\author{Zhiwei  Hao}
\address{Key Laboratory of Mathematics Mechanization\\
Academy of Mathematics and System Sciences\\
Beijing 100190, China}
\email{haozhiwei@mmrc.iss.ac.cn}
\thanks{}

\author{Wenrong Jiang}
\address{Key Laboratory of Mathematics Mechanization\\
Academy of Mathematics and System Sciences\\
Beijing 100190, China}
\email{jiangwr@amss.ac.cn}
\thanks{}

\author{Nan Li}
\address{ Center for Applied Mathematics\\
      Tianjin University Nankai District \\
       Tianjin 300072, China}
\email{nan@tju.edu.cn}
\thanks{}

\author{Lihong Zhi}
\address{Key Laboratory of Mathematics Mechanization\\
School of Mathematical Sciences\\
University of Chinese Academy of Sciences\\
Academy of Mathematics and System Sciences\\
Beijing 100190, China}
\email{lzhi@mmrc.iss.ac.cn}
\thanks{This research was supported in part by the National Key  Research Project of China 2016YFB0200504 (Zhi), the National Natural Science Foundation of China under Grants 11571350 (Zhi), 11626250 (Li) and 11601378 (Li).}



\dedicatory{}

\begin{abstract}

Given a  polynomial system $f$ associated with a simple multiple zero $x$ of multiplicity $\mu$, we give a computable lower bound on the minimal distance between the simple multiple zero $x$ and  other zeros of $f$. If $x$ is only given with limited accuracy,  we propose a numerical criterion that $f$ is certified to have $\mu$ zeros (counting multiplicities) in a small  ball   around $x$. Furthermore, for simple double zeros and simple triple zeros whose Jacobian is of normalized form, we define modified Newton iterations and prove the quantified quadratic convergence when the starting point is close to the exact simple multiple zero.  For simple multiple zeros of arbitrary multiplicity whose Jacobian matrix may not have a normalized form, we perform unitary transformations and  modified Newton iterations,  and prove its non-quantified quadratic convergence and  its  quantified convergence for simple triple zeros.
\end{abstract}

\maketitle

\tableofcontents

\section{Introduction}

  Consider an ideal $I_f$ generated by a polynomial system $f=\{f_1,
\ldots, f_n\}$, where $f_i\in\mathbb{C}[X_1, \ldots, X_n]$.   An isolated zero of multiplicity $\mux$ for
 $f$ is a point $\xxi\in \CC^n$ such that
\begin{enumerate}
  \item $f(\xxi)=0$,
  \item there exists a ball $B(\xxi,r)$ of radius $r>0$ such that $B(\xxi,r) \cap f^{-1}(0)=\{\xxi\}$,
  \item $\mux=\dim(\CC[X]/Q_{f,\xxi})$,
\end{enumerate}
where
\[B(\xxi,r):=
  \{\yy\in\mathbb{C}^n:\|\yy-\xxi\|<r\},\]
   and $Q_{f,\xxi}$ is an isolated primary component of the ideal
$I_f$ whose associate prime is
\[m_{\xxi}=(X_1-\xxi_1,\ldots,X_n-\xxi_n).\]
In \cite{Dedieu2001On}, based on Rouch\'e's Theorem \cite{Berenstein1993Residue},
 the condition (3) is replaced by
\begin{enumerate}
\item[(3a)] a generic analytic $g$ sufficiently close to $f$ has $m$ simple zeros in $B(\xxi, r)$.
\end{enumerate}

We recall $\alpha$-theory below  according to \cite{Blum:1997} and refer to \cite{Smale1981The,Shub1985Computational,Shub1986Computational,Smale1986Newton,ShubSmale:1996,Wang1990On,
Hauenstein:2012} for more details.    
 Let $Df(x)$ denote the Jacobian matrix of $f$ at $\xx$. Suppose $Df(x)$ is invertible, $\xx$ is called  a simple zero of $f$.  The  Newton's  iteration is defined by
   \begin{equation}
   N_f(x)=x-Df(x)^{-1} f(x).
   \end{equation}

  Shub and Smale \cite{ShubSmale:1996} defined
   \begin{equation}\label{gamma}
\gamma(f,\xxi)=\sup_{k\geq 2}\left\|Df(\xxi)^{-1}\cdot\frac{D^k f(\xxi)}{k!}\right\|^{\frac{1}{k-1}},
\end{equation}
where  $D^{k} f$ denotes the $k$-th derivative of $f$ which is a symmetric tensor whose components are the partial derivatives of $f$ of order $k$, $\|\cdot\|$ 
denotes the classical operator norm.

    According to \cite[Theorem 1]{Blum:1997}, if
\begin{equation}\label{approxx}
\|z-\xxi\|\leq\frac{3-\sqrt{7}}{2\gamma(f,\xxi)},
\end{equation}
then Newton's iterations
starting at $z$ will converge quadratically to the simple zero $\xxi$.

If $y$ is another zero of $f$, according to  \cite[Corollary 1]{Blum:1997}, we have
\begin{equation}\label{lowerbound}
\|y-\xxi\|\geq\frac{5-\sqrt{17}}{4 \gamma(f,\xxi)}, 
\end{equation}
which  separates the  simple zero $\xxi$ from other zeros of $f$.

Furthermore, according to \cite[Theorem 2]{Blum:1997},  if only a system $f$ and a point $x$ are given such that
\begin{equation}
\alpha(f,x)\leq\frac{13-3\sqrt{17}}{4}\approx0.157671,
\end{equation}
where $\alpha(f,x)=\beta(f,x)  \gamma(f,x)$ and
\[\beta(f,x)=\left\|x-N_f(x)\right\|=\|Df(x)^{-1} f(x)\|,\]
 then Newton's iterations starting at $x$ will converge quadratically to a simple zero $\xi$ of $f$ and \[\|x-\xi\| \leq 2 \beta(f,x).\]

In \cite{Dedieu2001On}, Dedieu and Shub gave quantitative  results for simple double zeros satisfying $f(\xxi)=0$ and
\begin{enumerate}
  \item[(A)]
  $\dim \ker Df(\xxi)=1$,
  \item[(B)]
$D^2 f(\xxi)(v,v) \notin \im Df(\xxi)$,
  \end{enumerate}
where $\ker Df(\xxi)$ is spanned by a unit  vector $v\in \CC^n$.   They generalized the definition of $\gamma$  (\ref{gamma}) to
\begin{equation}\label{gamma2}
\gamma_2(f,\xxi)=\max\left(1, \sup_{k\geq 2}\left\|A(f,\xxi, v)^{-1}\cdot\frac{D^k f(\xxi)}{k!}\right\|^{\frac{1}{k-1}} \right),
\end{equation}
where
\begin{equation}
A(f,x,v)=Df(x).+\frac{1}{2}D^2f(x)(v,\Pi_v),
\end{equation}
 is a linear operator which is invertible at the {simple} double zero $\xxi$, and $\Pi_v$ denotes {the} Hermitian projection onto the subspace $[v] \subset \CC^n$.

In \cite[Theorem 1]{Dedieu2001On}, Dedieu and Shub also presented  a lower bound for separating simple double zeros $\xxi$ from the other zeros $y$ of $f$,
\begin{equation}\label{lowerbounddouble}
\|y-\xxi\|\geq\frac{d}{ 2\gamma_2(f,\xxi)^2}, 
\end{equation}
where $d \approx 0.2976$ is a positive real root of
\begin{equation}\label{rootd2}
\sqrt{1-d^2}-2d\sqrt{1-d^2}-d^2-d=0.
\end{equation}

\begin{remark}
There are two  typos  in the statements of \cite[Theorem 1]{Dedieu2001On} and \cite[Lemma 4]{Dedieu2001On}: 1) the degree of $\gamma_2$ in (\ref{lowerbounddouble}) is $2$ instead of $1$ in \cite[Theorem 1]{Dedieu2001On}; 2)  the coefficient of the second $\sqrt{1-d^2}$ in (\ref{rootd2}) is $-2d$ instead of $-d$ in \cite[Lemma 4]{Dedieu2001On}.
\end{remark}

In \cite[Theorem 4]{Dedieu2001On}, Dedieu and Shub  showed that if the following criterion is satisfied at  a given  point $x$  and a given vector $v$
 \begin{equation}\label{rouche2}
 \|f(x)\|+\|Df(x) v\|\frac{d}{4 {\gamma}_2(f,x,v)^2} <\frac{d^3}{32{\gamma}_2^4\|B(f,x,v)^{-1}\|},
  \end{equation}
   then   $f$ has two zeros in the ball
   of radius
   \begin{equation}
   \frac{d}{4{ \gamma_2(f,x)}^2},
   \end{equation}
    around $x$.  Let us set
 \begin{equation*}
 B(f,x,v)=A(f,x,v)-L,
 \end{equation*}
 {where}  $L(v)=Df(x) v$, $L(w)=0$ for $w \in v^\bot$, {and}
 \begin{equation}
  \gamma_2(f,x)= \max \left(1, \sup_{k\geq 2}\left\|B(f, x,v)^{-1}\cdot\frac{D^k f(x)}{k!}\right\|^{\frac{1}{k-1}}\right).
  \end{equation}

\paragraph*{\bf{Our Contributions}}

We  generalize Dedieu and Shub's quantitative results  about simple double zeros  to simple multiple zeros whose Jacobian matrix has corank one.

Let us recall our previous work on computing simple multiple zeros.   Suppose $\xxi$ is an isolated singular  zero  of $f$  satisfying  {$\dim \ker Df(\xxi)=1$}.  
 Let $\mathcal{D}_{f,\xxi}$ denote the local dual space of an ideal  $I_f=(f_1, \ldots, f_n)$ at $\xx$:
\begin{equation}\label{dualspace}
\mathcal{D}_{f,\xxi}:=\{\Lambda\in
\mathfrak{D}_{\xxi}\ \ |\ \  \Lambda(g)=0, ~\forall g\in I_f\},
\end{equation}
where $\mathfrak{D}_{\xx}=\ran_\mathbb{C}\{\mathbf{d}^{\alpha}_{\xx}\}$ is the $\CC$-vector space generated by differential functionals $\mathbf{d}^{\alpha}_{\xx}$ of order $\alpha\in \mathbb{N}^n$, see (\ref{dualdef}). 
Let  $\mu$ denote  the multiplicity, {then} starting from   $\Lambda_0=1$, and \[\Lambda_1=d_1+a_{1,2}d_2+\cdots+a_{1,n}d_n,\]
 where $d_1, \ldots, d_n$ are the first order differential functionals,
  we can construct
  \[\Lambda_k=\Delta_k+a_{k,2}d_2+\cdots+a_{k,n}d_n,\]
    incrementally  for $k=2, \ldots, \mu-1$   by formulas (\ref{dualbasis}) and (\ref{solveaij}), s.t.
  $\{\Lambda_0,\Lambda_1,\ldots,\Lambda_{\mu-1}\}$ is a closed basis  of the local dual space $\mathcal{D}_{f,\xxi}$ \cite{LiZhi:2009,LZ:2012}. The method is  efficient since the size of matrices involved in the computation  is bounded by $n$. 

 We  generalize the definition of   simple double zeros in \cite{Dedieu2001On}. A  simple multiple zero $\xxi$ of $f$ satisfies   $f(\xxi)=0$  and
\begin{enumerate}
  \item[(A)]
  $\dim\ker Df(\xxi)=1$,

  \item[(B)] $\Delta_{k}(f)\in\img\  Df(\xxi)$, for $k=2,\ldots,\mux-1$,

  \item[(C)]
  $\Delta_{\mux}(f)\notin\img\ Df(\xxi)$.
\end{enumerate}
Without loss of generality, we can  assume  that $\Jac$ has a normalized form: 
\begin{eqnarray}\label{normalizeform1}
       \Jac=\left(
               \begin{array}{cc}
                 0 & D\hat{f}(x) \\
                 0 & 0 \\
               \end{array}
             \right),
\end{eqnarray}
   where $D\hat{f}(x)$ is the nonsingular Jacobian {matrix} of polynomials $\hat{f}=\{f_1, \ldots, f_{n-1}\}$ with respect to variables $\hat{X}=\{X_2, \ldots, X_n\}$.    We will show in Section \ref{subsectionnormalform} that it is always possible to perform unitary transformations  to $f$ and variables $X$ to obtain an equivalent polynomial system whose Jacobian matrix at the singular solution has the  normalized form (\ref{normalizeform1}). This normalization step is similar to the reduction to one variable technique used in \cite{GLSY:2007}.

If $\xxi$ is a simple multiple  zero  of $f$  of  multiplicity $\mux$  and $Df(\xxi)$ has the normalized form (\ref{normalizeform1}), then it is clear that
\begin{align*}
  \Delta_{k}(f)\in\img\ Df(\xxi) & \Leftrightarrow  \Delta_{k}(f_n)=0,
\end{align*}
and the above (B) and (C) conditions can be simplified to

\begin{enumerate}
  \item[(B)]
  $\Delta_{k}(f_n)=0$, for $k=2,\ldots,\mux-1$,
  \item[(C)]
  $\Delta_{\mux}(f_n)\neq0$.
\end{enumerate}

  For a simple multiple zero  $\xxi$ satisfying conditions {(A),(B) and (C)}, we  generalized the definition of $\gamma_2$ in  (\ref{gamma2}) to
\begin{equation}\label{gammamu}
  {\gamma}_{\mu}=\gamma_{\mu}(f, \xxi)=\max(\hat{\gamma}_{\mu},\gamma_{\mu,n}),
  \end{equation}
  where
\begin{equation}\label{gammahat}
\hat{\gamma}_{\mu}=\hat{\gamma}_{\mu}(f,\xxi)=\max\left(1,\ \sup_{k\geq 2}\left\|D\hat{f}(\xxi)^{-1}\frac{D^k \hat{f}(\xxi)}{k!}\right\|^{\frac{1}{k-1}}\right),
\end{equation}
and
\begin{equation}\label{gammamun}
{\gamma}_{\mu,n}={\gamma}_{\mu,n}(f,\xxi)=\left(1,\ \sup_{k\geq 2}\left\|\frac{1}{\Amu}\cdot\frac{D^k f_n(\xxi)}{k!}\right\|^{\frac{1}{k-1}}\right),
\end{equation}
where  $D^{k} \hat f(x)$ for  $k\geq 2$  denote the partial derivatives of $\hat {f}$ of order $k$ with respect to $X_1, X_2, \ldots, X_n$ evaluated at $x$. We generalize main results  in \cite{Dedieu2001On} to simple multiple zeros of higher multiplicities.

Firstly, in Theorem \ref{thm5}, we present  a lower bound for separating simple multiple zeros $\xxi$ of multiplicity $\mu \geq 2$ from {another} zero $y$ of $f$,
\begin{equation}\label{lowerboundmultiple}
\|y-\xxi\|\geq\frac{d}{ 2\gamma_{\mu}(f,\xxi)^{\mu}}, 
\end{equation}
where $d$ is a positive real root of a univariate polynomial $p(d)$ defined in  (\ref{mainunivariatepoly}). The explicit formulas of $p(d)$ for multiplicity $2$ and $3$ are given in
(\ref{formulad2}) and (\ref{formulad3}).  In Section \ref{section3.3}, we also compare our local separation bound for simple double zeros with the one given in \cite{Dedieu2001On}.
 Although  the smallest positive real root $d\approx 0.2865$ of (\ref{formulad3})  is  smaller than  $d \approx 0.2976$ given in \cite{Dedieu2001On}, our value of $\gamma_2$ could be smaller too. Therefore, for some examples (see  Example \ref{ex1}),   our local separation bound still could be larger than the  one given in  \cite{Dedieu2001On}.

Secondly, we
define
\begin{equation}\label{H1}
H_1=\left(\begin{array}{cc}
  \frac{\partial \hat{f}(x)}{\partial X_1} & 0 \\
  \frac{\partial f_n(x)}{\partial X_1} & \frac{\partial f_n(x)}{\partial \hat{X}} \\
\end{array}
\right),
\end{equation}
and tensors
\begin{equation}\label{Hk}
H_k=\left(\left(\begin{array}{cc}
  0 & 0\\
  \Delta_k(f_n) & 0
\end{array}\right) {\mathbf{0}}_{\underbrace{n\times\cdots\times n}_k {\textstyle\times (n-1)}} \right),\ 2\leq k\leq \mux-1,
\end{equation}
and polynomials
 \[g(X)=f(X)-f(x)-\sum_{1\leq k\leq \mux-1}H_k(X-x)^k.\]
Let ${\mathcal{A}}$ be  an invertible matrix
\begin{equation}\label{operatorA}
{\mathcal{A}}=\left(\begin{array}{cc} \sqrt{2}D\hat{f}(x) & 0\\ 0 & \frac{1}{\sqrt{2}}\Amu \end{array}\right).
\end{equation}

In Theorem \ref{thm8}, we show  that if
  \begin{equation}\label{separationradius0}
  \|f(x)\|+\sum_{1\leq k\leq \mux-1}\|H_k\|\left(\frac{d}{4 {\gamma_{\mu}(g,x)}^{\mux}}\right)^k<\frac{d^{\mux+1}}
  {2\left(4{\gamma_{\mu}(g,x)}^{\mux}\right)^{\mux}}\|{\mathcal{A}}^{-1}\|,
  \end{equation}
  then $f$ has $\mux$ zeros (counting multiplicities) in the ball of radius $\frac{d}{4{\gamma_{\mu}(g,x)}^{\mux}}$ around $x$.

Thirdly, we design modified Newton iterations and extend the $\gamma$-theorem   for simple double zeros and simple triple zeros whose Jacobian matrix  has a normalized form (\ref{normalizeform1}).
Given an approximate zero $\Ax$ of $f$ with an associated exact simple double zero  $\Exi$,   we show in Theorem \ref{thmnewtonmul2}
that  when
\[\|\Ax-\Exi\| < \frac{0.0318}{\gamma_2(f,\Exi)^2},\]
  after $k$ times of the modified Newton iteration defined in Algorithm \ref{alg:newtonmult2}, we have:
 \[
 \left \|N_f^k(\Ax)-\Exi \right \| < \left( \frac{1}{2} \right)^{2^k-1}  \left \| \Ax-\Exi \right \|.
 \]
Similarly, for  an approximate zero $\Ax$ of $f$ with an associated exact simple triple  zero  $\Exi$,   we show in Theorem \ref{thmnewtonmul3} that  when
\[\|\Ax-\Exi\| < \frac{0.0154}{\gamma_3(f,\Exi)^3},\]
  after $k$ times of the modified Newton iteration defined in Algorithm \ref{alg:newtonmult3}, we have:
 \[
 \left \|N_f^k(\Ax)-\Exi \right \| < \left( \frac{1}{2} \right)^{2^k-1}  \left \| \Ax-\Exi \right \|.
 \]

 Finally, for simple multiple zeros whose Jacobian matrix is of corank one but it  does not  have a normalized form (\ref{normalizeform1}),   we  apply the unitary transformations   defined  in Section \ref{subsectionnormalform} to obtain an equivalent polynomial system whose Jacobian matrix at   the  approximate simple multiple zeros $\Ax$ has the normalized  form approximately:
   \begin{equation}\label{approxnormal}
   Df(\Ax)=\left(
   \begin{array}{cc}
    0 & \Sigma_{n-1} \\
    \sigma_n & 0 \\
     \end{array}
     \right),
     \end{equation}
where $\sigma_n$ is its smallest singular value and  $\Sigg$ is a nonsingular  diagonal matrix.  Then we perform the Newton iteration to refine the last $n-1$ variables. After the Newton iteration, we need to perform the unitary transformations again to ensure  that the Jacobian matrix at the refined approximate solution satisfies (\ref{approxnormal}). We define the modified  Newton iterations based on our previous work in \cite[Algorithm 1]{LZ:2011} to refine the first variable.  We show in Theorem \ref{thm4.1}, for
 \[\hat{\gamma}_{\mu}(f, \Ax)\|\Ax-\Exi\|<\frac{1}{2},\]
 the refined singular solution $N_f(\Ax)$ returned by     Algorithm \ref{MultiStructure} satisfies
\[
\|N_f(\Ax)- \Exi\|=\bigO(\|\Ax-\Exi\|^2).
\]
Furthermore, we  show in Theorem \ref{newtonmainthm} that for an approximate zero $\Ax$ of a system $f$ associated to a simple triple  zero  $\Exi$, when
\[\|\Ax-\Exi\| < \frac{0.0098}{\gamma_3(f,\Exi)^3},\]
  after $k$ times of the modified Newton iterations defined in    Algorithm \ref{MultiStructure}, we have:
 \[
 \left \|N_f^k(\app)-\acc \right \| < \left( \frac{1}{2} \right)^{2^k-1}  \left \| \app -\acc  \right \|.
 \]
It is clear that the proof of Theorem \ref{newtonmainthm} can be generalized to give a quantitative quadratic convergence of  Algorithm \ref{MultiStructure} for simple multiple zeros of arbitrary higher multiplicities.

\vskip 10pt
\paragraph*{\bf Related Works}

 Yakoubsohn~\cite{Yakoubsohn2000Finding} extended $\alpha$-theory to  clusters  of univariate polynomials and provided an algorithm to compute clusters of univariate polynomials \cite{YAKOUBSOHN2002}. Giusti, Lecerf, Salvy and Yakoubsohn~\cite{GLSY:2005} presented point estimate criteria for cluster location of analytic functions in the univariate case. They  provided  bounds on the diameter of the cluster which contains $\mu$ zeros (counting multiplicities) of $f$. They proposed  an algorithm based on the Schr\"oder iterations for  approximating clusters  and provided a stopping criterion which guarantees  the algorithm converges to the cluster quadratically.    In \cite{GLSY:2007}, they further  generalized   their   results  to locate and approximate   clusters of zeros of analytic maps  of embedding dimension one  in the multivariate case. They   reduced this  particular multivariate case to univariate case via implicit theorem and deflation techniques. We  are inspired by   their technique of reduction to one variable  but we try to avoid the  use of implicitly known univariate analytic function.  Dedieu and Shub~\cite{Dedieu2001On} gave  explicitly a  lower bound for separating simple double zeros $\xxi$ from other zeros of $f$ (\ref{lowerbounddouble}) and  a criterion (\ref{rouche2}) depending only on the approximate solution which guarantees the existence of a cluster of two zeros.  Based on our previous work  \cite{LiZhi:2009,LZ:2012} on computing multiplicity structure of simple multiple zeros, we  generalize Dedieu and Shub's results and deal with  simple multiple zeros with higher multiplicities.  The proof of the  non-quantified quadratic convergence \cite[Theorem 3.16]{LZ:2011}  of Algorithm 1 in  \cite{LZ:2011} has also been simplified.

   There are other approaches to compute isolated multiple zeros or zero clusters, e.g., corrected Newton methods~\cite{Rall66,DeckerKelley:1980I,DeckerKelley:1980II,DeckerKelley:1982,
Griewank:1980,GriewankOsborne:1981,Reddien:1978,Reddien:1980,Morgan1992Computing},  deflation techniques~\cite{OWM:1983,YAMAMOTONORIO:1984,Ojika:1987,Lecerf:2002,DZ:2005,LVZ06,LVZ:2008,WuZhi:2009,DLZ:2009,RuGr09,LZ:2011,MM:2011,
GiustiYakoubsohn:2013,Li2014Verified,Hauenstein2013,Hauenstein:2015}. We refer to \cite{GLSY:2007, Hauenstein:2015} for  excellent introductions of previous works  on approximating  multiple zeros.

\vskip 10pt

\paragraph*{\bf Structure of the Paper}
In Section \ref{sec2}, we recall some notations   and show how to compute incrementally a closed basis of the local dual space of $I_f$ at a given multiple root $\xxi$ of corank $1$ and multiplicity $\mu$.  In Section \ref{mul3}, we  begin with explaining  how to extend main results in \cite{Dedieu2001On} to  simple triple zeros. We  present  a lower bound for separating simple triple zeros  from  other zeros of $f$ and  an explicit  criterion  that guarantees the existence of a cluster of three  zeros of $f$ around  the approximate singular solution $\xx$. Then we generalize these results to simple multiple zeros with arbitrary higher multiplicities. We also compare our local separation bound for simple double zeros with the one given in \cite{Dedieu2001On}.  In Section \ref{sec4}, we design modified Newton iterations and extend the $\gamma$-theory  for simple double zeros and simple triple zeros whose  Jacobian matrix  has a normalized form.
 For a  simple multiple zero of arbitrary large  multiplicity whose Jacobian matrix does not have a normalized form, we perform unitary transformations and  modified Newton iterations,  and show  non-quantified quadratic convergence of new algorithm. Furthermore, we show its  quantified convergence  for simple triple zeros.

\section{Preliminaries}\label{sec2}

\subsection{Local Dual Space}

Let
$\mathbf{d}^{\boldsymbol{\alpha}}_{\xx}: \CC[X] \rightarrow \mathbb{C}$
denote the differential functional defined by
\begin{equation}\label{dualdef}
\mathbf{d}^{\boldsymbol{\alpha}}_{\xx}(g)=\frac{1} {\alpha_1!\cdots
\alpha_n!}\cdot\frac{\partial^{|\boldsymbol{\alpha}|} g}{\partial
x_1^{\alpha_1}\cdots
\partial x_n^{\alpha_n}}(\xx),\quad\forall g\in \CC[X],
\end{equation}
where $\xx\in \mathbb{C}^n$ and $\boldsymbol{\alpha}=[\alpha_1,\ldots,
\alpha_n]\in
\mathbb{N}^n$. We have
 \begin{equation}\label{diffab}
 \mathbf{d}^{\boldsymbol{\alpha}}_{\xx}\left((X-x)^{\boldsymbol{\beta}}\right)=
 \left\{\begin{array}{ll}
 1, & \mbox{if} ~~\boldsymbol{\alpha}=\boldsymbol{\beta}, \\
                                                            0, & \mbox{otherwise.}
                                                          \end{array}
\right.
\end{equation}
Let  $I_f$ be an ideal  generated by a polynomial system $f=\{f_1,
\ldots, f_n\}$, where $f_i\in\mathbb{C}[X_1, \ldots, X_n]$.
The local dual space of $I_f$ at  a given isolated singular solution $\xx$ is a subspace $\mathcal{D}_{f,\xx}$ of
$\mathfrak{D}_{\xx}=\ran_\mathbb{C}\{\mathbf{d}^{\alpha}_{\xx}\}$
such that
\begin{equation}\label{dualspace}
\mathcal{D}_{f,\xx}=\{\Lambda\in
\mathfrak{D}_{\xx}\ \ |\ \  \Lambda(g)=0, ~\forall g\in I_f\}.
\end{equation}
When the evaluation point  $\xx$ is clear from the context,  we  write $d_1^{\alpha_1}\cdots d_n^{\alpha_n}$ instead of $\mathbf{d}^{\boldsymbol{\alpha}}_{\xx}$ for
simplicity.

Let $\mathcal{D}_{f,\xx}^{(k)}$ be the subspace of $\mathcal{D}_{f,\xx}$ with differential functionals of  orders bounded by $k$, we define
\begin{enumerate}
  \item breadth $\br=\dim\left(\mathcal{D}_{f,\xx}^{(1)}\setminus\mathcal{D}_{f,\xx}^{(0)}\right)$,
  \item depth $\dps=\min\left(\left\{k\ \ |\ \dim\left(\mathcal{D}_{f,\xx}^{(k+1)}\setminus\mathcal{D}_{f,\xx}^{(k)}
      \right)=0\right\}\right)$,\\
  \item multiplicity $\mux=\dim\left(\mathcal{D}_{f,\xx}^{(\dps)}\right)$. 
\end{enumerate}
If $\xx$ is an isolated singular solution of $f$, then $1\leq\br\leq n$ and $\dps<\mux<\infty$. 

Let us introduce a morphism $\Phi_{\sigma}: \mathfrak{D}_{\xx}\rightarrow\mathfrak{D}_{\xx}$ which is an anti-differentiation  operator defined by
\[
\Phi_{\sigma}(d_1^{\alpha_1}\cdots d_n^{\alpha_n})=\left\{\begin{array}{ll}
                                                            d_1^{\alpha_1}\cdots d_{\sigma}^{\alpha_{\sigma}-1} \cdots d_n^{\alpha_n}, & \mbox{if }\alpha_{\sigma}>0,\\
                                                            0, & \mbox{otherwise.}
                                                          \end{array}
\right.
\]

Computing a closed basis of the local dual space is done essentially
by matrix-kernel computations based on the stability property  of $\mathcal{D}_{f,\xx}$
\cite{MMM:1995,Mourrain:1996,Stetter:2004,DZ:2005}:
\begin{equation}\label{closed}
\forall \Lambda\in\mathcal{D}_{f,\xx}^{(k)}, \, \,  \Phi_{\sigma}(\Lambda)\in\mathcal{D}_{f,\xx}^{(k-1)}, \  \   \sigma=1,\ldots,n.
\end{equation}



\subsection{Simple Multiple Zeros}

In this paper, we deal  with  simple multiple zeros satisfying  $f(\xxi)=0$, $\dim \ker Df(\xxi)=1$. It  is also  called  breadth one  singular zero in \cite{DZ:2005}  as
\begin{equation}
\dim(\mathcal{D}_{f,\xx}^{(k)}\setminus\mathcal{D}_{f,\xx}^{(k-1)})=1, ~ k=1\ldots,\dps, ~\dps =\mux-1.
\end{equation}
Therefore, the local dual space of  $I_f$ at  a given isolated simple singular solution $\xx$ is
\[\mathcal{D}_{f,\xx}=\ran_\mathbb{C}\{\Lambda_0,\Lambda_1,\ldots,\Lambda_{\mu-1}\},\]
 where $\degree(\Lambda_k)=k$ and $\Lambda_0=1$.  
Suppose $\Lambda_1=a_{1,1}d_1+\cdots+a_{1,n}d_n$, without loss of generality, we assume $a_{1,1}=1$.
Let
  $\Psi_{\sigma}: \mathfrak{D}_{\xx}\rightarrow\mathfrak{D}_{\xx}$ be a differential operator defined by
\[
\Psi_{\sigma}(d_1^{\alpha_1}\cdots d_n^{\alpha_n})=\left\{\begin{array}{ll}
                                                             d_{\sigma}^{\alpha_{\sigma}+1} \cdots d_n^{\alpha_n}, & \mbox{if }\alpha_1=\cdots=\alpha_{\sigma-1}=0,\\
                                                            0, & \mbox{otherwise.}
                                                          \end{array}
\right.
\]
 For $k=2, \ldots, \mu-1$,  by the stability property, we have
\begin{equation}\label{dualbasis}
\left\{\begin{array}{l}
         \Phi_{1}(\Lambda_{k}) = a_{1,1}\Lambda_{k-1}+\cdots+a_{k-1,1}\Lambda_1+a_{k,1}\Lambda_0, \\
         \vdots \\
         \Phi_{n}(\Lambda_{k}) = a_{1,n}\Lambda_{k-1}+\cdots+a_{k-1,n}\Lambda_1+a_{k,n}\Lambda_0.
       \end{array}
\right.
\end{equation}
Let  $a_{1,1}=1$, $a_{k,1}=0$ ($k=2,\ldots,n$), the system (\ref{dualbasis})
 has  a unique solution $\Lambda_{k}=\Delta_{k}+a_{k,2}d_2+\cdots+a_{k,n}d_n$, where
\begin{equation}\label{computedelta}
\Delta_{k}=\sum_{\sigma=1}^{n}\Psi_{\sigma}(a_{1,\sigma}\Lambda_{k-1}+\cdots+a_{k-1,\sigma}\Lambda_1),
\end{equation}
and $a_{k,2},\ldots,a_{k,n}$ are determined by solving the linear system  obtained from  $\Lambda_{k}(f_i)=0, i=1, \ldots, n$:
\begin{equation}\label{solveaij}
     \left(
            \begin{array}{ccc}
              d_2(f_1) & \cdots & d_n(f_1) \\
              \vdots & \ddots & \vdots \\
              d_2(f_{n}) & \cdots & d_n(f_{n}) \\
            \end{array}
          \right)\left(
                   \begin{array}{c}
                     a_{k,2} \\
                     \vdots \\
                     a_{k,n} \\
                   \end{array}
                 \right)= -\left(
    \begin{array}{c}\Delta_{k}(f_1) \\
      \vdots \\
      \Delta_{k}(f_{n}) \\
    \end{array}
  \right).
\end{equation}
We refer to  \cite{LiZhi:2009,LZ:2011, LZ:2012} for the justification of above arguments.

  The following definition generalizes  the simple double zero in \cite{Dedieu2001On}.
  \begin{definition}
   Let
$f:\mathbb{C}^n \rightarrow \mathbb{C}^n$, where $f_i\in\mathbb{C}[X]$ and suppose $f(\xx)=0$. Then $\xx$ is a simple zero of multiplicity $\mux$ for $f$ if
\begin{enumerate}
  \item[(A)]
  $\dim\ker Df(\xxi)=1$,

  \item[(B)] $\Delta_{k}(f)\in\img\  Df(\xxi)$, for $k=2,\ldots,\mux-1$,

  \item[(C)]
  $\Delta_{\mux}(f)\notin\img\ Df(\xxi)$.
\end{enumerate}%
%
\end{definition}

In fact, for $\mux=2$, suppose $\ker\Jac=\ran_\mathbb{C}\{v\}$ with $\|v\|=1$, then $\Lambda_1(f)=\Jac \cdot v =v_1 d_1(f) +\cdots + v_n d_n(f)$  and
\begin{align*}
\Delta_2(f) &=\sum_{\sigma=1}^{n} \Psi_{\sigma} (v_{\sigma} \Lambda_1)(f)\\
&=\sum_{\sigma=1}^{n} \Psi_{\sigma} (v_{\sigma} (v_1 d_1 +\cdots + v_n d_n))(f)\\
               &=\sum_{i > j}v_i v_j d_i d_j(f)+\Sigma v_i^2 d_i^2(f)\\
               &=\frac{1}{2}D^2f(\xx)(v,v).
\end{align*}
 Hence, the condition  $\Delta_{2}(f)\notin\img\ Df(\xxi)$ is equivalent to  $D^2f(\xx)(v,v)\notin\img\ \Jac$, the condition given  for the simple double zero in \cite{Dedieu2001On}.


\subsection{Normalized Form}\label{subsectionnormalform}

   We show below  that it is always possible to perform unitary transformations to obtain an equivalent polynomial system whose Jacobian matrix at the simple multiple zero  has a normalized form. 

\begin{definition}\label{defnormalizedform}
For a polynomial function $f:\mathbb{C}^n \rightarrow \mathbb{C}^n$, where $f_i\in\mathbb{C}[X_1, \ldots, X_n]$,  $\Jac$  has a normalized form if 
\begin{eqnarray}\label{normalizeform}
       \Jac=\left(
               \begin{array}{cc}
                 0 & D\hat{f}(x) \\
                 0 & 0 \\
               \end{array}
             \right),
\end{eqnarray}
   $D\hat{f}(x)$ is the nonsingular Jacobian matrix of polynomials $\hat{f}=\{f_1, \ldots, f_{n-1}\}$ with respect to variables $X_2, \ldots, X_n$.

\end{definition}

 Let $\Jac=U\cdot  \left(
         \begin{array}{cc}
           \Sigg & 0 \\
            0 & 0 \\
         \end{array}
       \right) \cdot V^{\ast}$ be the  singular value decomposition of $\Jac$ of corank $1$, where $U=(u_1,\ldots,u_n)$ and $V=(v_1,\ldots,v_n)$ are unitary matrices, $V^{\ast}$ is the Hermitian transpose of $V$,  and
$\Sigg$ is a nonsingular  diagonal matrix.
We can always assume that $\Jac$ has a normalized form (\ref{normalizeform}). Otherwise,  let $g=U^{\ast}\cdot f(W\cdot X)$, where $W=(v_n,v_1,\ldots,v_{n-1})$ is also a unitary matrix. Suppose $\xx$ is a simple multiple zero of $f$ of multiplicity $\mux$, then $W^{\ast}\xx$ is a simple multiple zero of $g$ of multiplicity $\mux$  and the  Jacobian matrix  of $g$ at $W^*x$  has a normalized form:
\begin{align*}
  Dg(W^{\ast}\xx) & = U^{\ast}\cdot\Jac\cdot W\\
   & = U^{\ast}\cdot U\cdot\Sigma\cdot V^{\ast}\cdot W\\
   & = \left(
         \begin{array}{cc}
           \Sigg & 0 \\
            0 & 0 \\
         \end{array}
       \right)
   \cdot \left(
           \begin{array}{cc}
             0 & I_{n-1} \\
             1 & 0 \\
           \end{array}
         \right)\\
         & = \left(
               \begin{array}{cc}
                 0 & \Sigg \\
                 0 & 0 \\
               \end{array}
             \right).
\end{align*}
 Furthermore, suppose $\yy$ is another zero of $f$, then $W^{\ast}\yy$ is another zero of $g$, and the Euclidean distance between  zeros $\xx$ and $y$ does not change under the unitary  transformation:
\[\|W^{\ast}\xx-W^{\ast}\yy\|=\|W^{\ast}(\xx-\yy)\|=\|\xx-\yy\|.\]

If $\xxi$ is a simple {multiple} zero of multiplicity $\mux$ for $f$ and $Df(\xxi)$ has the normalized form (\ref{normalizeform}). Then we have
\[\img\ \Jac =\img\left(\begin{array}{c}
                                              D\hat{f}(x) \\
                                              0
                                            \end{array}\right),\]
and
\begin{align*}
  \Delta_{k}(f)\in\img\ Df(\xxi) & \Leftrightarrow  \Delta_{k}(f_n)=0.
\end{align*}
The  (B)(C) conditions can be simplified to check  only the last polynomial $f_n$:

\begin{enumerate}
  \item[(B)]
  $\Delta_{k}(f_n)=0$, for $k=2,\ldots,\mux-1$,
  \item[(C)]
  $\Delta_{\mux}(f_n)\neq0$.
\end{enumerate}
 {The linear system  (\ref{solveaij}) for getting the values of $a_{k,2},\ldots,a_{k,n}$ 
 can be simplified to:
\begin{equation}\label{solveaijnormalized}
     \left(
            \begin{array}{ccc}
              d_2(f_1) & \cdots & d_n(f_1) \\
              \vdots & \ddots & \vdots \\
              d_2(f_{n-1}) & \cdots & d_n(f_{n-1}) \\
            \end{array}
          \right)\left(
                   \begin{array}{c}
                     a_{k,2} \\
                     \vdots \\
                     a_{k,n} \\
                   \end{array}
                 \right)= -\left(
    \begin{array}{c}\Delta_{k}(f_1) \\
      \vdots \\
      \Delta_{k}(f_{n-1}) \\
    \end{array}
  \right).
\end{equation}
}

\section{Local Separation Bound and Cluster Location}\label{mul3}

 We  begin with explaining  how to extend main results in \cite{Dedieu2001On} to  simple triple zeros. Then we generalize these results to simple multiple zeros with arbitrary higher multiplicities. We also compare our local separation bound for simple double zeros with the one given in \cite{Dedieu2001On}.

\subsection{Simple Triple Zeros}\label{sec3.1}

  Let $\xx$ be a simple triple    zero of  $f$ and $Df(\xxi)$ has the normalized form (\ref{normalizeform}), i.e.
 \[ \frac{\partial f_i(x)}{\partial X_1}=0, \ \ \frac{\partial f_n(x)}{\partial X_i}=0,  \ \ 1 \leq i \leq n\]
  and 
   \begin{equation*}
    \Delta_{2}(f_n)=0, \  \Delta_{3}(f_n)\neq 0.
\end{equation*}

 Let $\Lambda_0=1$, $\Lambda_1=d_1$. By (\ref{computedelta}),  we have
\begin{equation*}\label{delta2}
\Delta_2=\sum_{\sigma=1}^{n} \Psi(a_{1,\sigma} \Lambda_1) =d_1^2,
\end{equation*}
and
\[
\Lambda_2=d_1^2+a_{2,2}d_2+ \cdots +a_{2,n}d_n,
\]
where $a_{2,2}, \ldots, a_{2,n}$ satisfy
\[
\left(
                   \begin{array}{c}
                     a_{2,2} \\
                     \vdots \\
                     a_{2,n} \\
                   \end{array}
                 \right)=-\Sig^{-1}\left(
    \begin{array}{c}
      \Delta_{2}(f_1) \\
      \vdots \\
      \Delta_{2}(f_{n-1}) \\
    \end{array}
  \right)
  =-\Sig^{-1}\left(
    \begin{array}{c}
      d_1^2(f_1) \\
      \vdots \\
      d_1^{2}(f_{n-1}) \\
    \end{array}
  \right),
\]
since $d_2(f_n)=\cdots=d_n(f_n)=0$,  $\Delta_2(f_n)=0$ and the Jacobian $D\hat{f}(x)$ of   polynomials $\hat{f}=\{f_1, \ldots, f_{n-1}\}$
with respect to variables
$\hat{X}=\{X_2, \ldots, X_n\}$
 is invertible.  Moreover, since $a_{1,1}=1$, $a_{2,1}=0$, we have
\begin{align*}
\Delta_3 & =\sum_{\sigma=1}^{n} \Psi_{\sigma}(a_{1,\sigma} \Lambda_2 + a_{2,\sigma} \Lambda_1)\\
& =\Psi_1(\Lambda_2)+\sum_{\sigma=1}^{n}\Psi_{\sigma}(a_{2,\sigma}d_1)\\
& =d_1^3+a_{2,2}d_1 d_2+\cdots +a_{2,n}d_1 d_n.\\
& =d_1^3+ \left(d_1d_2, \ldots, d_1d_n  \right) \cdot \left(-\Sig^{-1}\right) \cdot \left(
    \begin{array}{c}
      d_1^2(f_1) \\
      \vdots \\
      d_1^{2}(f_{n-1}) \\
    \end{array}
  \right).
\end{align*}
For simplicity, we use  the  following equivalent conditions
  \begin{equation}\label{mul3delta2}
  \Delta_{2}(f_n)=\frac{1}{2}\frac{\partial^2f_n(x)}{\partial X_1^2}=0,
  \end{equation}
 \begin{equation}\label{mul3delta3}
  \Delta_{3}(f_n)=\frac{1}{6}\frac{\partial^3 f_n(x)}{\partial X_1^3}-\frac{\partial^2 f_n(x)}{\partial X_1\partial \hat{X}}\cdot\Sig^{-1}\frac{1}{2}\frac{\partial^2 \hat{f}(x)}{\partial X_1^2}\neq0.
  \end{equation}

The definition of $\gamma_3$ has been given in  (\ref{gammamu}):
\[
  {\gamma}_{3}=\gamma_{3}(f, \xxi)=\max(\hat{\gamma}_3,\gamma_{3,n}),
  \]
  where 
  \[
\hat{\gamma}_3=\hat{\gamma}_3(f,\xxi)=\max\left(1,\ \sup_{k\geq 2}\left\|D\hat{f}(\xxi)^{-1}\frac{D^k \hat{f}(\xxi)}{k!}\right\|^{\frac{1}{k-1}}\right),
\]
and
\[{\gamma}_{3,n}={\gamma}_{3,n}(f,\xxi)=\max\left(1,\ \sup_{k\geq 2}\left\|\frac{1}{\Athree}\cdot\frac{D^k f_n(\xxi)}{k!}\right\|^{\frac{1}{k-1}}\right).\]


For two nonzero vectors $a,\ b\in\mathbb{C}^n$,  we define their angle by
\begin{equation}\label{angle}
d_P(a,b)=\arccos\frac{|\langle a,b\rangle|}{\|a\|\cdot\|b\|}.
\end{equation}
Let  $y$ be another vector in  $\mathbb{C}^n$ and $y\not=x$ and  define  \[w=y-x=\left(\begin{array}{c}\aaa\\\bbb_2\\\vdots\\\bbb_{n}\end{array}\right),\ \bbb=\left(\begin{array}{c}\bbb_2\\\vdots\\\bbb_{n}\end{array}\right).\]
Let  $\varphi=d_P(v,y-x)$, $v=(1,0, \ldots,0)^T$, then we have
\[|\aaa|=\|w\|\cos\varphi,\ \ \|\bbb\|=\|w\|\sin\varphi.\]
  For $k\geq 2$,  we use $D^{k} \hat f(x)$ to denote the partial derivatives of $\hat f$ of order $k$ with respect to $X_1, X_2, \ldots, X_n$. We generalize main results  in \cite{Dedieu2001On} to simple triple zeros.

\begin{lem}\label{oldlemma2}
   If $\hat{\gamma}_3(f,x)\|w\|\leq\frac{1}{2}$, then
  \[\left\|\Sig^{-1}\hat{f}(y)\right\|\geq \|w\|\sin\varphi-2\hat{\gamma}_3(f,x)\|w\|^2.\]
\end{lem}
\begin{proof}
  By Taylor's expansion of $\hat{f}(y)$ at $x$, and  $\frac{\partial {\hat f}(x)}{\partial X_1}=0$, we have
  \[\hat{f}(y)=\hat{f}(x)+\Sig\bbb+\sum_{k\geq2}\frac{D^k \hat{f}(x)(y-x)^k}{k!}.\]
  Noticing that $\hat{f}(x)=0$ and $\Sig$ is invertible,  we have
  \[\bbb=\Sig^{-1}\hat{f}(y)-\sum_{k\geq2}\Sig^{-1}\frac{D^k \hat{f}(x)(y-x)^k}{k!}.\]
  By the triangle inequality, we have
  \begin{align*}
    \|w\|\sin\varphi=\|\bbb\|
    &\leq\left\|\Sig^{-1}\hat{f}(y)\right\|+\sum_{k\geq2}\left\|\Sig^{-1}\frac{D^k \hat{f}(x)}{k!}\right\|\|y-x\|^k\\
    &\leq\left\|\Sig^{-1}\hat{f}(y)\right\|+\sum_{k\geq2}\hat{\gamma}_3(f,x)^{k-1}\|w\|^k\\
    &\leq\left\|\Sig^{-1}\hat{f}(y)\right\|+2\hat{\gamma}_3(f,x) \|w\|^2,
  \end{align*}
  where the last inequality comes from the assumption that $\hat{\gamma}_3(f,x)\|w\|\leq\frac{1}{2}$.
\end{proof}

Let
\begin{equation}\label{Aoperator}
\mathcal{A}=\left(\begin{array}{cc} \sqrt{2}\Sig & 0\\ 0 & \frac{1}{\sqrt{2}}\Athree\end{array}\right) \in \mathbb{C}^{n \times n},
\end{equation}
since $\Sig$ is invertible and $\Delta_{3}(f_n)\neq 0$, we have 
\begin{equation}\label{InverseAoperator}
\mathcal{A}^{-1}=\left(\begin{array}{cc} \frac{1}{\sqrt{2}}\Sig^{-1} & 0\\ 0 & \frac{\sqrt{2}}{\Athree} \end{array}\right).
\end{equation}

\begin{lem}\label{oldlemma3}
  If 
  $ \gamma_{3}(f,x)\|w\|\leq\frac{1}{2}$,
  then
\begin{align*}
 &\|\mathcal{A}^{-1}f(y)\|\geq\frac{\cos^3\varphi-8\gamma_3^2 \cos^2\varphi\sin\varphi-7\gamma_3^2 \cos\varphi\sin^2\varphi-2\gamma_3^2\sin^3\varphi}{1+2\cos\varphi+\sin\varphi}\|w\|^3-2\gamma_3^3\|w\|^4.
\end{align*}
\end{lem}
\begin{proof}
  By Taylor's expansion of $\hat{f}(y)$ at $x$,
     and 
    $\frac{\partial {\hat f(x)}}{\partial X_1}=0$,  we have
\begin{align*}
        \bbb=&\Sig^{-1}\left(\hat{f}(y)-\frac{1}{2}\frac{\partial^2 \hat{f}(x)}{\partial X_1^2}\aaa^2-\frac{\partial^2 \hat{f}(x)}{\partial X_1\partial \hat{X}}\aaa\bbb-\frac{1}{2}\frac{\partial^2 \hat{f}(x)}{\partial \hat{X}^2}\bbb^2-\sum_{k\geq3}\frac{D^k \hat{f}(x)(y-x)^k}{k!}\right).
  \end{align*}
   By expanding $f_n(y)$ at $x$ and $\frac{\partial f_n(x)}{\partial X_1}=\cdots=\frac{\partial f_n(x)}{\partial X_n}=\frac{\partial^2 f_n(x)}{\partial X_1^2}=0$, we have:
  \begin{align*}
  \hspace{-20pt}
    &f_n(y)=\left(\frac{\partial^2 f_n(x)}{\partial X_1\partial \hat{X}}\aaa \bbb+\frac{1}{2}\frac{\partial^2 f_n(x)}{\partial \hat{X}^2}\bbb^2\right)+\frac{1}{6}\frac{\partial^3 f_n(x)}{\partial X_1^3}\aaa^3+\frac{1}{2}\frac{\partial^3 f_n(x)}{\partial X_1^2\partial \hat{X}}\aaa^2\bbb\\
    \notag&\ \ +\frac{1}{2}\frac{\partial^3 f_n(x)}{\partial X_1\partial \hat{X}^2}\aaa\bbb^2+\frac{1}{6}\frac{\partial^3 f_n(x)}{\partial \hat{X}^3}\bbb^3+\sum_{k\geq4}\frac{D^k f_n(x)(y-x)^k}{k!}.
  \end{align*}
  Substituting one $\bbb$ in $\frac{\partial^2 f_n(x)}{\partial X_1\partial \hat{X}}\aaa \bbb$ and  $\frac{1}{2}\frac{\partial^2 f_n(x)}{\partial \hat{X}^2}\bbb^2$  by the  expansion of $\bbb$, as $\Athree \neq 0$, we have
  \begin{align*}
  \hspace{-30pt}
    &\frac{1}{\Athree}f_n(y)=\frac{1}{\Athree}\frac{\partial^2 f_n(x)}{\partial X_1\partial \hat{X}}\Sig^{-1}\hat{f}(y)\aaa+\frac{1}{\Athree}\frac{1}{2}\frac{\partial^2 f_n(x)}{\partial \hat{X}^2}\Sig^{-1}\hat{f}(y)\bbb+\aaa^3\\
   \notag &+\frac{1}{\Athree}C_{2,1}\aaa^2\bbb+\frac{1}{\Athree}C_{1,2}\aaa\bbb^2+\frac{1}{\Athree}C_{0,3}\bbb^3+\sum_{k\geq4}\frac{1}{\Athree}\frac{D^k f_n(x)(y-x)^k}{k!}\\
   \notag &+\frac{1}{\Athree}T_{1,0}\sum_{k\geq3}\Sig^{-1}\frac{D^k \hat{f}(x)(y-x)^k}{k!}\aaa+\frac{1}{\Athree}T_{0,1}\sum_{k\geq3}\Sig^{-1}\frac{D^k \hat{f}(x)(y-x)^k}{k!}\bbb.
  \end{align*}
where
  \begin{align*}
  C_{2,1}&=\frac{1}{2}\frac{\partial^3 f_n(x)}{\partial X_1^2\partial \hat{X}}-\frac{\partial^2 f_n(x)}{\partial X_1\partial \hat{X}}\cdot\Sig^{-1}\frac{\partial^2 \hat{f}(x)}{\partial X_1\partial \hat{X}}-\frac{1}{2}\frac{\partial^2 f_n(x)}{\partial \hat{X}^2}\cdot\Sig^{-1}\frac{1}{2}\frac{\partial^2 \hat{f}(x)}{\partial X_1^2},\\
  C_{1,2}&=\frac{1}{2}\frac{\partial^3 f_n(x)}{\partial X_1\partial \hat{X}^2}-\frac{\partial^2 f_n(x)}{\partial X_1\partial \hat{X}}\cdot\Sig^{-1}\frac{1}{2}\frac{\partial^2 \hat{f}(x)}{\partial \hat{X}^2}-\frac{1}{2}\frac{\partial^2 f_n(x)}{\partial \hat{X}^2}\cdot\Sig^{-1}\frac{\partial^2 \hat{f}(x)}{\partial X_1\partial \hat{X}},\\
  C_{0,3}&=\frac{1}{6}\frac{\partial^3 f_n(x)}{\partial \hat{X}^3}-\frac{1}{2}\frac{\partial^2 f_n(x)}{\partial \hat{X}^2}\cdot\Sig^{-1}\frac{1}{2}\frac{\partial^2 \hat{f}(x)}{\partial \hat{X}^2},\\
  T_{1,0}&= -\frac{\partial^2 f_n(x)}{\partial X_1\partial \hat{X}},\\
  T_{0,1}&=-\frac{1}{2}\frac{\partial^2 f_n(x)}{\partial \hat{X}^2}.
  \end{align*}

For the classical operator norm, we have  the following inequalities for $i+j=k$:
  \[
   \left\|\frac{\partial^k \hat{f}(x)}{\partial X_1^i\partial \hat{X}^j} \right\| \leq \|D^k \hat{f}(x)\|, \ \
  \left\|\frac{\partial^k f_n(x)}{\partial X_1^i\partial \hat{X}^j} \right\| \leq \|D^k f_n(x)\|.
  \]
Therefore,  by moving $\aaa^3$ to the left side and $\frac{1}{\Athree}f_n(y)$ to the right side of the equation and applying the triangle inequalities, we have
 \begin{align*}
  \hspace{-20pt}
   |\aaa|^3 &\leq \left|\frac{1}{ \Athree }f_n(y)\right|+2\gamma_{3,n}\left\|\Sig^{-1}\hat{f}(y)\right\||\aaa|+\gamma_{3,n}\left\|\Sig^{-1}\hat{f}(y)\right\|\|\bbb\|\\
    &\ \ +(3\gamma_{3,n}^2+2\gamma_{3,n}\cdot 2\hat{\gamma}_3+\gamma_{3,n}\hat{\gamma}_3)|\aaa|^2\|\bbb\|\\
    &\ \ +(3\gamma_{3,n}^2+2\gamma_{3,n}\hat{\gamma}_3+\gamma_{3,n}\cdot 2\hat{\gamma}_3)|\aaa|\|\bbb\|^2+(\gamma_{3,n}^2+\gamma_{3,n}\hat{\gamma}_3)\|\bbb\|^3\\
    &\ \ +\sum_{k\geq 4}\gamma_{3,n}^{k-1}\|w\|^k+2\gamma_{3,n}\sum_{k\geq 3}\hat{\gamma}_3^{k-1}\|w\|^k|\aaa|+\gamma_{3,n}\sum_{k\geq 3}\hat{\gamma}_3^{k-1}\|w\|^k\|\bbb\|\\
    &\leq\left|\frac{1}{\Athree}f_n(y)\right|+\left\|\Sig^{-1}\hat{f}(y)\right\|(2\gamma_{3,n}|\aaa|+\gamma_{3,n}\|\bbb\|)\\
    &\ \ +(3\gamma_{3,n}^2+5\gamma_{3,n}\hat{\gamma}_3)|\aaa|^2\|\bbb\|+(3\gamma_{3,n}^2+4\gamma_{3,n}\hat{\gamma}_3)|\aaa|\|\bbb\|^2\\
    &\ \ +(\gamma_{3,n}^2+\gamma_{3,n}\hat{\gamma}_3)\|\bbb\|^3+2\gamma_{3,n}^3\|w\|^4+4\gamma_{3,n}\hat{\gamma}_3^2\|w\|^3|\aaa|+2\gamma_{3,n}\hat{\gamma}_3^2\|w\|^3\|\bbb\|\\
    &\leq \left|\frac{1}{\Athree}f_n(y)\right|+\left\|\Sig^{-1}\hat{f}(y)\right\|(2\gamma_{3,n}|\aaa|+\gamma_{3,n}\|\bbb\|)+8\gamma_3^2|\aaa|^2\|\bbb\|\\
    &\ \ +7\gamma_3^2|\aaa|\|\bbb\|^2+2\gamma_3^2\|\bbb\|^3+2\gamma_3^3\|w\|^4+4\gamma_3^3\|w\|^3|\aaa|+2\gamma_3^3\|w\|^3\|\bbb\|,
  \end{align*}
  where the second inequality follows from the assumption that $\hat{\gamma}_3(f,x)\|w\|\leq\frac{1}{2}$ and $\gamma_{3,n}(f,x)\|w\|\leq\frac{1}{2}$, while the last inequality attributes to the fact that
   $\gamma_3=\max(\hat{\gamma}_3, \gamma_{3,n})$. 

  As $|\aaa|=\|w\|\cos\varphi,\ \|\bbb\|=\|w\|\sin\varphi$, we have
  \begin{align*}
    \|w\|^3\cos^3\varphi&\leq \left|\frac{1}{\Athree}f_n(y)\right|+\left\|\Sig^{-1}\hat{f}(y)\right\|\gamma_{3,n} \|w\|(2\cos\varphi+\sin\varphi) \\
    &\  +2\gamma_3^2\|w\|^3\sin^3\varphi+7\gamma_3^2 \|w\|^3\cos\varphi\sin^2\varphi +8\gamma_3^2 \|w\|^3\cos^2\varphi\sin\varphi \\
    &\ +2\gamma_3^3\|w\|^4(1+2\cos\varphi+\sin\varphi)
  \end{align*}
  For  $\varphi\in\left[0,\ \frac{\pi}{2}\right]$, $1\leq2\cos\varphi+\sin\varphi\leq\sqrt{5}$, we have
  \begin{align*}
  \hspace{-20pt}
    &\frac{\cos^3\varphi-8\gamma_3^2 \cos^2\varphi\sin\varphi-7\gamma_3^2 \cos\varphi\sin^2\varphi-2\gamma_3^2\sin^3\varphi}{1+2\cos\varphi+\sin\varphi}\|w\|^3-2\gamma_3^3\|w\|^4\\
    \leq&\frac{1}{1+2\cos\varphi+\sin\varphi}\left|\frac{1}{\Athree}f_n(y)\right|+\frac{\gamma_{3,n} w(2\cos\varphi+\sin\varphi)}{1+2\cos\varphi+\sin\varphi}\left\|\Sig^{-1}\hat{f}(y)\right\| \\
    \leq&\left|\frac{1}{\Athree}f_n(y)\right|+\frac{\sqrt{5}}{2+2\sqrt{5}} \left\|\Sig^{-1}\hat{f}(y)\right\|\\
    \leq&\sqrt{2}\left\|\left(\begin{array}{c}\frac{\sqrt{5}}{2+2\sqrt{5}}\Sig^{-1}\hat{f}(y)\\
    \frac{1}{\Athree}f_n(y)\end{array}\right)\right\|\\
    \leq &\left\|\left(\begin{array}{cc} \frac{1}{\sqrt{2}}\Sig^{-1} & 0\\ 0 & \frac{\sqrt{2}}{\Athree} \end{array}\right)
  \left(\begin{array}{c} \hat{f}(y) \\ f_n(y) \end{array}\right)\right\|\\
  =&\|\mathcal{A}^{-1}f(y)\|.
  \end{align*}

\end{proof}

\begin{lem}\label{oldlemma4}
  Let $d \approx 0.08507$ be the  positive root of the equation
  \begin{equation}\label{formulad3}
  (1-2d-8d^2)\sqrt{1-d^2}-9d-d^2+6d^3=0.
  \end{equation}
  Let  $\theta$ be defined  by
  \begin{equation}\label{theta}
  \sin\theta=\frac{d}{\gamma_3^2}.
  \end{equation}
  Then,  for  $\gamma_{3}(f,x)\|w\|\leq\frac{1}{2}$ and $\forall y\in\mathbb{C}^n$, either
  \[
  \theta \leq  \varphi \leq \frac{\pi}{2}
  \mbox{ and }\|\mathcal{A}^{-1}f(y)\|\geq \sqrt{2}\gamma_3 \|w\|\left(\frac{\sin\theta}{2\gamma_3}-\|w\|\right),\]
  or \[0 \leq \varphi\leq\theta \mbox{ and }\|\mathcal{A}^{-1}f(y)\|\geq 2\gamma_3^3 \|w\|^3\left(\frac{\sin\theta}{2\gamma_3}-\|w\|\right).\]
\end{lem}
\begin{proof}
  For $\theta \leq  \varphi \leq \frac{\pi}{2}$, by Lemma \ref{oldlemma2}, we have
  \begin{eqnarray*}
    \sqrt{2}\|\mathcal{A}^{-1}f(y)\|&=&\left\|\left(\begin{array}{c}\Sig^{-1}\hat{f}(y)\\ \frac{2}{\Athree}f_n(y)\end{array}\right)\right\| \geq\left\|\Sig^{-1}\hat{f}(y)\right\|\\
    &\geq& \|w\|\sin\theta-2\hat{\gamma}_3(f,x)\|w\|^2\geq2\gamma_3 \|w\|\left(\frac{\sin\theta}{2\gamma_3}-\|w\|\right).
  \end{eqnarray*}
  For $0 \leq \varphi\leq\theta$, by Lemma \ref{oldlemma3}, we have
  \begin{align*}
  &\|\mathcal{A}^{-1}f(y)\|\geq 2\gamma_3^3\|w\|^3
 \left(\frac{\cos^3\varphi-8\gamma_3^2 \cos^2\varphi\sin\varphi-7\gamma_3^2 \cos\varphi\sin^2\varphi-2\gamma_3^2\sin^3\varphi}{2\gamma_3^3(1+2\cos\varphi+\sin\varphi)}-\|w\|\right).
  \end{align*}
  Let
  \[
  h(\varphi)=\frac{\cos^3\varphi-8\gamma_3^2 \cos^2\varphi\sin\varphi-7\gamma_3^2 \cos\varphi\sin^2\varphi-2\gamma_3^2\sin^3\varphi}{2\gamma_3^3(1+2\cos\varphi+\sin\varphi)}.
  \]
  We claim that
  \[
  h(\theta) \geq \frac{\sin\theta}{2 \gamma_3}.
  \]
  To prove this claim, as $\sin\theta=\frac{d}{\gamma_3^2}$, it is sufficient to show that
  \begin{equation*}
  \left(1-\frac{d^2}{\gamma_3^4}\right)^{\frac{3}{2}}-8d\left(1-\frac{d^2}{\gamma_3^4}\right)-\frac{7d^2}{\gamma_3^2}\sqrt{1-\frac{d^2}{\gamma_3^4}}-\frac{2d^3}{\gamma_3^4}-d-2d\sqrt{1-\frac{d^2}{\gamma_3^4}}-\frac{d^2}{\gamma_3^2}\geq 0.
  \end{equation*}
Since this function  for $\gamma_3\geq 1$, is increasing for any $d\in\left[0,\ \frac{1}{6}\right]$, similar to the proof of \cite[Lemma 4]{Dedieu2001On},  it is sufficient to check this inequality for $\gamma_3=1$,
  \[
  (1-2d-8d^2)\sqrt{1-d^2}-9d-d^2+6d^3\geq 0.
  \]
The smallest positive root of  the equation (\ref{formulad3}) obtained by setting  the above inequality to $0$ is
 \[d \approx 0.08507,\]
  which lies in  $ \left[0,\ \frac{1}{6}\right]$.  The claim  $h(\theta) \geq \frac{\sin\theta}{2 \gamma_3}$ follows.

Furthermore, the polynomial
  $h(\varphi)$ is non-negative and decreasing  for
  \begin{equation}\label{condiphi}
  \varphi\in[0,\ \theta],  \ \  \theta\in \left[0,\ \arcsin\frac{2}{\sqrt{5}}\right],
  \end{equation}
 as its numerator %
is decreasing  and its numerator is increasing, and both are nonnegative  for  $\varphi$ satisfying (\ref{condiphi}).
 Hence, we have
 \[h(\varphi) \geq h(\theta) \geq  \frac{\sin\theta}{2 \gamma_3},\]
 for $\varphi$ satisfying (\ref{condiphi}).
  Together with
Lemma \ref{oldlemma3}, we have
\[
\|\mathcal{A}^{-1}f(y)\|  \geq 2\gamma_3^3 \|w\|^3\left(h(\varphi)-\|w\|\right)  \geq 2\gamma_3^3 \|w\|^3\left(\frac{\sin\theta}{2\gamma_3}-\|w\|\right).
\]
\end{proof}

     Let $d \approx 0.08507$ be the smallest  positive root of the equation (\ref{formulad3}). The following four theorems generalize the results in \cite{Dedieu2001On} to simple triple zeros.

\begin{thm}\label{thm1}
   Let $x$ be an isolated simple triple  zero of the polynomial system $f$, and $y$ is another zero of $f$,  then
  \begin{equation}
  \|y-x\|\geq\frac{d}{2\gamma_3^3}.
  \end{equation}
\end{thm}
\begin{proof}
 Since $f(y)=0$, when $\gamma_3 \|w\|\leq\frac{1}{2}$,  by Lemma \ref{oldlemma4} and (\ref{theta}), we have
  \[\|y-x\|=\|w\|\geq \frac{\sin\theta}{2\gamma_3}=\frac{d}{2\gamma_3^3}.\]
  For $\gamma_3 \|w\|\geq\frac{1}{2}$,  the same conclusion holds as $\gamma_3\geq 1$ and $d<1$.
\end{proof}

\begin{thm}\label{thm2}
  Let $x$ be an isolated simple triple  zero of the polynomial system $f$,  and $\|y-x\|\leq\frac{d}{4\gamma_3^3}$, then
  \[\|f(y)\|\geq \frac{d\|y-x\|^3}{2\|\mathcal{A}^{-1}\|}.\]
\end{thm}
\begin{proof}
  When \[\|w\|=\|y-x\|\leq\frac{d}{4\gamma_3^3}=\frac{\sin\theta}{4\gamma_3},\]
  by Lemma \ref{oldlemma4}, 
   we have
  \[\|\mathcal{A}^{-1}f(y)\|\geq   2\gamma_3^3 \|w\|^3\left(\frac{\sin\theta}{2\gamma_3}-\|w\|\right)
  \geq 2\gamma_3^3 \|w\|^3\frac{\sin\theta}{4\gamma_3}=2\gamma_3^3\|w\|^3\frac{d}{4\gamma_3^3}=\frac{d}{2}\|w\|^3.\]
\end{proof}

For $R>0$, let us  define
\begin{equation}\label{dR}
d_R(f,g)=\max_{\|y-x\|\leq R}\|f(y)-g(y)\|.
\end{equation}

\begin{thm}\label{thm3}
   Let $x$ be an isolated simple triple  zero of the polynomial system $f$ and
  \[0<R\leq\frac{d}{4\gamma_3^3}.\]
  If \[d_R(f,g)<\frac{dR^3}{2\|\mathcal{A}^{-1}\|},\]
  then the sum of the multiplicities of the zeros of $g$ in $B(x,R)$ is three.
\end{thm}
\begin{proof}
  By Theorem \ref{thm2}, for any $y$ such that $\|y-x\|=R$, we have
  \[\|f(y)-g(y)\|\leq d_R(f,g)<\frac{dR^3}{2\|\mathcal{A}^{-1}\|}=\frac{d\|y-x\|^3}{2\|\mathcal{A}^{-1}\|}\leq\|f(y)\|,\]
  by Rouch\'e's   Theorem, $f$ and $g$ have the same number of zeros inside $B(x,R)$. By Theorem \ref{thm1}, when $R\leq\frac{d}{4\gamma_3^3}$, the only zero of $f$ in $B(x,R)$ is $x$. Therefore, $g$ has three zeros in $B(x,R)$.
\end{proof}

Given $f:\mathbb{C}^n\rightarrow\mathbb{C}^n,\ x\in\mathbb{C}^n$, such that $D\hat{f}(x)$ is invertible, and
\[\Athree=\frac{1}{6}\frac{\partial^3 f_n(x)}{\partial X_1^3}-\frac{\partial^2 f_n(x)}{\partial X_1 \partial \hat{X}}D\hat{f}(x)^{-1}\cdot\frac{1}{2}\frac{\partial^2\hat{f}(x)}{\partial X_1^2}\not=0.\]
We define tensors \[H_1=\left(\begin{array}{cc}
  \frac{\partial \hat{f}(x)}{\partial X_1} & 0 \\
  \frac{\partial f_n(x)}{\partial X_1} & \frac{\partial f_n(x)}{\partial \hat{X}} \\
\end{array}
\right),\]
\[H_2=\left(\left(\begin{array}{cc}
  0 & 0\\
  \frac{1}{2}\frac{\partial^2 f_n(x)}{\partial X_1^2} & 0
\end{array}\right) \mathbf{0}_{n\times n \times (n-1)} \right),\]
and polynomials
 \[g(X)=f(X)-f(x)-H_1(X-x)-H_2(X-x)^2.\]

\begin{thm}\label{thm4}
  Let $\gamma_3=\gamma_3(g,x)$, if \[\|f(x)\|+\|H_1\|\frac{d}{4\gamma_3^3}+\|H_2\|\frac{d^2}{16\gamma_3^6}<\frac{d^4}{128\gamma_3^9\|{\mathcal{A}}^{-1}\|},\]
   then $f$ has three zeros (counting multiplicities) in the ball of radius $\frac{d}{4{\gamma_3}^3}$ around $x$.
\end{thm}
\begin{proof}
  We have $g(x)=0$,
  \[\ Dg(x)=Df(x)-H_1=\left(\begin{array}{cc}
  0 & D\hat{f}(x) \\
  0 & 0 \\
\end{array}
\right),\]
  Moreover, we have
  \[\Delta_2(g_n)=\frac{1}{2}\frac{\partial^2 g_n(x)}{\partial X_1^2}=\frac{1}{2}\frac{\partial^2 f_n(x)}{\partial X_1^2}-\frac{1}{2}\frac{\partial^2 f_n(x)}{\partial X_1^2}=0,\]
  \begin{align*}
  \Delta_3(g_n)&=\frac{1}{6}\frac{\partial^3 g_n(x)}{\partial X_1^3}-\frac{\partial^2 g_n(x)}{\partial X_1 \partial \hat{X}}\cdot D\hat{g}(x)^{-1}\cdot\frac{1}{2}\frac{\partial^2\hat{g}(x)}{\partial X_1^2}\\
  &=\frac{1}{6}\frac{\partial^3 f_n(x)}{\partial X_1^3}-\frac{\partial^2 f_n(x)}{\partial X_1 \partial \hat{X}}\cdot D\hat{f}(x)^{-1}\cdot\frac{1}{2}\frac{\partial^2\hat{f}(x)}{\partial X_1^2}\not=0.
  \end{align*}
   Hence $Dg(x)$ satisfies the normalized form, and $x$ is a simple singular root of $g$ with multiplicity three.
  Let $R=\frac{d}{4{\gamma_3}^3}$, we have
  \begin{align*}
     d_R(g,f)&=\max_{\|y-x\|\leq R}\|g(y)-f(y)\|\\
     &=\max_{\|y-x\|\leq R}\|f(x)+H_1(y-x)+H_2(y-x)^2\|\\
     &\leq \|f(x)\|+\|H_1\|R+\|H_2\|R^2\\
     &=\|f(x)\|+\|H_1\|\frac{d}{4\gamma_3^3}+\|H_2\|\frac{d^2}{16\gamma_3^6}.
  \end{align*}
  If
  \[\|f(x)\|+\|H_1\|\frac{d}{4\gamma_3^3}+\|H_2\|\frac{d^2}{16\gamma_3^6}<\frac{d^4}{128\gamma_3^9\|{\mathcal{A}}^{-1}\|},\]
  then
  \[d_R(g,f)<\frac{d^4}{128\gamma_3^9\|{\mathcal{A}}^{-1}\|}=\frac{dR^3}{2\|{\mathcal{A}}^{-1}\|}.\]
  By Theorem \ref{thm3}, the sum of the multiplicities of the zeros of $f$ in $B(x,R)$ is three.
\end{proof}

\begin{remark}
The equality of  $\gamma_{\mu}(g,x)=\gamma_{\mu} (f,x)$ is true for $\mu=2$ \cite[Theorem 4]{Dedieu2001On}.  In   Example \ref{ex2}, we show that $\left\|\frac{1}{\Delta_3(f_2)}\cdot\frac{D^2 f_2(x)}{2}\right\| \neq \left\|\frac{1}{\Delta_3(g_2)}\cdot\frac{D^2 g_2(x)}{2}\right\|$. Hence,   $\gamma_{3,n}(g,x)$ might be not equal to $\gamma_{3,n}(f,x)$ if they are not  equal to $1$.
\end{remark}

\subsection{Simple Multiple Zeros}\label{sec3.2}
 We  generalize  results in Section \ref{sec3.1} to the simple multiple zeros of higher multiplicities.

 Let
$f:\mathbb{C}^n \rightarrow \mathbb{C}^n$,  and $\xx$ be a simple zero of $f$ of multiplicity $\mux$, where $\Jac$ has the normalized form
 $\Jac=\left(
               \begin{array}{cc}
                 0 & D\hat{f}(x) \\
                 0 & 0 \\
               \end{array}
             \right)$, $D\hat{f}(x)$ is invertible and
\begin{equation}\label{simplem}
\Delta_{k}(f_n)=0, ~{\text{for}}~ k=2,\ldots,\mux-1, \  \  \Delta_{\mux}(f_n)\neq 0.
\end{equation}

{Let $y$ be another vector in $\mathbb{C}^n$ and $y\not=x$.}
Recall that $\varphi=d_P(v,y-x)$, $v=(1,0, \ldots,0)^T$ and $w=x-y=(\aaa, \bbb_2, \ldots, \bbb_n)^T$, $\bbb=(\bbb_2, \ldots, \bbb_n)^T$,  then we have $|\aaa|=\|w\|\sin\varphi,\  \|\bbb\|=\|w\|\cos\varphi$. 
Let \[{\mathcal{A}}=\left(\begin{array}{cc} \sqrt{2}D\hat{f}(x) & 0\\ 0 & \frac{1}{\sqrt{2}}\Amu \end{array}\right),\]
and ${\gamma}_{\mu}=\max(\hat{\gamma}_{\mu},\ \gamma_{\mu,n})$, where
  \[\gamma_{\mu,n}=\gamma_{\mu,n}(f,x)=\max\left(1,\ \sup_{k\geq 2}\left\|\frac{1}{\Amu}\cdot\frac{D^k f_n(x)}{k!}\right\|^{\frac{1}{k-1}}\right).\]

{\large\bf Case 1}: For $\theta\leq\varphi\leq\frac{\pi}{2}$, assume that $\gamma_{\mu}\|w\|\leq\frac{1}{2}$. The Taylor's expansion of  $\hat{f}(y)$ at $x$ is:
\[\hat{f}(y)=\Sig\bbb+\frac{1}{2}\frac{\partial^2 \hat{f}(x)}{\partial X_1^2}\aaa^2+\frac{\partial^2 \hat{f}(x)}{\partial X_1\partial \hat{X}}\aaa\bbb+\frac{1}{2}\frac{\partial^2 \hat{f}(x)}{\partial \hat{X}^2}\bbb^2+\sum_{k\geq3}\frac{D^k \hat{f}(x)(y-x)^k}{k!}.\]
%
By the triangle inequality, we have
  \begin{align*}
    \|w\|\sin\varphi=\|\bbb\|
    &\leq\left\|\Sig^{-1}\hat{f}(y)\right\|+\sum_{k\geq2}\hat{\gamma}_{\mu}(f,x)^{k-1}\|w\|^k\\
    &\leq\left\|\Sig^{-1}\hat{f}(y)\right\|+2\hat{\gamma}_{\mu}(f,x) \|w\|^2.\\
  \end{align*}
%
%
Therefore, we have the following claim.
\begin{claim}\label{claim0}
For $\theta\leq\varphi\leq\frac{\pi}{2}$, assume that $\gamma_{\mu}\|w\|\leq\frac{1}{2}$, we have
 \[\left\|\Sig^{-1}\hat{f}(y)\right\|\geq {2}\gamma_{\mu} \|w\|\left(\frac{\sin\theta}{2\gamma_{\mu}}-\|w\|\right),\]
 and
\[\|w\|\geq\frac{\sin\varphi}{2\hat{\gamma}_{\mu}}\geq\frac{\sin\theta}{2\hat{\gamma}_{\mu}}
\geq\frac{\sin\theta}{2\gamma_{\mu}}.\]
\end{claim}

{\large\bf Case 2}: For $0\leq\varphi<\theta \leq \frac{\pi}{2}$, assume that $\gamma_{\mu}\|w\|\leq\frac{1}{2}$. The Taylor expansion of $f_n$ at $y$ is:
\begin{align*}
f_n(y)&=\frac{1}{2}\frac{\partial^2f_n(x)}{\partial X_1^2}\aaa^2+\frac{\partial^2 f_n(x)}{\partial X_1\partial \hat{X}}\aaa\bbb+\frac{1}{2}\frac{\partial^2 f_n(x)}{\partial \hat{X}^2}\bbb^2+\cdots+\frac{1}{\mux!}\frac{\partial^{\mux}f_n(x)}{\partial X_1^{\mux}}\aaa^{\mux}\\
&+\frac{1}{(\mux-1)!}\frac{\partial^{\mux}f_n(x)}{\partial X_1^{\mux-1}\partial\hat{X}}\aaa^{\mux-1}\bbb+\cdots+\frac{1}{\mux!}\frac{\partial^{\mux}f_n(x)}{\partial \hat{X}^{\mux}}\bbb^{\mux}+\sum_{k\geq\mux+1}\frac{D^k f_n(x)(y-x)^k}{k!}.
\end{align*}
The coefficient of the term $\aaa^i\bbb^j$  in the Taylor expansion of $f_n$ is
$\frac{1}{i!j!} \frac{\partial^{i+j} f_n(x)}{\partial X_1^i\partial \hat{X}^j}$, whose norm divided by $\Delta_{\mux}(f_n)$
  is  bounded by
  \begin{equation}\label{coeff1}
\frac{(i+j)!}{i!j!}\gamma_{\mu}^{i+j-1}.
\end{equation}
For the monomial $\aaa^i\bbb^j,\ i+j<\mux \mbox{ and }j>0$, after substituting the first $\eta$ in $\aaa^i\bbb^j$ by
\begin{align*}
  \bbb=&-\Sig^{-1}\left(\frac{1}{2}\frac{\partial^2 \hat{f}(x)}{\partial X_1^2}\aaa^2+\frac{\partial^2 \hat{f}(x)}{\partial X_1\partial \hat{X}}\aaa\bbb+\frac{1}{2}\frac{\partial^2 \hat{f}(x)}{\partial \hat{X}^2}\bbb^2+\cdots\right.\\
 \notag& +\sum_{0\leq k\leq \mux+1-i-j}\frac{1}{(\mux+1-i-j-k)!k!}\frac{\partial^{\mux+1-i-j}\hat{f}(x)}{\partial X_1^{\mux+1-i-j-k}\partial \hat{X}^k}\aaa^{\mux+1-i-j-k}\bbb^k\\
 \notag& \left.+\sum_{k\geq\mux+2-i-j}\frac{D^k \hat{f}(x)(y-x)^k}{k!} {{-\hat{f}(y)}}\right),
\end{align*}
solved from the  Taylor's expansion formula for $\hat{f}(y)$ at $x$,  we have
\begin{align*}
  \aaa^i\bbb^j=&-\Sig^{-1}\left(\frac{1}{2}\frac{\partial^2 \hat{f}(x)}{\partial X_1^2}\aaa^{i+2}\bbb^{j-1}+\frac{\partial^2 \hat{f}(x)}{\partial X_1\partial \hat{X}}\aaa^{i+1}\bbb^{j}+\frac{1}{2}\frac{\partial^2 \hat{f}(x)}{\partial \hat{X}^2}\aaa^i\bbb^{j+1}+\cdots\right.\\
   \notag& +\sum_{0\leq k\leq \mux+1-i-j}\frac{1}{(\mux+1-i-j-k)!k!}\frac{\partial^{\mux+1-i-j}\hat{f}(x)}{\partial X_1^{\mux+1-i-j-k}\partial \hat{X}^k}\aaa^{\mux+1-j-k}\bbb^{k+j-1}\\
   \notag& \left.+\sum_{k+i+j-1\geq\mux+1}\frac{D^k \hat{f}(x)(y-x)^k}{k!}\aaa^i\bbb^{j-1} {{-\hat{f}(y)\aaa^i\bbb^{j-1}}}\right),
\end{align*}
 the total degree of each term in  the above expression  is at least $i+j+1$.
Moreover, the norm of the coefficient  of  the new  term $\aaa^{i+k}\bbb^{j-1+l}$, $i+k+j-1+l\leq \mux$ obtained after the substitution and  divided by $\Delta_{\mux}(f_n)$ is bounded by
\begin{equation}\label{coeff2}
\left(\frac{(i+j)!}{i!j!}\gamma_{\mu}^{i+j-1}\right) \left\|\Sig^{-1}
\frac{1}{k!l!}\frac{\partial^{k+l}\hat{f}(x)}{\partial X_1^k\partial \hat{X}^l}\right\|
\leq \frac{(i+j)!}{i!j!}\frac{(k+l)!}{k!l!}  \gamma_{\mu}^{i+k+j+l-2}.
\end{equation}

Starting from $i+j=2, j \geq 1$,  after at most $\mu-2$ substitutions,    we can write $f_n$ in the following form:
\begin{align}\label{keyexpansion}
f_n(y)&=C_{2}\aaa^2+\cdots+C_{\mux}\aaa^{\mux}+\sum_{i+j=\mux,j>0}C_{i,j}\aaa^i \bbb^j+\sum_{k\geq\mux+1}\frac{D^k f_n(x)(y-x)^k}{k!}\\
\notag&\ \ +\sum_{1\leq i+j-1\leq\mux-2}T_{i,j-1}\cdot \left( \sum_{k+i+j-1\geq \mux+1}\Sig^{-1}\frac{D^k\hat{f}(x)(y-x)^k}{k!}\aaa^i \bbb^{j-1} \right)\\
\notag&\ \ {{-\sum_{1\leq i+j-1\leq\mux-2}T_{i,j-1}\Sig^{-1}\hat{f}(y)\aaa^i\bbb^{j-1}}},
\end{align}
where $C_2,\ldots,C_{\mux}$ are constants, and  the coefficients $C_{i,j}$ and $T_{i, j-1}$ divided by $\Delta_{\mux}(f_n)$ are bounded:
\begin{equation}\label{cijtijbound}
\left\|\frac{1}{\Delta_{\mux}(f_n)}C_{i,j}\right\|\leq c_{i,j}\gamma_{\mu}^{i+j-1},\ \left\|\frac{1}{\Delta_{\mux}(f_n)}T_{i,j-1}\right\|\leq t_{i,j-1}\gamma_{\mu}^{i+j-1},
\end{equation}
where $c_{i,j},t_{i,j-1}\in\mathbb{R}$ are constants. This can be deduced  by using (\ref{coeff1}) and (\ref{coeff2}).
%
%

\begin{claim}\label{claimmulzhi}
We have $C_t=\Delta_{t}(f_n)$, for $t=2,\ldots,\mux$.
\end{claim}

 For simplicity, we
 replace $j-1$ by $j$ in the last two terms of (\ref{keyexpansion}).
\begin{proof}
 Let us apply the differential functional $\Delta_t$ to both sides of  (\ref{keyexpansion}):
\begin{align*}
\Delta_{t}(f_n)=&C_{2}\Delta_{t}(\aaa^2)+\cdots+C_{\mux}\Delta_{t}(\aaa^{\mux})+\sum_{i+j=\mux,j>0}C_{i,j}\Delta_{t}(\aaa^i \bbb^j)\\
&\ \ +\sum_{1\leq i+j\leq\mux-2}T_{i,j}\cdot \left( \sum_{k\geq \mux+1-i-j}\Sig^{-1}\Delta_{t}\left(\frac{D^k\hat{f}(x)(y-x)^k}{k!}\aaa^i \bbb^j\right)\right)\\
&\ \ {{-\sum_{1\leq i+j\leq\mux-2}T_{i,j}\Sig^{-1}\Delta_{t}\left(\hat{f}(y)\aaa^i\bbb^j\right)}}+\sum_{k\geq\mux+1}\Delta_{t}\left(\frac{D^k f_n(x)(y-x)^k}{k!}\right).
\end{align*}
Based on (\ref{diffab}), {(\ref{closed})} and the {fact} that    $d_1^t$ is the only differential monomial of the highest order $t$  in  $\Delta_{t}$ and no other $d_1^s$ with $s<t$  in  $\Delta_{t}$,
  we derive for $2 \leq t \leq \mux$ that:
\begin{enumerate}
  \item $\Delta_{t}(\aaa^{s})=1$ if $s=t$ and $0$ otherwise;
  \item $\Delta_{t}(\aaa^i \bbb^j)=0$  for   $t \leq i+j=\mu $ and $j>0$;
  \item $\Delta_{t}\left(\frac{D^k\hat{f}(x)(y-x)^k}{k!}\aaa^i \bbb^j\right)=0$  for  $ t \leq {\mux < i+j+k}$;
  \item $\Delta_{t}\left(\hat{f}(y)\aaa^i\bbb^j\right)=0$  for $1 \leq i+j\leq \mu-2$;
  \item $\Delta_{t}\left(\frac{D^k f_n(x)(y-x)^k}{k!}\right)=0$ for  $ t \leq \mux <k$.
\end{enumerate}
Hence, we have $C_t=\Delta_{t}(f_n)$ for $t=2,\ldots,\mux$.
\end{proof}

By Claim \ref{claimmulzhi} and (\ref{simplem}), we have
\begin{equation}\label{deltacon}
C_t=\Delta_{t}(f_n)=0,\ \ t=2,\cdots,\mux-1, \  \  C_{\mux}=\Delta_{\mux}(f_n)\not=0.
\end{equation}

From (\ref{keyexpansion}) and (\ref{deltacon}), we obtain
\begin{align*}
\aaa^{\mux}=&-\frac{1}{\Delta_{\mux}(f_n)}\sum_{i+j=\mux,j>0}C_{i,j}\aaa^i \bbb^j-\frac{1}{\Delta_{\mux}(f_n)}\sum_{k\geq\mux+1}\frac{D^k f_n(x)(y-x)^k}{k!}\\
&-\frac{1}{\Delta_{\mux}(f_n)}\sum_{1\leq i+j\leq\mux-2}T_{i,j}\cdot \left(\sum_{k\geq \mux+1-i-j}\Sig^{-1}\frac{D^k\hat{f}(x)(y-x)^k}{k!}\aaa^i \bbb^j \right)\\
&{{+\frac{1}{\Delta_{\mux}(f_n)}\sum_{1\leq i+j\leq\mux-2}T_{i,j}\Sig^{-1}\hat{f}(y)\aaa^i\bbb^j+\frac{1}{\Delta_{\mux}(f_n)}f_n(y)}}.
\end{align*}
By the triangle inequality, we have
\begin{align*}
|\aaa|^{\mux}\leq&\sum_{i+j=\mux,j>0}\left\|\frac{1}{\Delta_{\mux}(f_n)}C_{i,j}\right\||\aaa|^i \|\bbb\|^j+\sum_{k\geq\mux+1}\left\|\frac{1}{\Delta_{\mux}(f_n)}\frac{D^k f_n(x)}{k!}\right\|\|w\|^k\\
&+\sum_{1\leq i+j\leq\mux-2}\left\|\frac{1}{\Delta_{\mux}(f_n)}T_{i,j}\right\| \cdot \left( \sum_{k\geq \mux+1-i-j}\left\|\Sig^{-1}\frac{D^k\hat{f}(x)}{k!}\right\|\|w\|^k|\aaa|^i \|\bbb\|^j\right)\\
&{{+\sum_{1\leq i+j\leq\mux-2}\left\|\frac{1}{\Delta_{\mux}(f_n)}T_{i,j}\right\| \cdot\left\|\Sig^{-1}\hat{f}(y)\right\||\aaa|^i \|\bbb\|^j+\left|\frac{1}{\Delta_{\mux}(f_n)}f_n(y)\right|}}\\
\leq&\sum_{i+j=\mux,j>0}c_{i,j}\gamma_{\mu}^{i+j-1}|\aaa|^i \|\bbb\|^j+\sum_{k\geq\mux+1}\gamma_{\mu,n}^{k-1}\|w\|^k\\
&+\sum_{1\leq i+j\leq\mux-2}t_{i,j}\gamma_{\mu}^{i+j}\cdot 2\hat{\gamma}_{\mu}^{\mux-i-j}\|w\|^{\mux-i-j+1}|\aaa|^i \|\bbb\|^j\\
&{{+\sum_{1\leq i+j\leq\mux-2}t_{i,j}\gamma_{\mu}^{i+j} \cdot\left\|\Sig^{-1}\hat{f}(y)\right\||\aaa|^i \|\bbb\|^j+\left|\frac{1}{\Delta_{\mux}(f_n)}f_n(y)\right|}}\\
\leq&\sum_{i+j=\mux,j>0}c_{i,j}\gamma_{\mu}^{\mux-1}|\aaa|^i \|\bbb\|^j+\sum_{1\leq i+j\leq\mux-2}2t_{i,j}\gamma_{\mu}^{\mux}\|w\|^{\mux-i-j+1}|\aaa|^i \|\bbb\|^j+2\gamma_{\mu}^{\mux}\|w\|^{\mux+1}\\
&{{+\sum_{1\leq i+j\leq\mux-2}t_{i,j}\gamma_{\mu}^{i+j} \cdot\left\|\Sig^{-1}\hat{f}(y)\right\||\aaa|^i \|\bbb\|^j+\left|\frac{1}{\Delta_{\mux}(f_n)}f_n(y)\right|}}.\\
\end{align*}
By $|\aaa|=\|w\|\sin\varphi,\  \|\bbb\|=\|w\|\cos\varphi$, we have:
\begin{align*}
   & \|w\|^{\mux}\cos^{\mux}\varphi \leq \sum_{i+j=\mux,j>0}c_{i,j}\gamma_{\mu}^{\mux-1}\|w\|^{\mux}\cos^i\varphi \sin^j\varphi\\
	&\ \ +\sum_{1\leq i+j\leq\mux-2}2t_{i,j}\gamma_{\mu}^{\mux}\|w\|^{\mux+1}\cos^i\varphi \sin^j\varphi+2\gamma_{\mu}^{\mux}\|w\|^{\mux+1}\\
    &\ \ {{+\sum_{1\leq i+j\leq\mux-2}t_{i,j}\gamma_{\mu}^{i+j}\|w\|^{i+j}\cos^i\varphi\sin^j\varphi \cdot\left\|\Sig^{-1}\hat{f}(y)\right\|+\left|\frac{1}{\Delta_{\mux}(f_n)}f_n(y)\right|}}.\\
\end{align*}
Therefore, we have
{{
\begin{align*}
  &\frac{\cos^{\mux}\varphi-\sum_{i+j=\mux,j>0}c_{i,j}\gamma_{\mu}^{\mux-1}\cos^i\varphi \sin^j\varphi}
  {1+\sum_{1\leq i+j\leq\mux-2}t_{i,j}\cos^i\varphi \sin^j\varphi}\cdot \|w\|^{\mux}- 2\gamma_{\mu}^{\mux}\|w\|^{\mux+1}\\
  \leq&\frac{1}{1+\sum_{1\leq i+j\leq\mux-2}t_{i,j}\cos^i\varphi \sin^j\varphi}\left|\frac{1}{\Amu}f_n(y)\right|\\
  &\ \ +\frac{\sum_{1\leq i+j\leq\mux-2}t_{i,j}\gamma_{\mu}^{i+j}\|w\|^{i+j}\cos^i\varphi\sin^j\varphi}{1+\sum_{1\leq i+j\leq\mux-2}t_{i,j}\cos^i\varphi \sin^j\varphi}\left\|\Sig^{-1}\hat{f}(y)\right\|\\
  \leq &\left|\frac{1}{\Amu}f_n(y)\right|+\frac{1}{2}\left\|\Sig^{-1}\hat{f}(y)\right\|\\
  \leq &\|\mathcal{A}^{-1}f(y)\|.\\
\end{align*}
}}
 We have the following inequality:  
\begin{equation}\label{wboundf}
\|\mathcal{A}^{-1}f(y)\| \geq h(\varphi)\cdot 2\gamma_{\mu}^{\mux}\|w\|^{\mux}- 2\gamma_{\mu}^{\mux}\|w\|^{\mux+1}.
\end{equation}
where
\begin{equation}\label{hpolynomial}
  h(\varphi)=\frac{\cos^{\mux}\varphi-\sum_{i+j=\mux,j>0}c_{i,j}\gamma_{\mu}^{\mux-1}\cos^i\varphi \sin^j\varphi}
  {\sum_{1\leq i\leq\mux-2}2t_{i,0}\gamma_{\mu}^{\mux}+\sum_{1\leq i+j\leq\mux-2,j>0}2t_{i,j}\gamma_{\mu}^{\mux}\cos^i\varphi \sin^j\varphi+2\gamma_{\mu}^{\mux}}.
  \end{equation}

\begin{definition}\label{allds}
 We define $d=\min (d_1, d_2, d_3)$, where
 \[d_1=\sqrt{\frac{1}{c_{\mux-1,1}^2+1}},\  d_2=\sqrt{\frac{1}{\mux-1}},\] 
  and $d_3$ is the smallest positive real root of the polynomial
\begin{align}\label{mainunivariatepoly}
&p(d)=(1-d^2)^{\frac{\mux}{2}}-\sum_{i+j=\mux,j>0}c_{i,j}d(1-d^2)^{\frac{i}{2}}d^{j-1}\\
\notag&-d\left(\sum_{1\leq i\leq\mux-2}t_{i,0}+\sum_{1\leq i+j\leq\mux-2,j>0}t_{i,j}(1-d^2)^{\frac{i}{2}}d^j+1\right).
\end{align}
\end{definition}

In the sequel, we always assume that $d$  has the above definition.

\begin{claim}\label{generaltheta} We have
$
h(\theta) \geq \frac{\sin (\theta)}{2 \gamma_{\mu}},
$  where
$\sin \theta=\frac{d}{\gamma_{\mu}^{\mu-1}}.
$
\end{claim}

To prove this claim, substituting $\sin\varphi$ and $\cos\varphi$ in (\ref{hpolynomial}) by  $\sin\theta=\frac{d}{\gamma_{\mu}^{\mu-1}}$,
 $\cos\theta=\left(1-\frac{d^2}{\gamma_{\mu}^{2(\mux-1)}}\right)^{1/2}$,
we need to show
%
%
\begin{align}\label{functiongamma}
&\left(1-\frac{d^2}{\gamma_{\mu}^{2(\mux-1)}}\right)^{\frac{\mux}{2}}-c_{\mux-1,1}d\left(1-\frac{d^2}{\gamma_{\mu}^{2(\mux-1)}}\right)^{\frac{\mux-1}{2}}\\
\notag&-\sum_{i+j=\mux-1,j>0}c_{i,j+1}d\left(1-\frac{d^2}{\gamma_{\mu}^{2(\mux-1)}}\right)^{\frac{i}{2}}\frac{d^j}{\gamma_{\mu}^{j(\mux-1)}}\\
\notag&-d\left(\sum_{1\leq i\leq\mux-2}t_{i,0}+\sum_{1\leq i+j\leq\mux-2,j>0}t_{i,j}\left(1-\frac{d^2}{\gamma_{\mu}^{2(\mux-1)}}\right)^{\frac{i}{2}}\frac{d^j}{\gamma_{\mu}^{j(\mux-1)}}+1\right)\geq 0.
\end{align}

\begin{itemize}
\item  The sum of the  first two terms in (\ref{functiongamma}) is non-negative and  increasing  in $\gamma_{\mu}$ for  $\gamma_{\mu} \geq 1$ as it equals to
%
\[\left(1-\frac{d^2}{\gamma_{\mu}^{2(\mux-1)}}\right)^{\frac{\mux-1}{2}}\left(\sqrt{1-\frac{d^2}
{\gamma_{\mu}^{2(\mux-1)}}}-c_{\mux-1,1}d\right),
\]%
%
and  
$d \leq d_1= \sqrt{\frac{1}{c_{\mux-1,1}^2+1}}$.
\item The terms $\cos^i\varphi\sin^j\varphi,\ j>0$ are increasing with respect to $\varphi$ for $\varphi\in\left[0,\ \arctan\sqrt{\frac{1}{i}}\right]$ since
\[(\cos^i\varphi\sin\varphi)'
=\cos^{i-1}\varphi(\cos^2\varphi-i\sin^2\varphi)\geq 0.\]
 Hence, for  $1\leq i+j \leq \mux-1,\ j>0$,  $\left(1-\frac{d^2}{\gamma_{\mu}^{2(\mux-1)}}\right)^{\frac{i}{2}}\frac{d^j}{\gamma_{\mu}^{j(\mux-1)}}$,
  is decreasing in $\gamma_{\mu}$ for $d\in\left[0,\ \sqrt{\frac{1}{\mux-1}} \right]$. 
 \item The left side of (\ref{functiongamma}) is a function which  is increasing in $\gamma_{\mu}$ and it is sufficient to prove that the inequality is true when $\gamma_{\mu}=1$:
\begin{align*}
&p(d)=(1-d^2)^{\frac{\mux}{2}}-\sum_{i+j=\mux,j>0}c_{i,j}d(1-d^2)^{\frac{i}{2}}d^{j-1}\\
\notag&-d\left(\sum_{1\leq i\leq\mux-2}t_{i,0}+\sum_{1\leq i+j\leq\mux-2,j>0}t_{i,j}(1-d^2)^{\frac{i}{2}}d^j+1\right)\geq 0,
\end{align*}
which is obvious as $p(d)$ is decreasing in $d$  for $0 \leq d \leq d_3$, $p(0)>0$ and $d_3$ is the smallest real zero of $p(d)=0$.  
\end{itemize}

\begin{claim}\label{claimphidecrease} For $0 \leq \varphi \leq \theta$,
$h(\varphi)$ is non-negative and decreasing for $0 \leq \varphi \leq \theta$ and
\begin{equation}\label{phifun}
h(\varphi) \geq h(\theta) \geq \frac{ \sin \theta}{2 \gamma_{\mu}},
\end{equation}
and
\begin{equation}\label{phifunfy}
\|\mathcal{A}^{-1}f(y)\|\geq 2\gamma_{\mu}^{\mux} \|w\|^{\mux}\left(\frac{\sin\theta}{2\gamma_{\mu}}-\|w\|\right).
\end{equation}
Moreover, if $y$ is another zero of $f$, then
\begin{equation*}
\|w\|=\|y-x\|  \geq \frac{ \sin \theta}{2 \gamma_{\mu}}.
\end{equation*}
\end{claim}
 For  $\varphi\in\left[0,\ \arctan\sqrt{\frac{1}{\mu -1}}\ \right]$, since
$\cos^i\varphi \sin^j\varphi, ~i+j=\mu, ~j>0$ is increasing,  the numerator  of $h(\varphi)$  is  non-negative and decreasing,    and
 the denominator of $h(\varphi)$ is positive and increasing. Hence,  $h(\varphi)$ is non-negative and decreasing for $0 \leq \varphi \leq \theta$,
 we have (\ref{phifun}). Moreover, by (\ref{wboundf}), we have
\[ \|\mathcal{A}^{-1}f(y)\| \geq h(\varphi)\cdot 2\gamma_{\mu}^{\mux}\|w\|^{\mux}- 2\gamma_{\mu}^{\mux}\|w\|^{\mux+1} \geq  2\gamma_{\mu}^{\mux} \|w\|^{\mux}\left(\frac{\sin\theta}{2\gamma_{\mu}}-\|w\|\right).
\]

%
%

\begin{thm}\label{thm5}
   Let  $\xx$ be a simple multiple zero of $f$ of multiplicity $\mux$, and $y$ be another zero of $f$, then
  \[\|y-x\|\geq\frac{d}{2\gamma_{\mu}^{\mux}}.\]
\end{thm}
\begin{proof} 
By Claim \ref{claim0} and Claim  \ref{claimphidecrease}, we have
\[\|w\|=\|y-x\|  \geq \frac{ \sin \theta}{2 \gamma_{\mu}}=\frac{d}{2\gamma_{\mu}^{\mux}},\]
since $\sin \theta=\frac{d}{\gamma_{\mu}^{\mu-1}}$.
\end{proof}

\begin{thm}\label{thm6}  Let  $\xx$ be a simple multiple zero of $f$ of multiplicity $\mux$,
  and $\|y-x\|\leq\frac{d}{4\gamma_{\mu}^{\mux}}$, then we have
  \[\|f(y)\|\geq \frac{d\|y-x\|^{\mux}}{2\|\mathcal{A}^{-1}\|}.\]
\end{thm}
\begin{proof}
For  $ \theta \leq  \varphi \leq \frac{\pi}{2}$, by Claim \ref{claim0}, we can show that
\begin{align*}
\|\mathcal{A}^{-1}f(y)\|=&\left\|\left(\begin{array}{c}\frac{1}{\sqrt{2}}\Sig^{-1}\hat{f}(y)\\
\frac{\sqrt{2}}{\Amu}f_n(y)\end{array}\right)\right\| \geq \frac{1}{\sqrt{2}}\|\Sig^{-1}\hat{f}(y)\|\\
&\geq \sqrt{2}\gamma_{\mu} \|w\|\left(\frac{\sin\theta}{2\gamma_{\mu}}-\|w\|\right).
\end{align*}
For $0 \leq \varphi\leq\theta$, by Claim \ref{claimphidecrease}, we have
\[
\|\mathcal{A}^{-1}f(y)\|\geq 2\gamma_{\mu}^{\mux} \|w\|^{\mux}\left(\frac{\sin\theta}{2\gamma_{\mu}}-\|w\|\right).
\]
When $\|w\|=\|y-x\|\leq\frac{d}{4\gamma_{\mu}^{\mux}}=\frac{\sin\theta}{4\gamma_{\mu}}$, we have
  \[\|\mathcal{A}^{-1}f(y)\|\geq 2\gamma_{\mu}^{\mux} \|w\|^{\mux}\frac{\sin\theta}{4\gamma_{\mu}}=\frac{d\|y-x\|^{\mux}}{2}.\]
\end{proof}

\begin{thm}\label{thm7}
    Let  $\xx$ be a simple multiple zero of $f$ of multiplicity $\mux$  and
  \[0<R\leq\frac{d}{4\gamma_{\mu}^{\mux}}.\]
  If \[d_R(f,g)<\frac{dR^{\mux}}{2\|\mathcal{A}^{-1}\|},\]
  then the sum of the multiplicities of the zeros of $g$ in $B_R(x)$ is $\mux$.
\end{thm}
\begin{proof}
  By Theorem \ref{thm6}, for any $y$ such that $\|y-x\|=R < \frac{d}{4\gamma_{\mu}^{\mux}}$, we have
  \[\|f(y)-g(y)\|\leq d_R(f,g)<\frac{dR^\mu}{2\|\mathcal{A}^{-1}\|}=\frac{d\|y-x\|^\mu}{2\|\mathcal{A}^{-1}\|}\leq\|f(y)\|,\]
  by Rouch\'e's  Theorem, $f$ and $g$ have the same number of zeros inside $B_R(x)$. By Theorem \ref{thm5}, when $R\leq\frac{d}{4\gamma_{\mu}^3}$, the only root of $f$ in $B_R(x)$ is $x$. Therefore, $g$ has $\mu$ zeros in $B_R(x)$.
\end{proof}
Given $f:\mathbb{C}^n\rightarrow\mathbb{C}^n,\ x\in\mathbb{C}^n$, such that
$D\hat{f}(x)$ is invertible, and $\Delta_{\mux}(f_n)\neq 0$,
%
%
we define tensors \[H_1=\left(\begin{array}{cc}
  \frac{\partial \hat{f}(x)}{\partial X_1} & 0 \\
  \frac{\partial f_n(x)}{\partial X_1} & \frac{\partial f_n(x)}{\partial \hat{X}} \\
\end{array}
\right)\]
\[H_k=\left(\left(\begin{array}{cc}
  0 & 0\\
  \Delta_k(f_n) & 0
\end{array}\right) \mathbf{0}_{\underbrace{n\times\cdots\times n}_k {\textstyle\times (n-1)}} \right),\ 2\leq k\leq \mux-1,\]
and polynomials
 \[g(X)=f(X)-f(x)-\sum_{1\leq k\leq \mux-1}H_k(X-x)^k.\]

\begin{thm}\label{thm8}  Let $\gamma_{\mu}=\gamma_{\mu}(g,x)$, if
  \begin{equation}\label{separationradius}
  \|f(x)\|+\sum_{1\leq k\leq \mux-1}\|H_k\|\left(\frac{d}{4\gamma_{\mu}^{\mux}}\right)^k<\frac{d^{\mux+1}}
  {2\left(4\gamma_{\mu}^{\mux}\right)^{\mux}\|{\mathcal{A}}^{-1}\|},
  \end{equation}
  then $f$ has $\mux$ zeros (counting multiplicities) in the ball of radius $\frac{d}{4\gamma_{\mu}^{\mux}}$ around $x$.
\end{thm}
\begin{proof}

  We have $g(x)=0,\ Dg(x)=Df(x)-H_1=\left(\begin{array}{cc}
  0 & D\hat{f}(x) \\
  0 & 0 \\
\end{array}
\right)$. Moreover,
  \[\Delta_k(g_n)=\Delta_k(f_n)-\Delta_k(f_n)=0,\ 2\leq k\leq \mux-1, \  \Delta_{\mux}(g_n)=\Delta_{\mux}(f_n)\not=0.\]
   Therefore, $Dg(x)$ satisfies the normalized form, and $x$ is a simple multiple root of $g$ with multiplicity $\mux$. 

  Let $R=\frac{d}{4\gamma_{\mu}^{\mux}(g,x)}$, we have
  \begin{align*}
     d_R(g,f)&=\max_{\|y-x\|\leq R}\|g(y)-f(y)\|\\
     &=\max_{\|y-x\|\leq R}\|f(x)+\sum_{1\leq k\leq \mux-1}H_k(X-x)^k\|\\
     &\leq \|f(x)\|+\sum_{1\leq k\leq \mux-1}\|H_k\|R^k\\
     &=\|f(x)\|+\sum_{1\leq k\leq \mux-1}\|H_k\|\left(\frac{d}{4\gamma_{\mu}^{\mux}}\right)^k.
  \end{align*}
  If
  \[\|f(x)\|+\sum_{1\leq k\leq \mux-1}\|H_k\|\left(\frac{d}{4\gamma_{\mu}^{\mux}}\right)^k<\frac{d^{\mux+1}}{2\left(4{\gamma_{\mu}}^{\mux}\right)^{\mux}
  \|{\mathcal{A}}^{-1}\|},\]
  then
  \[d_R(g,f)<\frac{d^{\mux+1}}{2\left(4{\gamma_{\mu}}^{\mux}\right)^{\mux}
  \|{\mathcal{A}}^{-1}\|}=\frac{dR^{\mux}}
  {2\|{\mathcal{A}}^{-1}\|}.\]
  By Theorem \ref{thm7}, we know that $f$ has $\mux$ zeros (counting multiplicities) in the ball of radius $\frac{d}{4\gamma_{\mu}^{\mux}}$ around $x$.
\end{proof}

\subsection{Re-examining Double Simple Zeros}\label{section3.3}
In what follows, we assume  $x$ is a simple double zero of $f$, $Df(x)$ satisfies the normalized form, and $y$ is another zero of $f$.  By expanding $f_n(y)$ at $x$ and $\frac{\partial f_n(x)}{\partial X_1}=\cdots=\frac{\partial f_n(x)}{\partial X_n}=\frac{\partial^2 f_n(x)}{\partial X_1^2}=0$, we have:
\[f_n(y)=\frac{\partial^2 f_n(x)}{\partial X_1\partial \hat{X}}\aaa\bbb+\frac{1}{2}\frac{\partial^2 f_n(x)}{\partial \hat{X}^2}\bbb^2+\sum_{k\geq 3}\frac{D^k f_n(x)(y-x)^k}{k!}.\]
The nonzero terms are
\[C_{1,1}=\frac{\partial^2 f_n(x)}{\partial X_1\partial \hat{X}},\ C_{0,2}=\frac{1}{2}\frac{\partial^2 f_n(x)}{\partial \hat{X}^2}, \  c_{1,1}=2,\  c_{0,2}=1.\]
Hence, $p(d)$ has the following form:
\begin{equation}\label{formulad2}
p(d)=1-2d^2-2d\sqrt{1-d^2}-d.
\end{equation}
The smallest positive real root of $p(d)$  is
\[d \approx 0.2865.\]
Let  $y$ be another root of $f$, by Theorem \ref{thm5}, we  have
 \[\|y-x\|\geq\frac{d}{2\gamma_{2}^2}.\]

\begin{example}\label{ex1}
Suppose we are given polynomials:
\begin{equation*}
  \left\{ \begin{aligned}
    &f_1=X_1^2-\frac{1}{4}X_1-\frac{1}{2}X_2,\\
    &f_2=\frac{1}{2}X_1X_2.\\
  \end{aligned} \right.
\end{equation*}
It is easy to check that  $\xx=(0,0)$ is a simple double zero of $f=\{f_1,f_2\}$, and $(1/4,0)$ is another zero of $f$, and we have
\end{example}
\begin{equation*}
  Df(\xx)=\left(\begin{array}{cc}
    -\frac{1}{4} & -\frac{1}{2}\\
    0 & 0\\
  \end{array}\right).
\end{equation*}
Let $g(X)=f(W\cdot X)$, where $W=\left(\begin{array}{cc}
    \frac{2}{\sqrt{5}} & -\frac{1}{\sqrt{5}}\\
    -\frac{1}{\sqrt{5}} & -\frac{2}{\sqrt{5}}\\
  \end{array}\right)$,
then
\begin{equation}
  \left\{ \begin{aligned}
    &g_1=\frac{4}{5} X_1^2-\frac{4}{5}X_1X_2+\frac{1}{5} X_2^2+\frac{\sqrt{5}}{4}X_2,\\
    &g_2=-\frac{1}{5}X_1^2-\frac{3}{10}X_1X_2+\frac{1}{5}X_2^2,\\
  \end{aligned} \right.
\end{equation}
and $\xx=(0,0)$ 
is a simple double zero of $g$, $y=\left(\frac{1}{2\sqrt{5}},-\frac{1}{4\sqrt{5}}\right)$ 
is another zero. We have
\begin{equation*}
  Dg(\xx)=\left(\begin{array}{cc}
    0 & \frac{\sqrt{5}}{4}\\
    0 & 0\\
  \end{array}\right),
\end{equation*}
and $\ker Dg(x)= \ran\{v\}$, where $v=\left(\begin{array}{c}
  1\\0\\
\end{array}\right)$,
\begin{equation*}
  D^2g(\xx)=\left(\begin{array}{cc}
     \left(\begin{array}{cc}
    \frac{8}{5} & -\frac{4}{5}\\
    -\frac{2}{5} & -\frac{3}{10}\\
  \end{array}\right)
  & \left(\begin{array}{cc}
    -\frac{4}{5} & \frac{2}{5}\\
    -\frac{3}{10} & \frac{2}{5}\\
  \end{array}\right)\\
  \end{array}\right).
\end{equation*}
Since
\[\frac{\partial g_1(x)}{\partial X_2}=\frac{\sqrt{5}}{4},\  \  \frac{1}{2}\frac{\partial^2 g_2}{\partial X_1^2}=\Delta_2=-\frac{1}{5},\]
we have
 \[\hat{\gamma}_{2}=\max\left(1,\ \left\|\left(\frac{\partial g_1(x)}{\partial X_2}\right)^{-1}\cdot\frac{D^2 g_1(x)}{2}\right\|\right)=\frac{4}{\sqrt{5}},\]
 and
  \[\gamma_{2,2}=\max\left(1,\ \left\|\frac{1}{\Delta_2(g_2)}\cdot\frac{D^2 g_2(x)}{2}\right\|\right)=1.\]
 Therefore,  \[\gamma_2=\frac{4}{\sqrt{5}}\approx 1.7888,\]
and  the local  separation bound we obtained according to Theorem \ref{thm5} for $\mu=2$  is
\[\|y-x\|\geq\frac{d}{2\gamma_2^2}\geq 0.0447.\]

Now let us estimate  the local  separation bound by the method in \cite{Dedieu2001On}.
The invertible linear operator defined in \cite{Dedieu2001On} is
\begin{equation*}
A(f,x,v)=Df(x).+\frac{1}{2}D^2f(x)(v,\Pi_v),
\end{equation*}
where $v=\left(\begin{array}{c}
  \frac{2}{\sqrt{5}}\\ -\frac{1}{\sqrt{5}}\\
\end{array}\right)$, and
\begin{equation*}
  D^2f(\xx)=\left(\begin{array}{cc}
     \left(\begin{array}{cc}
    2 & 0\\
    0 & \frac{1}{2}\\
  \end{array}\right)
  & \left(\begin{array}{cc}
    0 & 0\\
    \frac{1}{2} & 0\\
  \end{array}\right)\\
  \end{array}\right).
\end{equation*}
Let $\omega=\left(\begin{array}{c}
  \omega_1\\\omega_2\\
\end{array}\right)\in\mathbb{C}^2$, denote $A=A(f,x,v)$, we have
\begin{align*}
 A\omega&=Df(x)\omega+\frac{1}{2}D^2f(x)(v,\Pi_v \omega)=\left(\begin{array}{cc}
    -\frac{1}{4} & -\frac{1}{2}\\
    0 & 0
  \end{array}\right)\left(\begin{array}{c} \omega_1\\\omega_2 \end{array}\right)\\
 & +\frac{1}{2}\left(\begin{array}{cc}
     \left(\begin{array}{cc}
    2 & 0\\
    0 & \frac{1}{2}
  \end{array}\right)
  & \left(\begin{array}{cc}
    0 & 0\\
    \frac{1}{2} & 0
  \end{array}\right)
  \end{array}\right)\cdot\left(\begin{array}{c}
  \frac{2}{\sqrt{5}}\\ -\frac{1}{\sqrt{5}}
\end{array}\right)\cdot\left(\begin{array}{c} \frac{4}{5}\omega_1-\frac{2}{5}\omega_2 \\ -\frac{2}{5}\omega_1+\frac{1}{5}\omega_2 \end{array}\right)\\
 &=\left(\begin{array}{cc}
    \left(-\frac{1}{4}+\frac{8}{5\sqrt{5}}\right)\omega_1+\left(-\frac{1}{2}-\frac{4}{5\sqrt{5}}\right)\omega_2 \\
    -\frac{2}{5\sqrt{5}}\omega_1+\frac{1}{5\sqrt{5}} \omega_2
  \end{array}\right).
\end{align*}
We have
\[A^{-1}\frac{D^2f(x)}{2}=\left(\begin{array}{cc}
     \left(\begin{array}{cc}
    -\frac{4}{5} & -\frac{2(25+8\sqrt{5})}{10\sqrt{5}}\\
    -\frac{8}{5} & -\frac{-25+32\sqrt{5}}{20\sqrt{5}}
  \end{array}\right)
  & \left(\begin{array}{cc}
    -\frac{2(25+8\sqrt{5})}{10\sqrt{5}} & 0\\
    -\frac{-25+32\sqrt{5}}{20\sqrt{5}} & 0
  \end{array}\right)
  \end{array}\right).\]

The computation of the  norm of the tensor $A^{-1}\frac{D^2f(x)}{2}$ is quite challenge. However, using our SOS certificates for global optima of polynomials and rational
functions \cite{KLYZ09}, we can  verify  that
\[\gamma_{2}(f,x)=\max\left(1,\ \left\|A^{-1}\frac{D^2 f(x)}{2}\right\|\right)\geq 3.1121.\]
Therefore, the local separation bound computed by the method in \cite{Dedieu2001On} satisfies
\[\frac{d}{2\gamma_2(f,x)^2}\leq 0.01546,\]
for $d \approx 0.2976$.
\begin{remark}
Although our $d$ is  smaller than the one obtained in \cite{Dedieu2001On},  the value of $\gamma_2$ computed by our method could be smaller too. Therefore, as shown by Example \ref{ex1}, we might get better local separation bound.
\end{remark}

\begin{example}\label{ex2}
Suppose we are given polynomials:
\begin{equation*}
  \left\{ \begin{aligned}
    &f_1=\frac{64}{73} X_1^2-\frac{48}{73}X_1X_2+\frac{9}{73} X_2^2+\frac{\sqrt{73}}{12}X_2,\\
    &f_2=(8X_1-3X_2)^2 (3X_1+8X_2).
    \end{aligned}\right.
\end{equation*}
Let $\xx=(0,0)$ be a simple triple  zero of $f=\{f_1, f_2\}$, and $y=(\frac{2}{\sqrt{73}}, -\frac{3}{4\sqrt{73}})$ be another zero  of $f$, $\|y-x\| =0.25$.
\end{example}

We have
\begin{equation*}
  Df(\xx)=\left(\begin{array}{cc}
    0 & \frac{\sqrt{73}}{12}\\
    0 & 0\\
  \end{array}\right),
\end{equation*}
\[\frac{\partial f_1(x)}{\partial X_2}=\frac{\sqrt{73}}{12},\]

\[\Delta_3(f_2)=\frac{1}{6}\frac{\partial^3 f_2(x)}{\partial X_1^3}-\frac{\partial^2 f_2(x)}{\partial X_1\partial X_2}\cdot\left(\frac{\partial f_1(x)}{\partial X_2}\right)^{-1}\frac{1}{2}\frac{\partial^2 f_1(x)}{\partial X_1^2}=192,
\]
 \[\hat{\gamma}_{3}(f,x)=\max\left(1,\ \left\|\left(\frac{\partial f_1}{\partial X_2}\right)^{-1}\cdot\frac{D^2 f_1(x)}{2}\right\|\right)=\frac{12}{\sqrt{73}},\]
  \[\gamma_{3,2}(f,x)=\max_{2\leq k\leq 3}\left(1,\ \left\|\frac{1}{\Delta_3(f_2)}\cdot\frac{D^k f_2(x)}{k!}\right\|^{\frac{1}{k-1}}\right)=\frac{\sqrt{73}\cdot 146^{\frac{1}{4}} }{24},\]
  \[\gamma_3(f,x)=\max(\hat \gamma_3(f,x), \gamma_{3,2}(f,x))= \frac{12}{\sqrt{73}}  \approx1.4045.\]
By Theorem \ref{thm1}, the local separation bound we obtain is
 \[\|y-x\|\geq\frac{d}{2\gamma_3^3} \approx 0.01545.\]
For an approximate solution {{$x=(-4.1291 \cdot 10^{-8}, -2.9505 \cdot 10^{-8})$}} obtained after applying twicely  the modified Newton iterations defined by  Algorithm \ref{alg:newtonmult3} to an approximate zero $x=(-0.01, 0.01)$, we have
\[ \gamma_3(g,x)=\max(\hat\gamma_3(g,x), \gamma_{3,2}(g,x))=\frac{12}{\sqrt{73}}\approx1.4045, \]
and
 \[\|f(x)\|+\|H_1\|\frac{d}{4\gamma_3^3}+\|H_2\|\frac{d^2}{16\gamma_3^6}\approx 1.937364 \cdot 10^{-8} <\frac{d^4}{128\gamma_3^9\|{\mathcal{A}}^{-1}\|}\approx 1.937370\cdot 10^{-8}.\]
  By Theorem \ref{thm4}, we can guarantee that $f$ has three zeros (counting multiplicities) in the ball of radius $\frac{d}{4{\gamma_3}^3}\approx 0.0076$ around $x$.

We  notice   that
 \[\left\|\frac{1}{\Delta_3(f_2)}\cdot\frac{D^2 f_2(x)}{2}\right\|=0.00004935  \neq \left\|\frac{1}{\Delta_3(g_2)}\cdot\frac{D^2 g_2(x)}{2}\right\|=0.00001467. \]
%

%

\section{Modified Newton Iterations}\label{sec4}

 For simple double zeros and simple triple  zeros whose Jacobian matrix  has a normalized form (\ref{normalizeform}), we define modified Newton iterations and show the quantified quadratic convergence if the approximate zeros are near the exact singular zeros.  For a  simple multiple zero of arbitrary large  multiplicity whose Jacobian matrix may not have a normalized form, we perform unitary transformations and  modified Newton iterations based on our previous work in \cite{LZ:2011},  and show  its non-quantified quadratic convergence for  simple multiple zeros and the quantified convergence for simple triple zeros.

\subsection{$\gamma$-theorem  for  Simple Double Zeros}

 Given an approximate zero $\app$ of $f$ with associated simple double zero  $\acc$ such that
  $D\hat{f}(\acc)$ is invertible and
 \[ \frac{\partial f_i (\acc)}{\partial X_1}=0, \ \ \frac{\partial f_n (\acc)}{\partial X_i}=0,  \ \ 1 \leq i \leq n, \ \  \Delc \neq 0, \]
   we aim to approximate $\acc$ by applying modified Newton's method to $\app$ and iterating $k$ times such that $\|N^k_f(\app)-\acc\| < \epsilon$ for a given accuracy $\epsilon$.

\begin{algorithm}[H]
\label{alg:newtonmult2}
\caption{ Modified Newton Iteration for   Simple Double Zero}
\begin{algorithmic}[1]
\REQUIRE~~\\
$f$: a polynomial system; \\
$\app =(\azone, \azhat)$:  an approximate simple double zero of $f$; \\
\ENSURE~~\\
$N_f(\app)=(N_2(\azone),N_1(\azhat) )$:  a refined solution after one  iteration;\\
\STATE $N_1(\azhat) \leftarrow \azhat-\Sigp^{-1}\hat{f}(\app)$;
\STATE $\bzhat \leftarrow N_1(\azhat)$;
\STATE $z \leftarrow (\azone, \bzhat)$;
\STATE $N_2(\azone) \leftarrow \app_1-\left(\Delp\right)^{-1}\frac{\partial f_n(\app)}{\partial X_1}$;

\end{algorithmic}
\end{algorithm}


\begin{definition}\label{defumult2}
For an approximate zero  $\app$ of $f$ with associated simple double zero  $\acc$, let $\gamma_2=\gamma_2(f, \xi)$, $u=\gamma_2^2\|\minus\|$, 
we define the following rational functions:
\begin{eqnarray*}
\btwoone&=&\frac{(1-2u)^2u}{\left[2(1-2u)^2-1\right](1-u)},\\
\btwotwo&=&\frac{u}{\left[2(1-2u)^2-1\right](1-u)},\\
\btwothree&=&\frac{u(32u^6-144u^5+272u^4-288u^3+174u^2-52u+5)}{(24u^3-36u^2+18u-1)(-1+u)^3(8u^2-8u+1)},\\
\btwofour&=&\frac{(-1+2u)^3(-2+u)u}{(24u^3-36u^2+18u-1)(-1+u)^3(8u^2-8u+1)}. 
\end{eqnarray*}

\end{definition}

\begin{thm}\label{thmnewtonmul2}
Let $\acc$ be a simple double zero of  $f$. 
\begin{enumerate}[(1)]
\item If $\itb < \itb_2 \approx 0.0418$, where $\itb_2$ is the smallest positive  solution  of the equation:
 \[
 2\btwoone^2+2 \btwothree^2=1,
  \]
 then the output of Algorithm \ref{alg:newtonmult2} satisfies:
 \[
 \left \|N_f(\app)-\acc \right \| < \left \| \app -\acc  \right \|.
 \]
 \item If $\itb < \itb_2^{\prime} \approx 0.0318$, where $\itb_2^{\prime}$ is the smallest positive solution  of the equation:
 \[
 2\btwoone^2+2\btwothree^2=\frac{1}{4},
  \]
 then after applying $k$ times of the iteration defined in Algorithm \ref{alg:newtonmult2},  we have:
 \[
 \left \|N_f^k(\app)-\acc \right \| < \left( \frac{1}{2} \right)^{2^k-1}  \left \| \app -\acc  \right \|.
 \]

\end{enumerate}
\end{thm}

\begin{lem}\label{normDinverse}
For $u \leq u_2$ we have:
\begin{equation}
\left\|\Sigp^{-1}\Sigc\right\| \leq \frac{(1-2u)^2}{2(1-2u)^2-1}.
\end{equation}
\end{lem}
\begin{proof}
The Talyor's expansions of $\hat{f}(\app)$ and $\Sigp$  at $\acc$ are
\begin{align*}
  \hat{f}(\app)
  &=\Sigc(\azhat-\hat{\acc})+\sum_{k\geq2}\sum_{i=0}^{k}\frac{1}{i!(k-i)!}\frac{\partial^k \hat{f}(\acc)}{\partial X_1^i\partial \hat{X}^{k-i}}(\minusa)^i(\azhat-\hat{\acc})^{k-i},
\end{align*}
and
\begin{align*}
\Sigp=\Sigc+\sum_{k\geq2}\sum_{i=0}^{k-1}\frac{1}{i!(k-i-1)!}\frac{\partial^k \hat{f}(\acc)}{\partial X_1^i\partial \hat{X}^{k-i}}(\minusa)^i(\azhat-\hat{\acc})^{k-i-1}.
\end{align*}
Since $\Sigc^{-1}$ exists, we have
\begin{align*}
  &\Sigc^{-1}\Sigp\\
  =& I_{n-1}+\Sigc^{-1}\sum_{k\geq2}\sum_{i=0}^{k-1}\frac{1}{i!(k-i-1)!}\frac{\partial^k \hat{f}(\acc)}{\partial X_1^i\partial \hat{X}^{k-i}}(\minusa)^i(\azhat-\hat{\acc})^{k-i-1}\\
  =&I_{n-1}+B.\\
\end{align*}
Hence, we have
\begin{align*}
  \|B\|=\left\|\Sigc^{-1}\Sigp-I_{n-1}\right\|&\leq\sum_{k\geq2}\sum_{i=0}^{k-1}\frac{k!}{i!(k-i-1)!}\hat{\gamma}_{2}^{k-1}\|\minus\|^{k-1}\\
  &=\sum_{k\geq2}k\cdot 2^{k-1}(\hat{\gamma}_{2}\|\minus\|)^{k-1}\\
  &\leq \frac{1}{(1-2\hat{\gamma}_{2}\|\minus\|)^2}-1\\
  &\leq \frac{1}{(1-2u)^2}-1.\\
\end{align*}
When $u<u_2$, $\|B\| < 1$, we have
\begin{align*}
  \left\|\Sigp^{-1}\Sigc\right\|=\left\|(I_{n-1}+B)^{-1}\right\|\leq\sum_{k=0}^{\infty}\|B\|^k &\leq 
  \frac{(1-2u)^2}{2(1-2u)^2-1}\\
\end{align*}
\end{proof}

\begin{lem}\label{newtonz2mul2}
For $u \leq u_2$, we have:
\begin{align*}
\left\|N_1(\azhat)-\hat{\acc}\right\| & \leq \frac{\hat{\gamma}_2 \|\azhat-\hat{\acc}\|^2}{\left[2(1-2u)^2-1\right](1-u)}+\frac{(1-2u)^2\hat{\gamma}_2 |\minusa|^2}{\left[2(1-2u)^2-1\right](1-u)}\\
&  \leq \btwoone |\minusa| + \btwotwo \|\azhat-\hat{\acc}\|.
\end{align*}
\end{lem}
\begin{proof}
\begin{align*}
   &\left\|N_1(\azhat)-\hat{\acc}\right\|\\
  =&\left\|\azhat-\hat{\acc}-\Sigp^{-1}\hat{f}(\app)\right\|\\
  =&\left\|\Sigp^{-1}\left[\Sigp(\azhat-\hat{\acc})-\hat{f}(\app)\right]\right\|\\
  =&\left\|\Sigp^{-1}\left[\Sigc(\azhat-\hat{\acc})+\sum_{k\geq2}\sum_{i=0}^{k-1}\frac{k-i}{i!(k-i)!}\frac{\partial^k \hat{f}(\acc)}{\partial X_1^i\partial \hat{X}^{k-i}}(\minusa)^i(\azhat-\hat{\acc})^{k-i}\right.\right.\\
  &\qquad\left.\left.-\Sigc(\azhat-\hat{\acc})-\sum_{k\geq2}\sum_{i=0}^{k}\frac{1}{i!(k-i)!}\frac{\partial^k \hat{f}(\acc)}{\partial X_1^i\partial \hat{X}^{k-i}}(\minusa)^i(\azhat-\hat{\acc})^{k-i}\right]\right\|\\
  \leq&\left\|\Sigp^{-1}\Sigc\right\|\cdot \left\|\Sigc^{-1}\sum_{k\geq2}\sum_{i=0}^{k-2}\frac{k-i-1}{i!(k-i)!}\frac{\partial^k \hat{f}(\acc)}{\partial X_1^i\partial \hat{X}^{k-i}}(\minusa)^i(\azhat-\hat{\acc})^{k-i}\right.\\
  &\qquad\left.-\Sigc^{-1}\sum_{k\geq2}\frac{1}{k!}\frac{\partial^k \hat{f}(\acc)}{\partial X_1^k}(\minusa)^k\right\|\\
  \leq& \left\|\Sigp^{-1}\Sigc\right\| \cdot \left(\sum_{k\geq2}\sum_{i=0}^{k-2}\frac{(k-i-1)k!}{i!(k-i)!}\hat{\gamma}_2^{k-1}\|\minus\|^{k-2}\|\azhat-\hat{\acc}\|^2 \right. \\
 & \qquad \left. +\sum_{k\geq2}\hat{\gamma}_2^{k-1}|\minusa|^k\right)\\
  \leq& \left\|\Sigp^{-1}\Sigc\right\|\cdot\left(\sum_{k\geq2}(k\cdot 2^{k-1}-2^k+1)\hat{\gamma}_2^{k-1}\|\minus\|^{k-2}\|\azhat-\hat{\acc}\|^2 \right.\\
  & \qquad \left.+\sum_{k\geq2}\hat{\gamma}_2^{k-1}\|\minus\|^{k-2}|\minusa|^2\right)\\
  \leq& \frac{(1-2u)^2}{2(1-2u)^2-1}\cdot\left(\frac{\hat{\gamma}_2}{(1-u)(1-2u)^2} \|\azhat-\hat{\acc}\|^2+\frac{\hat{\gamma}_2}{1-u}|\minusa|^2\right)\\
  \leq & \btwoone |\minusa| + \btwotwo \|\azhat-\hat{\acc}\|.
\end{align*}
\end{proof}

\begin{remark}\label{lemma45hold}
The Newton iteration defined by   $N_1$ operator works  for any simple multiple zero of multiplicity $\mu\geq 2$ whose Jacobian matrix has a normalized form. It is clear from the proofs of Lemma \ref{normDinverse} and Lemma \ref{newtonz2mul2}, if we set
$u= {\gamma}_{\mu}^{\mu}\|\minus\|$, then  the conclusions of both lemmas  still hold.
\end{remark}

 Let  $\app=(\azone, \bzhat)$, where $\bzhat=N_1(\azhat)$,  then we have the following Talyor's expansion of $f_n(\app)$ at $\acc$:
\begin{align*}
  f_n(\app)
  &=\sum_{k\geq2}\sum_{i=0}^{k}\frac{1}{i!(k-i)!}\frac{\partial^k f_n(\acc)}{\partial X_1^i\partial \hat{X}^{k-i}}(\minusa)^i(\bzhat-\hat{\acc})^{k-i}.
\end{align*}

\begin{lem}\label{normDelta2inverse}
When $u<u_2$, we have
\[
  \left|\left(\Delp\right)^{-1} \Delc\right| \leq \frac{(2u-1)^3}{24u^3-36u^2+18u-1}. 
\]
\end{lem}
\begin{proof} We have
\begin{align*}
  &\left(\Delc\right)^{-1} \Delp\\
  =& 1+\left(\Delc\right)^{-1}\sum_{k\geq3}\sum_{i=2}^{k}\frac{1}{(i-2)!(k-i)!}\frac{\partial^k f_n(\acc)}{\partial X_1^i\partial \hat{X}^{k-i}}(\minusa)^{i-2}(\minusb)^{k-i}\\
  =& 1+B,\\
\end{align*}
where
\begin{align*}
  |B|&=\left|\left(\Delc\right)^{-1} \Delp-1\right|\\
  &\leq\sum_{k\geq3}\sum_{i=2}^{k}\frac{k!}{(i-2)!(k-i)!}\gamma_{2,n}^{k-1}\|\minus\|^{k-2}\\
  &=\sum_{k\geq3}k(k-1)\cdot 2^{k-2}(\gamma_{2,n}\|\minus\|)^{k-3}\gamma_{2,n}^2\|\minus\|\\
  &\leq \frac{-16u^3+24u^2-12u}{(2u-1)^3}. 
\end{align*}
When $u<u_2$, $|B| < 1$,   we have
\begin{align*}
  \left|\left(\Delp\right)^{-1} \Delc\right|=\left|(1+B)^{-1}\right|\leq\sum_{k=0}^{\infty}|B|^k
  \leq 
  \frac{(2u-1)^3}{24u^3-36u^2+18u-1}. 
\end{align*}
\end{proof}

\begin{lem}\label{newtonz1mul2} When $u<u_2$, we have
\[
\left|N_2(\app_1)-\acc_1\right| \leq \btwothree |\minusa| + \btwofour \|\azhat-\hat{\acc}\|.
\]
\end{lem}
\begin{proof} For  $u<u_2$, we have
\begin{align*}
  &\left|N_2(\app_1)-\acc_1\right|\\
  =&\left|\app_1-\acc_1-\left(\Delp\right)^{-1}\frac{\partial f_n(\app)}{\partial X_1}\right|\\
  =&\left|\left(\Delp\right)^{-1}\left[\Delp(\minusa)-\frac{\partial f_n(\app)}{\partial X_1}\right]\right|\\
  =&\left|\left(\Delp\right)^{-1}\left[\sum_{k\geq2}\sum_{i=2}^{k}\frac{1}{(i-2)!(k-i)!}\frac{\partial^k f_n(\acc)}{\partial X_1^i\partial \hat{X}^{k-i}}(\minusa)^{i-1}(\minusb)^{k-i}\right.\right.\\
  &\qquad\left.\left.-\sum_{k\geq2}\sum_{i=1}^{k}\frac{1}{(i-1)!(k-i)!}\frac{\partial^k f_n(\acc)}{\partial X_1^i\partial \hat{X}^{k-i}}(\minusa)^{i-1}(\minusb)^{k-i}\right]\right|\\
  \leq& \left|\left(\Delp\right)^{-1} \Delc\right|\cdot \left| \left(\Delc\right)^{-1}\sum_{k\geq2}\frac{1}{(k-1)!}\frac{\partial^k \hat{f}(\acc)}{\partial X_1\partial \hat{X}^{k-1}}(\minusb)^{k-1} \right.\\
  &\qquad\left.- \left(\Delc\right)^{-1}\sum_{k\geq3}\sum_{i=3}^{k}\frac{i-2}{(i-1)!(k-i)!}\frac{\partial^k f_n(\acc)}{\partial X_1^i\partial \hat{X}^{k-i}}(\minusa)^{i-1}(\minusb)^{k-i}\right| \\
  \leq& \left|\left(\Delp\right)^{-1} \Delc\right| \cdot\left(\sum_{k\geq3}\sum_{i=3}^{k}\frac{(i-2)k!}{(i-1)!(k-i)!}\gamma_{2,n}^{k-1}\|\minus\|^{k-2}|\minusa| \right. \\
  & \left. \qquad+\sum_{k\geq2}k\gamma_{2,n}^{k-1} \|\minus\|^{k-2} \|\minusb\|\right)\\
  \leq& \left|\left(\Delp\right)^{-1} \Delc\right| \cdot\left(\sum_{k\geq3}\sum_{i=3}^{k}\frac{(i-2)k!}{(i-1)!(k-i)!}
  ({\gamma_{2,n}^2\|\minus\|})^{k-2}|\minusa| \right.\\
  & \left. \qquad+\sum_{k\geq2}k({{\gamma_{2,n}^2\|\minus\|}})^{k-2} \gamma_{2,n} \|\minusb\|\right)\\
    \leq& \left|\left(\Delp\right)^{-1} \Delc\right| \\
   &\qquad\cdot\left(\sum_{k\geq3}(k(k-3)2^{k-2}+k)u^{k-2}|\minusa|+\sum_{k\geq2}ku^{k-2} \gamma_{2,n} \|\minusb\|\right)\\
   \leq& \left|\left(\Delp\right)^{-1} \Delc\right| \cdot\left(\frac{u(4u-3)}{(u-1)^2(2u-1)^3}|\minusa|+\frac{(2-u)\gamma_{2,n}}{(u-1)^2} \|\minusb\|\right).
\end{align*}
Then by Lemma \ref{newtonz2mul2} and Lemma \ref{normDelta2inverse}, we have:
\begin{align*}
  &\left|N_2(\app_1)-\acc_1\right|\\
  \leq&\frac{u(32u^6-144u^5+272u^4-288u^3+174u^2-52u+5)}{(24u^3-36u^2+18u-1)(-1+u)^3(8u^2-8u+1)} |\minusa| \\ 
  & +\frac{(-1+2u)^3(-2+u)u}{(24u^3-36u^2+18u-1)(-1+u)^3(8u^2-8u+1)} \|\azhat-\hat{\acc}\|. \\
  = &\btwothree |\minusa| + \btwofour \|\azhat-\hat{\acc}\|. 
\end{align*}
\end{proof}

\begin{proof}
Now we can complete  the proof of Theorem \ref{thmnewtonmul2}: 

 \begin{enumerate}
 \item For $0 <\itb<\itb_2\approx 0.0418$, we have
 \[2\btwoone^2+2\btwothree^2 <1, \ \ 2\btwotwo^2+2\btwofour^2 <1.\]
 Hence, we have
 \begin{align*}
 &\left \| N_f(\app)-\acc \right \|^2
 \leq  \left\|N_1(\azhat)-\hat{\acc}\right\|^2 +\left|N_2(\app_1)-\acc_1\right| ^2 \\
  & \leq (\btwoone |\minusa| + \btwotwo \|\azhat-\hat{\acc}\|)^2 +(\btwothree |\minusa| + \btwofour \|\azhat-\hat{\acc}\|)^2 \\
 &\leq  (2\btwoone^2+2\btwothree^2) |\minusa|^2 +(2\btwotwo^2+2\btwofour^2) \|\azhat-\hat{\acc}\| ^2\\
 & <\left \| \app -\acc  \right \|^2.
\end{align*}

\item For $0<\itb<\itb_2^{\prime} \approx 0.0318$,  we have
 \[2\btwotwo^2+2\btwofour^2< 2\btwoone^2+2\btwothree^2 <\frac{1}{4}. \]
 Hence, we have
 \begin{align*}
 &\left \| N_f(\app)-\acc \right \|^2
 \leq  \left\|N_1(\azhat)-\hat{\acc}\right\|^2 +\left|N_2(\app_1)-\acc_1\right| ^2 \\
  & \leq (\btwoone |\minusa| + \btwotwo \|\azhat-\hat{\acc}\|)^2 +(\btwothree |\minusa| + \btwofour \|\azhat-\hat{\acc}\|)^2 \\
 &\leq  (2\btwoone^2+2\btwothree^2) |\minusa|^2 +(2\btwotwo^2+2\btwofour^2) \|\azhat-\hat{\acc}\| ^2\\
 & \leq (2\btwoone^2+2\btwothree^2) \left \| \app -\acc  \right \|^2\\
 & \leq \frac{1}{4} \left \| \app -\acc  \right \|^2.
\end{align*}
The following inequality is true for $k=1$:
  \[
 \left \|N_f^k(\app)-\acc \right \| < \left( \frac{1}{2} \right)^{2^k-1}  \left \| \app -\acc  \right \|.
 \]
 %
 For $k \geq 2$, assume by induction that
 \begin{align*}\label{quadmult2ind}
  \left \|N_f^{k-1}(\app)-\acc \right \| < \left( \frac{1}{2} \right)^{2^{k-1}-1}  \left \| \app -\acc  \right \|.
 \end{align*}
   Let $\itb^{(k-1)}=\gamma_2^2\left \|N_f^{k-1}(\app)-\acc \right \|$.  For $0<\itb<\itb_2^{\prime}$,  $k \geq 2$,  we have
    $\itb^{(k-1)} < \itb=\gamma_2^2\|\minus\|$ and  $\frac{\sqrt{2\btwoone^2+2\btwothree^2}\gamma_2^2}{\itb}$ is increasing.  Therefore, we have
 \begin{align*}
 & \left \|N_f^{k}(\app)-\acc \right \|\\
  =& \left \|N_f\left(N_f^{k-1}(\app)\right)-\acc \right \|\\
  <&  \frac{\sqrt{2b_{2,1}(\itb^{(k-1)})^2+2b_{2,3}(\itb^{(k-1)})^2}\gamma_2^2}{\itb^{(k-1)}} \left \| N_f^{k-1}(\app)-\acc  \right \|^2\\
  <&  \frac{\sqrt{2\btwoone^2+2\btwothree^2}\gamma_2^2}{\itb} \left \| N_f^{k-1}(\app)-\acc  \right \|^2\\
  <& \frac{\sqrt{2\btwoone^2+2\btwothree^2}\gamma_2^2}{\itb} \left( \frac{1}{2} \right)^{2^{k}-2}  \left \| \app-\acc  \right \|^2 \\
 = &\left( \frac{1}{2} \right)^{2^k-1}  \left \| \app -\acc  \right \|.
 \end{align*}
\end{enumerate}

\end{proof}

\subsection{$\gamma$-theorem  for  Simple Triple Zeros}

 Given an approximate zero $\app$ of $f$ with associated simple triple zero  $\acc$ such that
  $D\hat{f}(\acc)$ is invertible and
 \[ \frac{\partial f_i (\acc)}{\partial X_1}=0, \ \ \frac{\partial f_n (\acc)}{\partial X_i}=0,  \ \ 1 \leq i \leq n, \ \  \Delc = 0, \]
but
 \[
  \Delta_{3}(f_n)=\frac{1}{6}\frac{\partial^3 f_n(\acc)}{\partial X_1^3}-\frac{\partial^2 f_n(\acc)}{\partial X_1\partial \hat{X}}\cdot \Sigc^{-1}\frac{1}{2}\frac{\partial^2 \hat{f}(\acc)}{\partial X_1^2}\neq 0. 
  \]
  We aim to approximate $\acc$ by applying modified Newton's method to $\app$ and iterating $k$ times such that $\|N^k_f(\app)-\acc\| < \epsilon$ for a given accuracy $\epsilon$.

Let us define the differential  operator $\newl_3$,
\begin{eqnarray*}
  \newl_3(f_n)(\app)&=&\frac{1}{6}\frac{\partial^3 f_n(\app)}{\partial X_1^3}-\frac{\partial^2 f_n(\app)}{\partial X_1\partial \hat{X}}\cdot \atwo,
\end{eqnarray*}
 then we have $\newl_3(f_n)(\acc)=\Delta_{3}(f_n)$.
 As  $\Delta_3(f_n) \neq 0$, and $\app$ is near to $\acc$, we can assume  that $\newl_3(f_n)(\app) \neq 0$.
 Moreover, we define the differential operator $\newd_1$ such that
\begin{eqnarray*}
  \newd_1(f_n)(\app)&=&\frac{1}{6}\frac{\partial^2 f_n(\app)}{\partial X_1^2}- \frac{\partial f_n(\app)}{\partial \hat{X}} \cdot \atwo.
\end{eqnarray*}

\begin{algorithm}[H]
\label{alg:newtonmult3}
\caption{ Modified Newton Iteration for Simple Triple  Zero}
\begin{algorithmic}[1]
\REQUIRE~~\\
$f$: a polynomial system; \\
$\app =(\azone, \azhat)$:  an approximate simple triple zero of $f$; \\
\ENSURE~~\\
$N_f(\app)=(N_2(\azone),N_1(\azhat) )$:  a refined solution after one  iteration;\\
\STATE $N_1(\azhat) \leftarrow \azhat-\Sigp^{-1}\hat{f}(\app)$;
\STATE $\bzhat \leftarrow N_1(\azhat)$;
\STATE $z \leftarrow (\azone, \bzhat)$;
\STATE $N_2(\azone) \leftarrow \azone-\left(\newl_3(f_n)(\app)\right)^{-1} \cdot \newd_1(f_n)(\app)$;
\end{algorithmic}
\end{algorithm}

\begin{definition}\label{defumult3}
For an approximate root $\app$ of  $\acc$, let $\itb=\gamma_3^3\|\minus\|$.  We define the following rational functions:
\begin{eqnarray*}
\acoef&=&\frac{1}{\left[2(1-2u)^2-1\right](1-2u)},\\
 \deltacoef &=&\frac{(2\itb -1)^4 (8\itb^2 -8\itb +1)}{128 \itb^6-384 \itb^5+464 \itb^4 -320 \itb^3 +136 \itb^2 -30 \itb +1} ,  \\ 
 \bthreethree&=& \frac{-\deltacoef}{3(2\itb-1)^4(8\itb^2-8\itb+1)^2(\itb-1)^4} \cdot \left( 3072\itb^{12}-25088\itb^{11}\right.\\ 
 &&+92480\itb^{10}-202336\itb^9+289640\itb^8-282020\itb^7+188614\itb^6\\
 &&\left. -85997\itb^5+26342\itb^4-5368\itb^3+702\itb^2-42\itb\right),\\
\bthreefour&=&\frac{\deltacoef \left( 16\itb^6-72\itb^5+130\itb^4-106\itb^3+42\itb^2-9\itb \right)}{3(8\itb^2-8\itb+1)^2(\itb-1)^4(2\itb-1)}.
\end{eqnarray*}

\end{definition}

\begin{thm}\label{thmnewtonmul3}
Let $\acc$ be a simple triple zero of  $f$. 
\begin{enumerate}[(1)]
\item  If $\itb < \itb_3 \approx 0.0222$, where $\itb_3$ is the smallest positive solution  of the equation:
 \[
 2\btwoone^2+2\bthreethree^2=1,
  \]
 then the output of Algorithm \ref{alg:newtonmult3} satisfies:
 \[
 \left \|N_f(\app)-\acc \right \| < \left \| \app -\acc  \right \|.
 \]
 \item If $\itb < \itb_3^{\prime} \approx 0.0154$, where $\itb_3^{\prime}$ is the smallest positive solution  of the equation:
 \[
 2\btwoone^2+2\bthreethree^2=\frac{1}{4},
  \]
 then after $k$ times of iteration we have
 \[
 \left \|N_f^k(\app)-\acc \right \| < \left( \frac{1}{2} \right)^{2^k-1}  \left \| \app -\acc  \right \|.
 \]
\end{enumerate}
\end{thm}

  According to Remark \ref{lemma45hold}, similar to   Lemma \ref{normDinverse} and Lemma \ref{newtonz2mul2}, for  $\itb \leq \itb_3 \approx 0.0222$,  we  have:
\begin{equation*}
\left\|\Sigp^{-1}\Sigc\right\| \leq \frac{(1-2u)^2}{2(1-2u)^2-1},
\end{equation*}
and
\begin{align*}
\left\|N_1(\azhat)-\hat{\acc}\right\|
&  \leq \btwoone |\minusa| + \btwotwo \|\azhat-\hat{\acc}\|.
\end{align*}

Let  $\app=(\azone, \bzhat)$, where $\bzhat=N_1(\azhat)$,   we have the following Talyor's expansion of $f_n(\app)$ at $\acc$:
\begin{align*}
  f_n(\app)
  &=\sum_{k\geq2}\sum_{i=0}^{k}\frac{1}{i!(k-i)!}\frac{\partial^k f_n(\acc)}{\partial X_1^i\partial \hat{X}^{k-i}}(\minusa)^i(\bzhat-\hat{\acc})^{k-i}.
\end{align*}

\begin{lem}\label{atwomul3}
When $\itb<\itb_3$, we have
\[
\left \|\Sigp^{-1}\cdot\frac{1}{2}\frac{\partial^2 \hat{f}(\app)}{\partial X_1^2} \right \| \leq \acoef \gamma_3.
\]
\end{lem}
\begin{proof}
When $\itb<\itb_3$,  we have
\begin{align*}
  \left\|\Sigp^{-1}\Sigc\right\|\leq \frac{(1-2u)^2}{2(1-2u)^2-1}.\\
\end{align*}
Then, it is clear that
\begin{align*}
  &\left \| \Sigp^{-1}\cdot\frac{1}{2}\frac{\partial^2 \hat{f}(\app)}{\partial X_1^2} \right \|\\
  =&\left \| \Sigp^{-1}\cdot\frac{1}{2}\sum_{k\geq2}\sum_{i=2}^{k}\frac{1}{(i-2)!(k-i)!}\frac{\partial^k \hat{f}(\acc)}{\partial X_1^i\partial \hat{X}^{k-i}}(\minusa)^{i-2}(\bzhat-\hat{\acc})^{k-i} \right \|\\
  =& \left \| \Sigp^{-1}\Sigc\Sigc^{-1} \right. \\
  & \quad \left.\cdot\frac{1}{2}\sum_{k\geq2}\sum_{i=2}^{k}\frac{1}{(i-2)!(k-i)!}\frac{\partial^k \hat{f}(\acc)}{\partial X_1^i\partial \hat{X}^{k-i}}(\minusa)^{i-2}(\bzhat-\hat{\acc})^{k-i} \right \|\\
 \leq & \frac{1}{2} \left\|\Sigp^{-1}\Sigc\right\|\cdot \left(\sum_{k\geq2}\sum_{i=2}^{k}\frac{k!}{(i-2)!(k-i)!}\hat{\gamma}_3^{k-1}\|\minus\|^{k-2}\right)\\
 = &\frac{1}{2} \left\|\Sigp^{-1}\Sigc\right\|\cdot \left(\sum_{k\geq2} k(k-1)(2\hat{\gamma}_3\|\minus\|)^{k-2}\hat{\gamma}_3\right)\\
 \leq & \frac{(1-2u)^2}{2(1-2u)^2-1}\cdot \frac{1}{(1-2u)^3}\gamma_3\\
 = &\acoef \gamma_3.
\end{align*}
\end{proof}

\begin{lem}\label{normDelta3inverse}
When $\itb < \itb_3$, we have
\[
\left \|\newl_3(f_n)(\app)^{-1}\Delta_3(f_n) \right \| \leq \deltacoef.
\]
\end{lem}
\begin{proof}
By the Taylor's expansion of $f_n$ at $\acc$, we have:
\begin{align*}
  \frac{1}{6}\frac{\partial^3 f_n(\app)}{\partial X_1^3}&=
  \frac{1}{6}\frac{\partial^3 f_n(\acc)}{\partial X_1^3}+\frac{1}{6}\sum_{k\geq4}\sum_{i=3}^{k}\frac{1}{(i-3)!(k-i)!}\frac{\partial^k f_n(\acc)}{\partial X_1^i\partial \hat{X}^{k-i}}(\minusa)^{i-3}(\minusb)^{k-i},\\
  \frac{\partial^2 f_n(\app)}{\partial X_1 \partial \hat{X}}&=\frac{\partial^2 f_n(\acc)}{\partial X_1\partial \hat{X}}+\frac{\partial^3 f_n(\acc)}{\partial X_1^2\partial \hat{X}}(\azone-\acc_1)+\frac{\partial^3 f_n(\acc)}{\partial X_1\partial \hat{X}^2}(\bzhat-\hat{\acc})\\
  &+\sum_{k\geq4}\sum_{i=1}^{k-1}\frac{1}{(i-1)!(k-i-1)!}\frac{\partial^k f_n(\acc)}{\partial X_1^i\partial \hat{X}^{k-i}}(\minusa)^{i-1}(\minusb)^{k-i-1}.
  \end{align*}
  Then we have:
\begin{align*}
&\Delta_3(f_n)^{-1}\newl_3(f_n)(\app)\\
=& 1+ \Delta_3(f_n)^{-1} \left [\newl_3(f_n)(\app)-\Delta_3(f_n) \right]\\
=& 1+B,
\end{align*}
where
\begin{align*}
\|B\| =& \left \|\newl_3(f_n)(\acc)^{-1} \cdot \atwo \cdot \left[ \frac{\partial^3 f_n(\acc)}{\partial X_1^2\partial \hat{X}}(\azone-\acc_1)+\frac{\partial^3 f_n(\acc)}{\partial X_1\partial \hat{X}^2}(\bzhat-\hat{\acc}) \right. \right.\\
& \left. +\sum_{k\geq4}\sum_{i=1}^{k-1}\frac{1}{(i-1)!(k-i-1)!}\frac{\partial^k f_n(\acc)}{\partial X_1^i\partial \hat{X}^{k-i}}(\minusa)^{i-1}(\minusb)^{k-i-1}\right]  \\
&\left. -\frac{1}{6}\sum_{k\geq4}\sum_{i=3}^{k}\frac{1}{(i-3)!(k-i)!}\frac{\partial^k f_n(\acc)}{\partial X_1^i\partial \hat{X}^{k-i}}(\minusa)^{i-3}(\minusb)^{k-i} \right\|.
\end{align*}
By Lemma \ref{atwomul3}, we have
\begin{align*}
\|B\|\leq & 12 \acoef \cdot  \gamma_3^3  \cdot \| \app-\acc \| + \acoef \sum_{k\geq4}\sum_{i=1}^{k-1} \binom{k-2}{i-1} k (k-1) \cdot \gamma_3^k \| \app-\acc \|^{k-2} \\
&+\frac{1}{6}\sum_{k\geq4}\sum_{i=3}^{k-1}\binom{k-3}{i-3}(k-1)(k-2)(k-3) \cdot \gamma_3^{k-1} \| \app-\acc \|^{k-3}\\
\leq&12 \acoef \cdot  \gamma_3^3  \cdot \| \app-\acc \| + \acoef \sum_{k\geq4} 2^{k-2} \cdot k \cdot (k-1) \cdot (\gamma_3^2)^{k-2} \| \app-\acc \|^{k-2}\\
&+\frac{1}{6}\sum_{k\geq4} 2^{k-3}\cdot (k-1) \cdot(k-2) \cdot(k-3) \cdot (\gamma_3^3)^{k-3} \| \app-\acc \|^{k-3}\\
\leq &\acoef \left ( 12 \itb + \sum_{k\geq4} k (k-1)2^{k-2}  \itb^{k-2}\right) +\frac{1}{6}\sum_{k\geq4}  (k-1) (k-2) (k-3)2^{k-3}  \itb^{k-3} \\
= &\frac{2\itb (16 \itb^2 -20\itb +7)}{(2\itb-1)^4 (8\itb^2-8\itb +1)}.
\end{align*}

When $\itb <\itb_3$, $\| B \| <1$,  we have
\begin{align*}
\left \| \newl_3(f_n)(\app)^{-1}\Delta_3(f_n)\right \| \leq \|(1+B)^{-1} \| 
 \leq \deltacoef.
\end{align*}

\end{proof}

\begin{lem}\label{newtonz1mult3}When $\itb < \itb_3$, we have
\[
\left|N_2(\azone)-\acc_1\right|  \leq \bthreethree(\itb) |\minusa| + \bthreefour(\itb) \|\azhat-\hat{\acc}\|.
\]
\end{lem}
\begin{proof} We have
  \begin{align*}
  \left|N_2(\azone)-\acc_1\right|=&\left|\minusa-\left[\newl_3(f_n)(\app)\right]^{-1}\newd_1(f_n)(\app)\right|\\
  =&\left|\newl_3(f_n)(\app)^{-1}\left[\newl_3(f_n)(\app)(\minusa)-\newd_1(f_n)(\app)\right]\right|\\
  =&\left|\newl_3(f_n)(\app)^{-1}\Delta_3(f_n)\right|\cdot\left|\Delta_3(f_n)^{-1}
  \left[\newl_3(f_n)(\app)(\minusa)-\newd_1(f_n)(\app)\right]\right|.
\end{align*}
From the Taylor expansions of  $\frac{\partial^3 f_n(\app)}{\partial X_1^3},\frac{\partial^2 f_n(\app)}{\partial X_1^2},  \frac{\partial^2 f_n(\app)}{\partial X_1 \partial \hat{X}}, \frac{\partial f_n(\app)}{\partial \hat{X}}$ at $\acc$, we have
\begin{align*}
&\newl_3(f_n)(\app)(\minusa)-\newd_1(f_n)(\app) \\
=&-\frac{1}{6}\frac{\partial^3 f_n(\acc)}{\partial X_1^2 \partial \hat{X}}(\bzhat-\hat{\acc})-\frac{1}{6}\sum_{k\geq4}\frac{1}{(k-2)!}\frac{\partial^k f_n(\acc)}{\partial X_1^2\partial \hat{X}^{k-2}}(\minusb)^{k-2}\\
&+\frac{1}{6}\sum_{k\geq4}\sum_{i=3}^{k}\frac{i-3}{(i-2)!(k-i)!}\frac{\partial^k f_n(\acc)}{\partial X_1^i\partial \hat{X}^{k-i}}(\minusa)^{i-2}(\minusb)^{k-i}\\
&-\atwo \left[ \frac{1}{2}\frac{\partial^3 f_n(\acc)}{\partial X_1^2\partial \hat{X}}(\azone-\acc_1)^2-\frac{\partial^2 f_n(\acc)}{\partial \hat{X}^2}(\bzhat-\hat{\acc})-\frac{1}{2}\frac{\partial^3 f_n(\acc)}{\partial \hat{X}^3}(\bzhat-\hat{\acc})^2 \right. \\
& \left.+\sum_{k\geq4}\sum_{i=1}^{k-1}\frac{i-1}{i!(k-i-1)!}\frac{\partial^k f_n(\acc)}{\partial X_1^i\partial \hat{X}^{k-i}}(\minusa)^{i}(\minusb)^{k-i-1} \right. \\
& \left.-\sum_{k\geq4}\frac{1}{(k-1)!}\frac{\partial^k f_n(\acc)}{\partial \hat{X}^{k}}(\minusb)^{k-1} \right].
\end{align*}

Then we have
   \begin{align*}
   &  \left \|\frac{1}{|\Delta_3(f_n)|} \cdot \left( -\frac{1}{6}\frac{\partial^3 f_n(\acc)}{\partial X_1^2 \partial \hat{X}} \right)(\bzhat-\hat{\acc}) \right \| +\left \| \atwo \right \| \cdot \left \|\frac{1}{|\Delta_3(f_n)|} \cdot \frac{\partial^2 f_n(\acc)}{\partial \hat{X}^2}(\bzhat-\hat{\acc}) \right \|  \\
   \leq & \gamma_3^2 \cdot \left \| (\bzhat-\hat{\acc}) \right \|  +2 \acoef \gamma_3^2 \cdot \left \| (\bzhat-\hat{\acc}) \right \|\\
   \leq &(1+2\acoef) \left\|\Sigp^{-1}\Sigc\right\| \cdot \left(\sum_{k\geq2}\gamma_3^{k+1}\|\minus\|^{k-1}|\minusa| \right. \\
& \left.+ \sum_{k\geq2}(k\cdot 2^{k-1}-2^k+1)\gamma_3^{k+1}\|\minus\|^{k-1}\|\azhat-\hat{\acc}\|\right) \\
\leq &(1+2\acoef) \left\|\Sigp^{-1}\Sigc\right\| \cdot \left(\sum_{k\geq2}(\gamma_3^3 \|\minus\|)^{k-1}|\minusa| \right. \\
& \left.+ \sum_{k\geq2}(k\cdot 2^{k-1}-2^k+1)(\gamma_3^3 \|\minus\|)^{k-1}\|\azhat-\hat{\acc}\|\right) \\
\leq & \frac{\itb(2\itb-1)^2(1+2\acoef)}{(8\itb^2-8\itb+1)(1-\itb)}|\minusa| +\frac{\itb(1+2\acoef)}{(8\itb^2-8\itb+1)(1-\itb)} \|\azhat- \hat{\acc}\|. 
   \end{align*}
Furthermore, we have
\begin{align*}
&\frac{1}{|\Delta_3(f_n)|} \cdot \left\|\frac{1}{6}\sum_{k\geq4}\sum_{i=3}^{k}\frac{i-3}{(i-2)!(k-i)!}\frac{\partial^k f_n(\acc)}{\partial X_1^i\partial \hat{X}^{k-i}}(\minusa)^{i-2}(\minusb)^{k-i}\right \|\\
\leq & \frac{1}{6} \sum_{k\geq4}\sum_{i=3}^{k}\frac{k! \cdot (i-3)}{(i-2)!(k-i)!} \gamma_3^{k-1} \| \app-\acc  \|^{k-3} \cdot |\azone-\acc_1|\\
\leq  & \frac{1}{6} \sum_{k\geq4}\sum_{i=3}^{k} \binom{k-3}{i-3} \cdot k \cdot (k-1) \cdot (k-2) \cdot \gamma_3^{3k-9} \cdot  \| \app-\acc  \|^{k-3} \cdot |\azone-\acc_1|\\
\leq &\frac{1}{6} \sum_{k\geq4} k \cdot (k-1) \cdot (k-2) \cdot 2^{k-3} \cdot ( \gamma_3^3 \| \app-\acc  \| )^{k-3} |\azone-\acc_1| \\
\leq &\frac{1}{6} \sum_{k\geq4} k \cdot (k-1) \cdot (k-2)  \cdot 2^{k-3} \cdot \itb^{k-3}  |\azone-\acc_1| \\
= & \frac{-8\itb(2\itb^3-4\itb^2+3\itb-1)}{(2\itb-1)^4} |\azone-\acc_1|.
\end{align*}
\item
We have
\begin{align*}
&\frac{1}{|\Delta_3(f_n)|} \cdot \left\| \frac{1}{6}\sum_{k\geq4}\frac{1}{(k-2)!}\frac{\partial^k f_n(\acc)}{\partial X_1^2\partial \hat{X}^{k-2}}(\minusb)^{k-2} \right\| \\
\leq & \frac{1}{6} \sum_{k\geq4} k (k-1) \gamma_3^{k-1} \left\| \app-\acc \right\| ^{k-3} \|\bzhat-\hat{\acc}\|   \\
\leq & \frac{1}{6} \sum_{k\geq4} k (k-1) (\gamma_3^3 \left\| \app-\acc \right\|) ^{k-3} \|\bzhat-\hat{\acc}\| \\
\leq &  \frac{-\itb(3\itb^2-8\itb+6)}{3(\itb-1)^3} \|\bzhat-\hat{\acc}\| \\
\leq &\frac{(2\itb-1)^2\itb^2(3\itb^2-8\itb+6)}{3(8\itb^2-8\itb+1)(\itb-1)^4} |\azone-\acc_1| +\frac{\itb^2(3\itb^2-8\itb+6)}{3(8\itb^2-8\itb+1)(\itb-1)^4}|\azhat-\hat{\acc}\|
\end{align*}
By Lemma \ref{atwomul3}, we have
\begin{align*}
 \frac{\left\|\atwo\right\|}{|\Delta_3(f_n)|}\left \| \frac{1}{2}\frac{\partial^3 f_n(\acc)}{\partial X_1^2\partial \hat{X}}(\azone-\acc_1)^2 \right \| & \leq 3 \acoef \gamma_3^3 \left\| \app-\acc \right\| |\azone-\acc_1|\\
& \leq 3 \acoef \itb |\azone-\acc_1|.
\end{align*}
By Lemma \ref{atwomul3}, we have
\begin{align*}
&\frac{\left \| \atwo \right \|}
 {|\Delta_3(f_n)|} \left \| \frac{1}{2}\frac{\partial^3 f_n(\acc)}{\partial \hat{X}^3}(\bzhat-\hat{\acc})^2 \right \| \\
\leq & 3 \acoef \gamma_3^3 \left\| \app-\acc \right\| \|\bzhat-\hat{\acc}\| \\
 \leq & \frac{ 3 \acoef \itb^2}{\left[2(1-2\itb)^2-1\right](1-\itb)}\|\azhat-\hat{\acc}\|+\frac{3 \acoef \itb^2(1-2\itb)^2}
 {\left[2(1-2\itb)^2-1\right](1-\itb)}|\minusa|.
\end{align*}
We have
\begin{align*}
&\frac{\left\|\atwo\right\|}{|\Delta_3(f_n)|} \cdot \left \| \sum_{k\geq4}\sum_{i=1}^{k-1}\frac{i-1}{i!(k-i-1)!}\frac{\partial^k f_n(\acc)}{\partial X_1^i\partial \hat{X}^{k-i}}(\minusa)^{i}(\minusb)^{k-i-1} \right \|\\
\leq &\acoef \sum_{k\geq4}\sum_{i=1}^{k-1}\frac{k! \cdot (i-1)}{i!(k-i-1)!} \gamma_3^{k}  \| \app-\acc  \|^{k-2} \cdot |\azone-\acc_1|\\
\leq &\acoef \sum_{k\geq4} (k-1) \cdot(k-3) \cdot 2^k \cdot  \itb^{k-2} |\azone-\acc_1| \\
\leq &\frac{16\itb^2(2\itb-3)\acoef}{(\itb-1)^3}|\azone-\acc_1|.
\end{align*}
 By Lemma \ref{atwomul3}, we have
\begin{align*}
& \frac{\left\|\atwo\right\|}{|\Delta_3(f_n)|} \cdot \left \| \sum_{k\geq4}\frac{1}{(k-1)!}\frac{\partial^k f_n(\acc)}{\partial \hat{X}^{k}}(\minusb)^{k-1} \right \| \\
\leq &\acoef  \sum_{k\geq4} k  \gamma_3^{k} \left\| \app-\acc \right\| ^{k-2} \|\bzhat-\hat{\acc}\| \\
\leq &\frac{-\itb^2(3\itb-4)\acoef}{(\itb-1)^2} \|\bzhat-\hat{\acc}\| \\
\leq & \frac{(2\itb-1)^2\itb^3(3\itb-4)\acoef}{(8\itb^2-8\itb+1)(\itb-1)^3} |\azone-\acc_1|+ \frac{\itb^3(3\itb-4)\acoef}{(8\itb^2-8\itb+1)(\itb-1)^3}\|\azhat-\hat{\acc}\|.
\end{align*}
Finally,  by Lemma \ref{normDelta3inverse} and the above estimations, we have
\begin{align*}
&\left|N_2(\azone)-\acc_1\right| \\
\leq &\deltacoef \cdot \left( \left(\frac{\itb(2\itb-1)^2(1+2 \acoef)}{(8\itb^2-8\itb+1)(1-\itb)}+\frac{16\itb^2(2\itb-3)\acoef}{(\itb-1)^3} \right. \right. \\
&\left. \left.+\frac{-8\itb(2\itb^3-4\itb^2+3\itb-1)}{(2\itb-1)^4}+\frac{3\itb^2(1-2\itb)^2 \acoef}{\left[2(1-2\itb)^2-1\right](1-\itb)} + 3\itb \acoef + \right. \right. \\
 &\left. \left. \frac{(2\itb-1)^2\itb^2(3\itb^2-8\itb+6)}{3(8\itb^2-8\itb+1)(\itb-1)^4} +\frac{(2\itb-1)^2\itb^3(3\itb-4) \acoef}{(8\itb^2-8\itb+1)(\itb-1)^3}  \right) \cdot |\azone-\acc_1|   \right.\\
& \left. + \left( \frac{\itb(1+2\acoef)}{(8\itb^2-8\itb+1)(1-\itb)}+\frac{3\itb^2 \acoef}{\left[2(1-2\itb)^2-1\right](1-\itb)}+ \right. \right. \\
& \left. \left. \frac{\itb^2(3\itb^2-8\itb+6)}{3(8\itb^2-8\itb+1)(\itb-1)^4} + \frac{\itb^3(3\itb-4)\acoef}{(8\itb^2-8\itb+1)(\itb-1)^3} \right) \cdot \|\azhat-\hat{\acc}\| \right)\\
=&\bthreethree |\minusa| + \bthreefour \|\azhat-\hat{\acc}\|.
\end{align*}
\end{proof}

\begin{proof}
 Now we can complete  the proof of  Theorem \ref{thmnewtonmul3}. 

 \begin{enumerate}[(1)]
 \item For $\itb<\itb_3 \approx 0.0222$,  it is true that
 \[2\btwoone^2+2\bthreethree^2 <1, \ 2\btwotwo^2+2\bthreefour^2 <1.\]
 Therefore, we have
 \begin{align*}
 &\left \| N_f(\app)-\acc \right \|^2 \\
 \leq & \left\|N_1(\azhat)-\hat{\acc}\right\|^2 +\left|N_2(\app_1)-\acc_1\right| ^2 \\
 \leq & (\btwoone |\minusa| + \btwotwo \|\azhat-\hat{\acc}\|)^2 +(\bthreethree |\minusa| + \bthreefour \|\azhat-\hat{\acc}\|)^2 \\
 \leq & (2\btwoone^2+2\bthreethree^2) |\minusa|^2 +(2\btwotwo^2+2\bthreefour^2) \|\azhat-\hat{\acc}\| ^2\\
 < & \left \| \app -\acc  \right \|^2.
\end{align*}

 \item For $\itb<\itb_3^{\prime}\approx 0.0154$, it is true that
   \[ 2\btwotwo^2+2\bthreefour^2 < 2\btwoone^2+2\bthreethree^2  < \frac{1}{4}.\]
    Hence,  we have
 \begin{align*}
 &\left \| N_f(\app)-\acc \right \|^2
 \leq  \left\|N_1(\azhat)-\hat{\acc}\right\|^2 +\left|N_2(\app_1)-\acc_1\right| ^2 \\
  & \leq (\btwoone |\minusa| + \btwotwo \|\azhat-\hat{\acc}\|)^2 +(\bthreethree |\minusa| + \bthreefour \|\azhat-\hat{\acc}\|)^2 \\
 &\leq  (2\btwoone^2+2\bthreethree^2) |\minusa|^2 +(2\btwotwo^2+2\bthreefour^2) \|\azhat-\hat{\acc}\| ^2\\
 & \leq (2\btwoone^2+2\bthreethree^2) \left \| \app -\acc  \right \|^2\\
& \leq \frac{1}{4} \left \| \app -\acc  \right \|^2.
\end{align*}
Hence, the following inequality is true for $k=1$:
  \[
 \left \|N_f^k(\app)-\acc \right \| < \left( \frac{1}{2} \right)^{2^k-1}  \left \| \app -\acc  \right \|.
 \]
 For $k \geq 2$, assume by induction that
 \begin{align*}\label{quadmult2ind}
  \left \|N_f^{k-1}(\app)-\acc \right \| < \left( \frac{1}{2} \right)^{2^{k-1}-1}  \left \| \app -\acc  \right \|.
 \end{align*}
   Let $\itb^{(k-1)}=\gamma_3^3\left \|N_f^{k-1}(\app)-\acc \right \|$.  For $0<\itb<\itb_3^{\prime}$,  we have
    $\itb^{(k-1)} < \itb=\gamma_3^3\|\minus\|$ and  $\frac{\sqrt{2\btwoone^2+2\bthreethree^2}\gamma_3^3}{\itb}$  is increasing.  Therefore, we have

 \begin{align*}
 & \left \|N_f^{k}(\app)-\acc \right \|\\
  =& \left \|N_f\left(N_f^{k-1}(\app)\right)-\acc \right \|\\
  <&  \frac{\sqrt{2b_{2,1}(\itb^{(k-1)})^2+2b_{3,3}(\itb^{(k-1)})^2}\gamma_3^3}{\itb^{(k-1)}} \left \| N_f^{k-1}(\app)-\acc  \right \|^2\\
  <&  \frac{\sqrt{2\btwoone^2+2\bthreethree^2}\gamma_3^3}{\itb} \left \| N_f^{k-1}(\app)-\acc  \right \|^2\\
  <& \frac{\sqrt{2\btwoone^2+2\bthreethree^2}\gamma_3^3}{\itb} \left( \frac{1}{2} \right)^{2^k-2}  \left \| \app-\acc  \right \|^2 \\
 = &\left( \frac{1}{2} \right)^{2^k-1}  \left \| \app -\acc  \right \|.
 \end{align*}
\end{enumerate}
\end{proof}

\subsection{Simple Multiple Zeros}

 For simple double zeros and simple triple zeros of $f$, we have  defined modified Newton iterations  based on the  first, second and third order differential operators computed at the approximate solutions, and provided quantified criterions to guarantee its quadratic convergence. Although it is possible to extend the modified Newton iterations defined in Algorithm \ref{alg:newtonmult2}, \ref{alg:newtonmult3} to  simple multiple zeros of higher multiplicities,   the  iterations are only defined for systems whose  Jacobian matrix at the exact multiple zero has a normalized form (\ref{normalizeform}), they might be of    limited applications.

 In order to refining an approximate simple singular zero whose Jacobian matrix has corank one but it does not have a normalized form (\ref{normalizeform}), in  Algorithm  \ref{MultiStructure},  we  perform  the unitary transformations to  both variables and equations    defined   at the  approximate simple singular solutions, then  we define the modified  Newton iterations based on our previous work in \cite{LZ:2011}.   We  show firstly its non-quantified quadratic convergence  for simple multiple zeros of higher multiplicities, and then  its quantified convergence for simple triple zero.

\begin{algorithm}[H]
\label{MultiStructure}
\caption{ Modified  Newton Iteration  for Simple Multiple Zeros}
\begin{algorithmic}[1]
\REQUIRE~~\\
$\Ff$: a polynomial system; \\
$\simplex$:  an approximate simple multiple zero; \\
$\mu$: the multiplicity;\\
\ENSURE~~\\
$N_f(\simplex)$:  a refined solution after one  iteration;\\

\STATE
     $D\Ff(\simplex)=U\cdot\left(
      \begin{array}{cc}
      \Sigma_{n-1} & 0 \\
      0 & \sigma_n \\
       \end{array}
       \right)\cdot V^{\ast}$, $W_{\dag}=(v_n,v_1,\ldots,v_{n-1})$;

\STATE  $\Ff(X)\leftarrow U^{\ast}\cdot \Ff(W_{\dag}\cdot X)$, \quad $\Ax  \leftarrow W_{\dag}^{\ast}\simplex$;

\STATE   $N_1(\hat{\Ff},\hat{\Ax}) \leftarrow \hat{\simplex }-D\hat{\Ff}(\simplex)^{-1}\hat{\Ff}(\simplex)$, \quad
 $ \Bx=(\Bx_1, \hat{\Bx}) \leftarrow (\simplex_1, N_1(\hat{\Ff},\hat{\Ax}))$;

 \STATE $D\Ff(\simplex)=U\cdot\left(
      \begin{array}{cc}
      \Sigma_{n-1} & 0 \\
      0 & \sigma_n \\
       \end{array}
       \right)\cdot V^{\ast}$, $W_{\ddag}=(v_n,v_1,\ldots,v_{n-1})$;
\STATE
     $\Sf(X)\leftarrow U^{\ast}\cdot \Ff(W_{\ddag}\cdot X)$, \quad $\Cx=(\Cx_1,\hat{\Cx})  \leftarrow W_{\ddag}^{\ast} \Bx$;

\STATE
$N_2(\Sf_n,\Cx) \leftarrow \Cx_1-\frac{1}{\mu}\Delta_{\mu}(\Sf_n)^{-1}\Delta_{\mu-1}(\Sf_n)$,  \quad $ \Dx=(\Dx_1,\hat{\Dx}) \leftarrow (N_2(\Sf_n,\Cx), \hat{\Cx})$;

\STATE $N_f(\simplex) \leftarrow W_{\dag}\cdot W_{\ddag}\cdot \Dx $.

\end{algorithmic}

\end{algorithm}

\begin{thm}\label{thm4.1}
 Given an approximate zero $\simplex$ of a system $f$ associated to a simple multiple  zero  $\xi$  of multiplicity $\mu$ and satisfying
 $f(\xi)=0$, $\dim \ker Df(\xi)=1$.
 Suppose
 \[\hat{\gamma}_{\mu}(\Ff,\Ax)\|\Ax-\Exi\|<\frac{1}{2},\]
 where
 \[\hat{\gamma}_{\mu}(\Ff,\Ax)=\max \left\{1, \sup_{k\geq2}\left\|D\hat{f}(\Ax)^{-1}\frac{D^k \hat{f}(\Ax)}{k!}\right\|^{\frac{1}{k-1}} \right\},\] 
%
then the refined singular solution $N_f(\simplex)$ returned by     Algorithm \ref{MultiStructure} satisfies
\begin{equation}\label{quadconv}
\|N_f(\simplex)- \Exi\|=\bigO(\|\simplex-\xi\|^2).
\end{equation}
\end{thm}

In what follows, we give quantitative analysis of the convergency of the first five steps in   Algorithm \ref{MultiStructure}. For Step 6, we  show its  non-quantified quadratic convergence first, and then show its quantified convergency for simple triple zeros,  which can be generalized naturally to  simple multiple zeros of higher multiplicities.

In the second step of   Algorithm \ref{MultiStructure}, we  perform  the  unitary transformations to  both variables and equations  according to the singular value decomposition of the Jacobian matrix $D\Ff(\simplex)$.
Since $\hat{\gamma}_{\mu}(\Ff,\Ax)$ and the Euclidean distance between zeros $\xi$ and $\simplex$  do not change under the unitary
transformation, in what follows,
for simplicity, after the first two steps, we use the same notations for $\Ff,\xi, \simplex$, i.e.,
 \begin{equation}\label{ksizf}
  \xidag \leftarrow W_{\dag}^{\ast} \xi, \quad \Ax  \leftarrow W_{\dag}^{\ast}\simplex, \quad \Ff(X)\leftarrow U^{\ast}\cdot \Ff(W_{\dag}\cdot X).
  \end{equation}

 %
%



    \begin{claim}\label{normalisation_1}
    After the first two steps in  Algorithm \ref{MultiStructure}, we have
     \begin{equation}\label{approxnormal4}
     Df(\simplex)=\left(
     \begin{array}{cc}
      0 & \Sigma_{n-1} \\
      \sigma_n & 0 \\
      \end{array}
       \right),
     \end{equation}
     where $\Sigg$ is a nonsingular  diagonal matrix.  Moreover, we have $\sigma_n\leq L\|\simplex-\xi\|$,  where  $L$ is the Lipschitz constant of the function $Df(X)$.
    \end{claim}

    \begin{proof}
    According to the chain rule, we have
    \begin{align*}
  Df(\simplex) & = U^{\ast}\cdot U\cdot\left(
                                                         \begin{array}{cc}
                                                           \Sigma_{n-1} & 0 \\
                                                           0 & \sigma_n \\
                                                         \end{array}
                                                       \right)\cdot V^{\ast}\cdot W_{\dag}\\
   & = \left(
         \begin{array}{cc}
           \Sigg & 0 \\
            0 & \sigma_n \\
         \end{array}
       \right)
   \cdot \left(
           \begin{array}{cc}
             0 & I_{n-1} \\
             1 & 0 \\
           \end{array}
         \right)\\
         & = \left(
               \begin{array}{cc}
                 0 & \Sigg \\
                 \sigma_n & 0 \\
               \end{array}
             \right).
\end{align*}
Furthermore, since $\dim \ker Df(\xi)=1$, the following perturbation theorem about the singular values can be found in  \cite{GK1969, Golub:1996},
\[\sigma_n\leq\|Df(\simplex)-Df(\xi)\|\leq L\|\simplex-\xi\|.\]
\end{proof}


%

    \begin{claim}\label{N1quadratic}
    After running  the first three  steps in  Algorithm \ref{MultiStructure}, suppose
 $\hat{\gamma}_{\mu}(\Ff,\Ax)\|\Ax-\Exi\|<\frac{1}{2}$,  we have
     \begin{equation}\label{step31}
     \|\hat{\Bx}-\hat{\xidag}\|\leq\frac{1}{1-\kappatwo}\hat{\gamma}_{\mu}(\Ff,\Ax)\|\Exi-\Ax\|^2,
     \end{equation}
     and
     \begin{equation}\label{step32}
     \|\hat{f}(\Bx)\|\leq\frac{4\|D\hat{\Ff}(\Ax)\|}{1-2\kappatwo}\hat{\gamma}_{\mu}(\Ff,\Ax)\|\Exi-\Ax\|^2,
     \end{equation}
    where
    \begin{equation}
    \Bx=(\Bx_1, \hat{\Bx})  \leftarrow (\simplex_1, N_1(\hat{\Ff},\hat{\Ax})).
    \end{equation}
    \end{claim}

    \begin{proof}
    According to  Claim \ref{normalisation_1}, we have
     $\frac{\partial\hat{f}(\simplex)}{\partial X_1}=0$ and $D\hat{f}(\simplex)=\Sigg$ is an invertible diagonal matrix.
Therefore, by the Taylor expansion of $\hat{f}$  at $z$, we have
    \[0=\hat{\Ff}(\Exi)=\hat{\Ff}(\Ax)+D\hat{\Ff}(\Ax)(\hat{\Exi}-\hat{\Ax})+
    \sum_{k\geq2}\frac{D^k\hat{\Ff}(\Ax)}{k!}(\Exi-\Ax)^k,\]
    \[0=D\hat{\Ff}(\Ax)^{-1} \hat{\Ff}(\Ax)+\hat{\Exi}-\hat{\Ax}+\sum_{k\geq2}D\hat{\Ff}(\Ax)^{-1}
    \frac{D^k\hat{\Ff}(\Ax)}{k!}(\Exi-\Ax)^k.\]
  Then we have
    \begin{align*}
      \|N_1(\hat{\Ff},\Ax)-\hat{\Exi}\| & \leq\sum_{k\geq2}\left\|D\hat{\Ff}(\Ax)^{-1}\frac{D^k\hat{\Ff}(\Ax)}{k!}\right\|\|\Exi-\Ax\|^k\\
      &\leq \sum_{k\geq2}\hat{\gamma}_{\mu}(\Ff,\Ax)^{k-1}\|\Exi-\Ax\|^k\\
      &\leq \hat{\gamma}_{\mu}(\Ff,\Ax)\|\Exi-\Ax\|^2\sum_{k\geq2}\hat{\gamma}_{\mu}(\Ff,\Ax)^{k-2}\|\Exi-\Ax\|^{k-2}\\
      &\leq \frac{1}{1-\kappatwo}\hat{\gamma}_{\mu}(\Ff,\Ax)\|\Exi-\Ax\|^2.
    \end{align*}

Furthermore, 
 we have
    \begin{align*}
      \|\hat{\Ff}(\Bx)\| & =\left\|\hat{\Ff}(\Ax)+D\hat{\Ff}(\Ax)(\hat{\Bx}-\hat{\Ax})+\sum_{k\geq2}\frac{D^k\hat{\Ff}(\Ax)}{k!}(\Bx-\Ax)^k\right\| \\
       & \leq \|D\hat{\Ff}(\Ax)\|\sum_{k\geq2}\left\|D\hat{\Ff}(\Ax)^{-1}\frac{D^k\hat{\Ff}(\Ax)}{k!}\right\|\|\Bx-\Ax\|^k \\
       & \leq \|D\hat{\Ff}(\Ax)\|\sum_{k\geq2}\hat{\gamma}_{\mu}(\Ff,\Ax)^{k-1}\|\Bx-\Ax\|^k\\
       & \leq \|D\hat{\Ff}(\Ax)\| \hat{\gamma}_{\mu}(\Ff,\Ax) \|\Bx-\Ax\|^2\sum_{k\geq2}\hat{\gamma}_{\mu}(\Ff,\Ax)^{k-2}\|\Bx-\Ax\|^{k-2}\\
       & \leq \frac{4\|D\hat{\Ff}(\Ax)\|}{1-2\kappatwo}\hat{\gamma}_{\mu}(\Ff,\Ax)\|\Exi-\Ax\|^2,
    \end{align*}
    where
    \[\|\Bx-\Ax\|=\|\hat{\Bx}-\hat{\Ax}\|\leq\|\hat{\Bx}-\hat{\Exi}\|+\|\hat{\Ax}-\hat{\Exi}\|
    \leq\frac{\kappatwo}{1-\kappatwo}\|\Exi-\Ax\|+\|\Exi-\Ax\|\leq2\|\Exi-\Ax\|,\]
    and
    \[\hat{\gamma}_{\mu}(\Ff,\Ax)\|\Bx-\Ax\| \leq 2\kappatwo<1.\]
    \end{proof}


Let $(\ran_\mathbb{C}\{v_n(\Ax)\}, \ran_\mathbb{C}\{u_n(\Ax)\})$  be a pair of singular subspaces of $D\Ff(\Ax)$ corresponding to its smallest singular value $\sigma_n$, 
 and  $\delta=\sigma_{n-1}-\sigma_n=\bigO(1)$.  If
    \begin{equation}\label{delta5}
    \|D\Ff(\Bx)-D\Ff(\Ax)\|_F\leq\frac{\delta}{5},
    \end{equation}
   which could be  satisfied in general since $y$ is close to $z$ and $\delta=\bigO(1)$, then according to \cite[Theorem 8.6.5]{Golub:1996} or \cite[Theorem 6.4]{Stewart:1973}, we have
    \begin{equation}\label{vnyz}
    \|v_n(\Bx)-v_n(\Ax)\|_F\leq4\frac{\|D\Ff(\Bx)-D\Ff(\Ax)\|_F}
    {\delta}\leq4\frac{L\|\Bx-\Ax\|}{\delta}\leq8\frac{L\|\Exi-\Ax\|}{\delta},
    \end{equation}
    and  \begin{equation}\label{unyz}
    \|u_n(\Bx)-u_n(\Ax)\|_F\leq4\frac{\|D\Ff(\Bx)-D\Ff(\Ax)\|_F}
    {\delta}\leq4\frac{L\|\Bx-\Ax\|}{\delta}\leq8\frac{L\|\Exi-\Ax\|}{\delta},
    \end{equation}
  where $(\ran_\mathbb{C}\{v_n(\Bx)\},\ran_\mathbb{C}\{u_n(\Bx)\})$ is a pair of singular subspaces of $D\Ff(\Bx)$ corresponding to its smallest singular value.

   In what follows,  for the sake of simplicity, we always assume (\ref{delta5}) is satisfied and  set
   \[L \leftarrow \frac{8L}{\delta}.\]

According to  (\ref{approxnormal4}), we know that
   $v_n(\Ax)=(1,0,\ldots,0)^T$ and $u_n(\Ax)=(0,\ldots,0,1)$ generate  a pair of singular subspaces of  $D\Ff(\Ax)$ corresponding to its smallest singular value $\sigma_n$.

 Let
 \[W_{\ddag}=(v_n(\Bx),v_1(\Bx),\ldots,v_{n-1}(\Bx))=
    \left(
      \begin{array}{cc}
      W_1 & W_2 \\
       W_3 & W_4 \\
       \end{array}
       \right),
    \]
     and $v_n(\Bx)=(W_1,W_3)^T$,
       by (\ref{vnyz}), we have
    \begin{equation}\label{w13}
    |W_1|\geq1-L\|\Exi-\Ax\|,~\|W_3\|\leq L\|\Exi-\Ax\|.
    \end{equation}
    Since $W_{\ddag}$ is a unitary matrix, we have
    \begin{equation}\label{w24}
    \|W_2\|\leq L\|\Exi-\Ax\|,~\|W_4\|\leq\|W_{\ddag}\|=1.
    \end{equation}
 Let
        $U=\left(
        \begin{array}{cc}
        U_1 & U_2 \\
        U_3 & U_4 \\
        \end{array}
        \right)
    $ and $u_n(\Bx)=(U_3,U_4)$,  by (\ref{vnyz}), we have
    \begin{equation}\label{U13}
   \|U_3^*\|\leq L\|\Exi-\Ax\|.
    \end{equation}

 It is clear that  Step 4 and Step 5 in  Algorithm \ref{MultiStructure} are used to normalize the Jacobian matrix at the approximate solution $\Bx$ again after  running  the Newton iteration  for $\hat \simplex$, i.e.,  we  have
 \[Df(\Cx)=\left(
  \begin{array}{cc}
  0 & \Sigma_{n-1} \\
  \sigma_n & 0 \\
  \end{array}
   \right). \]

    \begin{claim}\label{normalisation_2}
    After running  Step 4 and Step 5 in  Algorithm \ref{MultiStructure},  we have
      \begin{equation}\label{wz45}
      \|\hat{\Cx}-\hat{\Fxi}\|\leq L\|\Exi-\Ax\|^2+\frac{1}{1-\kappatwo}\hat{\gamma}_{\mu}(\Ff,\Ax)\|\Exi-\Ax\|^2,
      \end{equation}
      and
      \begin{equation}\label{gw45}
       \|\Sf(\Cx)\| \leq\frac{4\|D\hat{\Ff}(\Ax)\|}{1-2\kappatwo}\hat{\gamma}_{\mu}(\Ff,\Ax)\|\Exi-\Ax\|^2+lL\|\Exi-\Ax\|^2,
      \end{equation}
      where
      \begin{equation}
      \Fxi \leftarrow W_{\ddag}^{\ast} \cdot \Exi,   \,  \Sf(X)\leftarrow U^{\ast}\cdot \Ff(W_{\ddag}\cdot X), \, \Cx=(\Cx_1,\hat{\Cx})  \leftarrow W_{\ddag}^{\ast} \Bx,
      \end{equation}
%
%
%
%
%
 and $l$ is the Lipschitz constant of the function $f_n(X)$.
    \end{claim}

    \begin{proof}

By (\ref{step31}) and (\ref{w24}),
    we  have
    \begin{align*}
      \|\hat{\Cx}-\hat{\Fxi}\|&=\|W_2^{\ast}(\Bx_1-\Exi_1)+W_4^{\ast}(\hat{\Bx}-\hat{\Exi})\|  \leq L\|\Exi-\Ax\|^2+\|\hat{\Bx}-\hat{\Exi}\| \\
       & \leq L\|\Exi-\Ax\|^2+\frac{1}{1-\kappatwo}\hat{\gamma}_{\mu}(\Ff,\Ax)\|\Exi-\Ax\|^2.
    \end{align*}
By (\ref{step32}) and (\ref{U13}),  we have
    \begin{align*}
      \|\Sf(\Cx)\|&=\|U_1^{\ast}\hat{\Ff}(\Bx)+U_3^{\ast}\Ff_n(\Bx)\| \leq\|\hat{\Ff}(\Bx)\|+L\|\Exi-\Ax\|\|\Ff_n(\Bx)\| \\
      & \leq \frac{4\|D\hat{\Ff}(\Ax)\|}{1-2\kappatwo}\hat{\gamma}_{\mu}(\Ff,\Ax)\|\Exi-\Ax\|^2+lL\|\Exi-\Ax\|^2.
    \end{align*}
    \end{proof}

     Let $\Delta_k$ and $\Lambda_k$ be differential functionals calculated by (\ref{computedelta}) and (\ref{solveaijnormalized}) incrementally from $\Lambda_1=d_1$ until  
      $\Delta_{\mu}(g_n) =\bigO(1)$.  It should be noted that
         $d_1^k$ is the only differential monomial of the highest order $k$  in  $\Delta_{k}$ and no other $d_1^s$ with $s<k$  in  $\Delta_{k}$.
    \begin{claim}\label{convergx1}
    After running the first six steps in Algorithm \ref{MultiStructure}, we have
    \begin{equation}\label{x1z1}
    \|\Dx_1-\Fxi_1\|=\bigO(\|\Exi-\Ax\|^2),
    \end{equation}
    where $\Dx=(\Dx_1,\hat{\Dx}) \leftarrow (N_2(\Sf_n,\Cx), \hat{\Cx})$, $\Fxi \leftarrow W_{\ddag}^{\ast}  \cdot \xi$. 
    \end{claim}
 \begin{proof}
 It is  straightforward to check that  $\Cx$ is a simple multiple zero of the system
    \[\left\{\begin{array}{l}
               \hat{\Nf}(X)=\hat{\Sf}(X)-\hat{\Sf}(\Cx), \\
               \\
               \Nf_n(X)=\Sf_n(X)-\Sf_n(\Cx)-\sum_{k=1}^{\mu-1}\Delta_k(\Sf_n)(X_1-\Cx_1)^k.
             \end{array}
    \right.\]
    with multiplicity $\mu$ and $Dg(\Cx)$ is of normalized form (\ref{normalizeform}) (the construction is very similar to Theorem \ref{thm8}). Thus, by (\ref{keyexpansion}), we have
    \begin{align*}
&\Nf_n(X)=-\sum_{1\leq i+j-1\leq\mu-2}T_{i,j-1}\cdot D\hat{\Nf}(\Cx)^{-1}\underbrace{\hat{\Nf}(X)(X_1-\Cx_1)^i (\hat{X}-\hat{\Cx})^{j-1}}_{\mathrm{term}~\#1}\\
\notag&\ \ +\Delta_{\mu}(\Nf_n)(X_1-\Cx_1)^{\mu}+\sum_{i+j=\mu,j>0}C_{i,j}\underbrace{(X_1-\Cx_1)^i (\hat{X}-\hat{\Cx})^j}_{\mathrm{term}~\#2} +\sum_{k\geq\mu+1}\underbrace{\frac{D^k \Nf_n(\Cx)(X-\Cx)^k}{k!}}_{\mathrm{term}~\#3}\\
\notag&\ \ +\sum_{1\leq i+j-1\leq\mu-2}T_{i,j-1}\cdot \left( \sum_{k+i+j-1\geq \mu+1}D\hat{\Nf}(\Cx)^{-1}\underbrace{\frac{D^k\hat{\Nf}(\Cx)(X-\Cx)^k}{k!}(X_1-\Cx_1)^i (\hat{X}-\hat{\Cx})^{j-1}}_{\mathrm{term}~\#4} \right).
\end{align*}

Let $\Sf_n(X)=\Nf_n(X)+\Sf_n(\Cx)+\sum_{k=1}^{\mu-1}\Delta_k(\Sf_n)(X_1-\Cx_1)^k$ and $\{1,\bar{\Lambda}_{1},\ldots,\bar{\Lambda}_{\mu-1}\}$ be a reduced basis of $\mathcal{D}_{\Sf,\Fxi}$, 
    \begin{align*}
      0=\bar{\Lambda}_{\mu-1}(\Sf_n) & =\bar{\Lambda}_{\mu-1}(\Nf_n)+\Delta_{\mu-1}(\Sf_n) \\
       & =\mu\Delta_{\mu}(\Nf_n)(\Fxi_1-\Cx_1)+\Delta_{\mu-1}(\Sf_n)+\bigO(\|\Exi-\Ax\|^2)\\
       & =\mu\Delta_{\mu}(\Sf_n)(\Fxi_1-\Cx_1)+\Delta_{\mu-1}(\Sf_n)+\bigO(\|\Exi-\Ax\|^2)\\
       & =-\mu\Delta_{\mu}(\Sf_n)(N_2(\Cx_1)-\Fxi_1)+\bigO(\|\Exi-\Ax\|^2),
    \end{align*}
    because of the following facts:
    \begin{itemize}
      \item for the $\mathrm{term}~\#1$, $\bar{\Lambda}_{\mu-1}\left(\hat{\Nf}(X)(X_1-\Cx_1)^i (\hat{X}-\hat{\Cx})^{j-1}\right)=\bigO(\|\Exi-\Ax\|^2)$ for $1\leq i+j-1 \leq \mu-2$, since $\bar{\Lambda}_{\mu-1}\in\mathcal{D}_{\Sf,\Fxi}$, $\hat{\Nf}(X)=\hat{\Sf}(X)-\hat{\Sf}(\Cx)$ and $\|\hat{\Sf}(\Cx)\|=\bigO(\|\Exi-\Ax\|^2)$ by (\ref{gw45});
      \item for the $\mathrm{term}~\#2$, $\bar{\Lambda}_{\mu-1}\left((X_1-\Cx_1)^i (\hat{X}-\hat{\Cx})^{j-1}(\hat{X}-\hat{\Cx})\right)=\bigO(\|\Exi-\Ax\|^2)$, since $j>0$, $i+j-1=\mu-1$ and $\|\hat{\Cx}-\hat{\Fxi}\|=\bigO(\|\Exi-\Ax\|^2)$ by (\ref{wz45});
      \item for the $\mathrm{term}~\#3$, $\bar{\Lambda}_{\mu-1}\left((X-\Cx)^k\right)=\bigO(\|\Exi-\Ax\|^2)$ for  $k-(\mu-1)\geq2$;
      \item for the $\mathrm{term}~\#4$, $\bar{\Lambda}_{\mu-1}\left((X-\Cx)^k(X_1-\Cx_1)^i (\hat{X}-\hat{\Cx})^{j-1}\right)=\bigO(\|\Exi-\Ax\|^2)$ for $k+i+j-1-(\mu-1)\geq2$.
    \end{itemize}
  Moreover, we have $\Delta_{\mu}(\Sf_n) =\bigO(1)$.   Therefore, after running Step 6, we have
  \[\|\Dx_1-\Fxi_1\|=\|N_2(\Cx_1)-\Fxi_1\|=\bigO(\|\Exi-\Ax\|^2).\]

    \end{proof}

   %

Now we are ready to prove Theorem \ref{thm4.1}.
    \begin{proof}
      Both  $W_{\dag}$ and $W_{\ddag}$ are  unitary matrices,  according to (\ref{ksizf}) and  Claim  \ref{normalisation_2}, \ref{convergx1}, we have
    \[\|W_{\dag}\cdot W_{\ddag}\cdot \Dx-\Exi\|=\left\|\left(\begin{array}{c}
     \Dx_1-W_{\ddag}^* \cdot W_{\dag}^* \cdot\Exi_1 \\
      \hat{\Dx}-W_{\ddag}^* \cdot W_{\dag}^*\cdot \hat{\Exi}
      \end{array}\right)
    \right\|=\bigO(\|\Exi-\Ax\|^2).\]

    \end{proof}

\begin{remark}
Theorem \ref{thm4.1} can be combined with Theorem \ref{thm8} to provide an algorithm for computing an certified ball which contains a simple singular solution of $f$.  After running the first six steps in Algorithm \ref{MultiStructure}, if  the condition (\ref{separationradius}) is satisfied by $g$ at $x$, then according to Theorem {\ref{thm8}}, it has $\mu$ zeros (counting multiplicities) in the ball of radius $ r=\frac{d}{4\gamma_{\mu}^{\mux}}$
 around $x$, which indicates that the input  polynomial system $f$ has
 $\mu$ zeros (counting multiplicities) in the ball $B(N_f(\simplex), r)$.
\end{remark}

%


The difficulty  of  giving  a quantified quadratic convergence of Step 6   is due to the complicate expression of the polynomial $h_n(X)$.  In order to avoid awkward large expressions,   in what follows, we only show the proof of the  quantitative version of  Claim \ref{convergx1} for simple triple  zeros.  It is clear from  proofs given  below that   there is no significant  obstacle to  {extend} the  quantified quadratic convergence proof  of the Algorithm \ref{MultiStructure} for simple triple  zeros to simple multiple zeros of higher multiplicities.  This can also be observed  by our analysis  in Section \ref{sec3.2}, which generalizes  results in Section \ref{sec3.1} for simple tripe zeros to simple multiple zeros of higher multiplicities.

\begin{definition}\label{defkappa3}
    Let $\Nb=\max\{ \gamma_3(\Ff,\Exi)^3\|\Exi-\Ax\|, L \gamma_3(\Ff,\Exi)^2\|\Exi-\Ax\| \}$. We define the following rational functions:
\begin{align*}
\coefhat(\Nb) =&\frac{(1-2\Nb)^2}{(2(1-2\Nb)^2-1)\cdot(1-\Nb)^3},\\
\coefn(\Nb)= &\frac{(2\Nb-1)^6}{(128\Nb^6-384\Nb^5+480\Nb^4-336\Nb^3+140\Nb^2-32\Nb+1)\cdot (1-\Nb)^3},\\
\coeftemp(\Nb)=&\sqrt{1+\left(\frac{\coefhat\Nb}{1-\coefhat\Nb}\right)^2},\\
\coeA(\Nb)=&\Nb+\frac{\coefhat\Nb}{1-\coefhat\Nb}, \\
\coeB(\Nb)=&\frac{\left(16\coefhat^2\coeftemp^2\Nb^4-(16\coefhat^2\coeftemp^2+16\coefhat\coeftemp)\Nb^3+(16\coefhat\coeftemp+4)\Nb^2-4\Nb+1\right)^2}{(1-2\Nb)^2(1-2\coefhat\coeftemp\Nb)^2}\\
&\ \ \cdot\left[\left(\frac{\coefn}{3}+\frac{17\coefn\coeftemp^2}{3}\right)\Nb+\frac{7\coefn^2}{3}\Nb^2+\frac{4\coefn^2}{3}\Nb^3+\left(\frac{\coefn}{3} +\frac{7\coefn^2\Nb}{3}+\frac{8\coefn^2\Nb^2}{3}\right)\frac{\coefhat\Nb}{1-\coefhat\Nb}\right.\\
&\ \  +\frac{4\coefn^2\Nb}{3}\left(\frac{\coefhat\Nb}{1-\coefhat\Nb}\right)^2 +\frac{\coefn}{3}\cdot\frac{\coefhat\coeftemp^2\Nb}{1-\coefhat\coeftemp\Nb} \\
&\ \ +\frac{8\coefn^3\coeftemp^2\Nb\left(12\coefn^2\coeftemp^3\Nb^3+(6\coefn^2\coeftemp^2-14\coefn\coeftemp^2)\Nb^2+(4\coeftemp-8\coefn\coeftemp)\Nb+3\right)}{3(1-2\coefn\coeftemp\Nb)^3} \\
&\ \ \left. +\frac{2\coefn\coefhat^2\coeftemp^3\Nb\left(16\coefhat^2\coeftemp^2\Nb^3+(4\coefhat^2\coeftemp^2-20\coefhat\coeftemp)\Nb^2+(6-6\coefhat\coeftemp)\Nb+3\right)}{(1-2\coefhat\coeftemp\Nb)^3} \right]\\
      &\ \ +\frac{\coefn^2}{3}\left(\Nb+\frac{\coefhat\Nb}{1-\coefhat\Nb}\right) \left(\frac{2(-4\coefhat^2\coeftemp^2\Nb^2+4\coefhat\coeftemp\Nb)^2}{(1-2\Nb)^4(1-2\coefhat\coeftemp\Nb)^4}+\frac{7(-4\coefhat^2\coeftemp^2\Nb^2+4\coefhat\coeftemp\Nb)}{(1-2\Nb)^2(1-2\coefhat\coeftemp\Nb)^2}+8\right).
      \end{align*}
    \end{definition}

\begin{thm}\label{newtonmainthm}
   Given an approximate zero $\Ax$ of a system $f$ associated to a simple triple  zero  $\Exi$  of multiplicity $3$ and satisfying
 $f(\xi)=0$, $\dim \ker Df(\xi)=1$. %
%
%
    \begin{enumerate}[(1)]
    \item  If  $\Nb < \Nb_3\approx0.0137$,
      where $\Nb_3$ is the smallest positive solution of the equation:
    \[
    \coeA(\Nb)^2+\coeB (\Nb)^2=1,
    \]
 then the output of Algorithm \ref{MultiStructure}  satisfies:
 \[
 \left \|N_f(\app)-\acc \right \| < \left \| \app -\acc  \right \|.
 \]
   %
    \item  If $\Nb < \Nb_3^{\prime}\approx0.0098$,
    where $\Nb_3^{\prime}$ is the smallest positive solution  of the equation:
      \[
    \coeA (\Nb)^2+\coeB (\Nb)^2=\frac{1}{4},
        \]
    then after $k$ times of iteration we have
 \[
 \left \|N_f^k(\app)-\acc \right \| < \left( \frac{1}{2} \right)^{2^k-1}  \left \| \app -\acc  \right \|.
 \]

%
%
%

    \end{enumerate}
%
    \end{thm}
 The   proof of Theorem \ref{newtonmainthm} is based on the following facts.

    \begin{claim}\label{gammadif}
When $\Nb \leq \Nb_3$, we have
\begin{align}
\label{gammarelationbba} \hat{\gamma}_3(\Ff,\Ax) &\leq \coefhat \cdot \gamma_3(\Ff,\Exi),\\
\label{gammarelationbbb}  \gamma_{3,n}(\Ff,\Ax) &\leq \coefn \cdot \gamma_3(\Ff,\Exi).
\end{align}

\end{claim}
\begin{proof}
For $k\geq2$, by the Taylor expansion of $D^k\hat{\Ff}(\Ax)$ at $\Exi$, we have
    \begin{align*}
      &\left\|D\hat{\Ff}(\Ax)^{-1}\frac{D^k\hat{\Ff}(\Ax)}{k!}\right\| \\
      \leq& \left\|D\hat{\Ff}(\Ax)^{-1}D\hat{\Ff}(\Exi)\right\|\cdot \left\|D\hat{\Ff}(\Exi)^{-1}\left(\frac{D^k\hat{\Ff}(\Exi)}{k!}+\sum_{i\geq1}\frac{D^{k+i}\hat{\Ff}(\Exi)}{k!i!}(\Exi-\Ax)^i\right)\right\|\\
      \leq& \left\|D\hat{\Ff}(\Ax)^{-1}D\hat{\Ff}(\Exi)\right\|\cdot\left(\gamma_3(\Ff,\Exi)^{k-1}+\sum_{i\geq1}\frac{(k+i)!}{k!i!}\gamma_3(\Ff,\Exi)^{k+i-1}\|\Exi-\Ax\|^i\right)\\
      \leq& \left\|D\hat{\Ff}(\Ax)^{-1}D\hat{\Ff}(\Exi)\right\|\cdot\frac{\gamma_3(\Ff,\Exi)^{k-1}}{(1-\gamma_3(\Ff,\Exi)\|\Exi-\Ax\|)^{k+1}}\\
      \leq& \left\|D\hat{\Ff}(\Ax)^{-1}D\hat{\Ff}(\Exi)\right\|\cdot \left(\frac{1}{1-\Nb}\right)^{k+1} \gamma_3(\Ff,\Exi)^{k-1}.
    \end{align*}
Then by the definition of $\hat{\gamma}_3$ and the above inequalities, we have
    \begin{align}\label{gammarelationaaa}
      \hat{\gamma}_3(\Ff,\Ax) & \leq  \max_{k\geq2} \left\|D\hat{\Ff}(\Ax)^{-1}\frac{D^k\hat{\Ff}(\Ax)}{k!}\right\|^{\frac{1}{k-1}}\\
      \notag&\leq \max_{k\geq2}\left(\left\|D\hat{\Ff}(\Ax)^{-1}D\hat{\Ff}(\Exi)\right\|^{\frac{1}{k-1}}\cdot \left(\frac{1}{1-\Nb}\right)^{\frac{k+1}{k-1}}\gamma_3(\Ff,\Exi)\right)\\
    \notag&\leq \left\|D\hat{\Ff}(\Ax)^{-1}D\hat{\Ff}(\Exi)\right\|\cdot \left(\frac{1}{1-\Nb}\right)^3\gamma_3(\Ff,\Exi).
    \end{align}
  According to Remark \ref{lemma45hold},   for $\Nb \leq \Nb_3 \approx 0.0137$,  we can show that:
\[
\left\|\Sigp^{-1}\Sigc\right\| \leq \frac{(1-2\Nb)^2}{2(1-2\Nb)^2-1},
\]
therefore (\ref{gammarelationbba}) holds by (\ref{gammarelationaaa}) and Definition \ref{defkappa3}.

 Similar to the proof of inequality (\ref{gammarelationaaa}), we have
 \begin{equation}\label{gammarelationaab}
     \gamma_{3,n}(\Ff,\Ax)\leq \left\|\Delta_3(\Ff_n)(\Ax)^{-1}\Delta_3(\Ff_n)(\Exi)\right\|\cdot \left(\frac{1}{1-\Nb}\right)^3\gamma_3(\Ff,\Exi).
     \end{equation}
 To prove (\ref{gammarelationbbb}), we notice that:
\[
\Delta_3(\Ff_n)(\Ax)^{-1}\Delta_3(\Ff_n)(\Exi)=(1+(\Delta_3(\Ff_n)(\Exi)^{-1}\Delta_3(\Ff_n)(\Ax)-1))^{-1}.
\]
By the Taylor expansion of $\frac{\partial^3 \Ff(\Ax)}{\partial X_1^3}$ and $\frac{\partial^2 \Ff(\Ax)}{\partial X_1\partial\hat{X}}$ at $\Exi$, we have:
 \begin{align*}
  &\Delta_3(\Ff_n)(\Ax)=\frac{1}{6}\frac{\partial^3 \Ff_n(\Ax)}{\partial X_1^3}+ a_{2,\Ax} \cdot\frac{\partial^2 \Ff_n(\Ax)}{\partial X_1\partial\hat{X}}\\
      &=\frac{1}{6}\frac{\partial^3 \Ff_n(\Exi)}{\partial X_1^3}+\sum_{k\geq1}\sum_{i=0}^{k}\frac{1}{6\cdot i!(k-i)!}\frac{\partial^{k+3} \Ff_n(\Exi)}{\partial X_1^{3+i}\partial \hat{X}^{k-i}}(\Exi_1-\Ax)^i(\hat{\Exi}-\hat{\Ax})^{k-i}\\
      &\,+a_{2,\Ax} \cdot\frac{\partial^2 \Ff_n(\Exi)}{\partial X_1\partial\hat{X}}+a_{2,\Ax} \cdot\sum_{k\geq1}\sum_{i=0}^{k}\frac{1}{i!(k-i)!}\frac{\partial^{k+2} \Ff_n(\Exi)}{\partial X_1^{1+i}\partial \hat{X}^{1+k-i}}(\Exi_1-\Ax)^i(\hat{\Exi}-\hat{\Ax})^{k-i}\\
      &=\Delta_3(\Ff_n)(\Exi)+\sum_{k\geq1}\sum_{i=0}^{k}\frac{1}{6\cdot i!(k-i)!}\frac{\partial^{k+3} \Ff_n(\Exi)}{\partial X_1^{3+i}\partial \hat{X}^{k-i}}(\Exi_1-\Ax)^i(\hat{\Exi}-\hat{\Ax})^{k-i}\\
      &\, +a_{2,\Ax} \cdot\sum_{k\geq1}\sum_{i=0}^{k}\frac{1}{i!(k-i)!}\frac{\partial^{k+2} \Ff_n(\Exi)}{\partial X_1^{1+i}\partial \hat{X}^{1+k-i}}(\Exi_1-\Ax)^i(\hat{\Exi}-\hat{\Ax})^{k-i},
    \end{align*}
    where
    \[
    a_{2,\Ax} = D\hat{\Ff}(\Ax)^{-1}\frac{1}{2}\frac{\partial^2 \hat{\Ff}(\Ax)}{\partial X_1^2}.
    \]
Therefore, we have
\begin{align}\label{gammarelationdda}
      &\left\|\Delta_3(\Ff_n)(\Exi)^{-1}\Delta_3(\Ff_n)(\Ax)-1\right\| \\
     \notag \leq& \sum_{k\geq1}\sum_{i=0}^{k}\frac{(k+3)!}{6\cdot i!(k-i)!}\gamma_{3,n}(\Ff,\Exi)^{k+2}\|\Exi-\Ax\|^k\\
    \notag  &\ \ +\|a_{2,\Ax}\| \cdot\sum_{k\geq1}\sum_{i=0}^{k}\frac{(k+2)!}{i!(k-i)!}\gamma_{3,n}(\Ff,\Exi)^{k+1}\|\Exi-\Ax\|^k\\
    \notag  \leq& \frac{1}{6}\sum_{k\geq1}(k+3)(k+2)(k+1)2^k \gamma_{3,n}(\Ff,\Exi)^{k+2}\|\Exi-\Ax\|^k\\
    \notag  &\ \ +\|a_{2,\Ax}\| \cdot\sum_{k\geq1}(k+2)(k+1)2^k \gamma_{3,n}(\Ff,\Exi)^{k+1}\|\Exi-\Ax\|^k,
    \end{align}
where
 \begin{align*}
 &\|a_{2,\Ax}\| =\left\| D\hat{\Ff}(\Ax)^{-1}\frac{1}{2}\frac{\partial^2 \hat{\Ff}(\Ax)}{\partial X_1^2}\right \| \\
      \leq& \left\|D\hat{\Ff}(\Ax)^{-1}D\hat{\Ff}(\Exi)\right\|\cdot\\
      &\left\|D\hat{\Ff}(\Exi)^{-1}\left(\frac{1}{2}\frac{\partial^2 \hat{\Ff}(\Exi)}{\partial X_1^2}+\sum_{k\geq1}\sum_{i=0}^{k}\frac{1}{2\cdot i!(k-i)!}\frac{\partial^{k+2} \Ff(\Exi)}{\partial X_1^{2+i}\partial \hat{X}^{k-i}}(\Exi_1-\Ax)^i(\hat{\Exi}-\hat{\Ax})^{k-i}\right)\right\|\\
      \leq& \hat{\gamma}_3(\Ff,\Exi)+\sum_{k\geq1}\sum_{i=0}^{k}\frac{(k+2)!}{2\cdot i!(k-i)!}\hat{\gamma}_3(\Ff,\Exi)^{k+1}\|\Exi-\Ax\|^k\\
      \leq& \hat{\gamma}_3(\Ff,\Exi)+\frac{1}{2}\sum_{k\geq1}(k+2)(k+1)2^k\hat{\gamma}_3(\Ff,\Exi)^{k+1}\|\Exi-\Ax\|^k\\
      \leq &\hat{\gamma}_3(\Ff,\Exi)+\frac{-2\Nb(4\Nb^2-6\Nb+3)}{(2\Nb-1)^3}.
    \end{align*}
By (\ref{gammarelationdda}), we have:
    \begin{align*}
    &\left\|\Delta_3(\Ff_n)(\Exi)^{-1}\Delta_3(\Ff_n)(\Ax)-1\right\| \\
       \leq& \frac{-8\Nb(2\Nb^3-4\Nb^2+3\Nb-1)}{(2\Nb-1)^4}+\frac{-2\Nb(4\Nb^2-6\Nb+3)}{(2\Nb-1)^3} \cdot \frac{-4\Nb(4\Nb^2-6\Nb+3)}{(2\Nb-1)^3} \\
        &\ \ +\sum_{k\geq1}(k+2)(k+1)2^k \gamma_{3,n}(\Ff,\Exi)^{k+2}\|\Exi-\Ax\|^k\\
    \leq &\frac{-4\Nb(16\Nb^5-48\Nb^4+60\Nb^3-44\Nb^2+20\Nb-5)}{(2\Nb-1)^6}    .
    \end{align*}
  When $\Nb \leq \Nb_3$,
  \[\frac{-4\Nb(16\Nb^5-48\Nb^4+60\Nb^3-44\Nb^2+20\Nb-5)}{(2\Nb-1)^6}<1,\] then we have
    \begin{align*}
   \| \Delta_3(\Ff_n)(\Ax)^{-1}\Delta_3(\Ff_n)(\Exi)\|=&\left\|(1+(\Delta_3(\Ff_n)(\Exi)^{-1}\Delta_3(\Ff_n)(\Ax)-1))^{-1}\right\| \\
   \leq &\frac{1}{1-\frac{-4\Nb(16\Nb^5-48\Nb^4+60\Nb^3-44\Nb^2+20\Nb-5)}{(2\Nb-1)^6}}\\
   =&\frac{(2\Nb-1)^6}{(128\Nb^6-384\Nb^5+480\Nb^4-336\Nb^3+140\Nb^2-32\Nb+1)}.
    \end{align*}
    Hence, the inequality  (\ref{gammarelationbbb})  holds. 
\end{proof}

  We notice that   after running  first five steps   in  Algorithm \ref{MultiStructure}, we have
 \[\Fxi \leftarrow W_{\ddag}^{\ast} \cdot \Exi,  \quad  \Sf(X)\leftarrow U^{\ast}\cdot \Ff(W_{\ddag}\cdot X), \quad  \Cx \leftarrow W_{\ddag}^{\ast} \Bx, \quad \Bx \leftarrow (\simplex_1, N_1(\hat{\Ff},\hat{\Ax})). \]
For  unitary matrices $U$ and $W_{\ddag}$, we have $\gamma_3(\Ff,\Exi)=\gamma_3(U^*\cdot\Ff,  W_{\ddag}^{\ast}\cdot\Exi)$ and the following inequalities hold:
\begin{align*}
 \hat{\gamma}_3(\Sf,\Cx) &\leq \coefhat \cdot \gamma_3(\Sf,\Fxi) = \coefhat \cdot \gamma_3(\Ff,\Exi),\\
 \gamma_{3,n}(\Sf,\Cx) &\leq \coefn \cdot \gamma_3(\Sf,\Fxi)=\coefn \cdot \gamma_3(\Ff,\Exi).
\end{align*}

\begin{claim}\label{preestimate}
When $\Nb \leq \Nb_3$, we have
\begin{equation}\label{wzksi}
 \|\Cx-\Fxi\| \leq \coeftemp \cdot \|\Ax-\Exi\|.
 \end{equation}
\end{claim}
\begin{proof}
It is clear that
  \[
   \|\Cx-\Fxi\| =\|W_{\ddag}^{\ast}(\Bx-\Exi )\|= \| \Bx-\Exi \|.
  \]
  By Claim \ref{N1quadratic} and Claim \ref{gammadif}, when $\Nb \leq \Nb_3$, we have:
  \begin{align*}
  \|\hat{\Bx}-\hat{\xidag}\| &\leq\frac{1}{1-\hat{\gamma}_3(\Ff,\Ax)\|\Exi-\Ax\|}\hat{\gamma}_3(\Ff,\Ax)\|\Exi-\Ax\|^2 \\
  & \leq \frac{\coefhat\Nb}{1-\coefhat\Nb}\|\Exi-\Ax\|.
  \end{align*}
  By the fact that $\Bx_1=\Ax_1$, when $\Nb \leq \Nb_3$, we have
  \[
    \|\Cx-\Fxi\| =\| \Bx-\Exi \| \leq \sqrt{1+\left(\frac{\coefhat\Nb}{1-\coefhat\Nb}\right)^2}\|\Exi-\Ax\| =\coeftemp \cdot \|\Ax-\Exi\|.
  \]
\end{proof}

 In   Claim \ref{convergx1}, we have shown that
 \[\|\Dx_1-\Fxi_1\|=\bigO(\|\Exi-\Ax\|^2).\]
    Below,  we give a quantitative version of  Claim \ref{convergx1} for the simple triple  zero case.

 Following the  proof of Claim \ref{convergx1}, we consider the following system:
    \begin{equation*}
    \left\{\begin{array}{l}
               \hat{\Nf}(X)=\hat{\Sf}(X)-\hat{\Sf}(\Cx) \\
               \\
               \Nf_n(X)=\Sf_n(X)-\Sf_n(\Cx)-\sum_{k=1}^{2}\Delta_k(\Sf_n)(X_1-\Cx_1)^k
             \end{array}
    \right.
    \end{equation*}
     $\Cx$ is a simple triple zero of $\Nf(X)$ whose  Jacobian matrix is of  normalized form. The polynomial $\Nf_n(X)$ can be written as:
    \begin{align*}
    \Nf_n(X)&=-\sum_{i+j=1}T_{i,j}\cdot D\hat{\Nf}(\Cx)^{-1}\hat{\Nf}(X)(X_1-\Cx_1)^i (\hat{X}-\hat{\Cx})^j\\
    \notag&\ \ +\Delta_3(\Nf_n)(X_1-\Cx_1)^{3}+\sum_{i+j=3,j>0}C_{i,j}(X_1-\Cx_1)^i (\hat{X}-\hat{\Cx})^j\\
    \notag&\ \ +\sum_{k\geq 4}\frac{D^k \Nf_n(\Cx)(X-\Cx)^k}{k!}\\
    \notag&\ \ +\sum_{i+j=1}T_{i,j}\cdot \sum_{k\geq 3}D\hat{\Nf}(\Cx)^{-1}\frac{D^k\hat{\Nf}(\Cx)(X-\Cx)^k}{k!}(X_1-\Cx_1)^i (\hat{X}-\hat{\Cx})^j\\
    &\triangleq \Delta_3(\Nf_n)(X_1-\Cx_1)^{3}+B.
    \end{align*}
    It is clear that  $D\hat{\Nf}(\Cx)=D\hat{\Sf}(\Cx),\ D^k \Nf(\Cx)=D^k \Sf(\Cx)$ for $k\geq3$,
    \begin{align*}
  C_{2,1}&=\frac{1}{2}\frac{\partial^3 \Sf_n(\Cx)}{\partial X_1^2\partial \hat{X}}-\frac{\partial^2 \Sf_n(\Cx)}{\partial X_1\partial \hat{X}}\cdot D \hat{\Sf}(\Cx)^{-1}\frac{\partial^2 \hat{\Sf}(\Cx)}{\partial X_1\partial \hat{X}}-\frac{1}{2}\frac{\partial^2 \Sf_n(\Cx)}{\partial \hat{X}^2}\cdot D\hat{\Sf}(\Cx)^{-1}\frac{1}{2}\frac{\partial^2 \hat{\Sf}(\Cx)}{\partial X_1^2},\\
  C_{1,2}&=\frac{1}{2}\frac{\partial^3 \Sf_n(\Cx)}{\partial X_1\partial \hat{X}^2}-\frac{\partial^2 \Sf_n(\Cx)}{\partial X_1\partial \hat{X}}\cdot D\hat{\Sf}(\Cx)^{-1}\frac{1}{2}\frac{\partial^2 \hat{\Sf}(\Cx)}{\partial \hat{X}^2}-\frac{1}{2}\frac{\partial^2 \Sf_n(\Cx)}{\partial \hat{X}^2}\cdot D\hat{\Sf}(\Cx)^{-1}\frac{\partial^2 \hat{\Sf}(\Cx)}{\partial X_1\partial \hat{X}},\\
  C_{0,3}&=\frac{1}{6}\frac{\partial^3 \Sf_n(\Cx)}{\partial \hat{X}^3}-\frac{1}{2}\frac{\partial^2 \Sf_n(\Cx)}{\partial \hat{X}^2}\cdot D\hat{\Sf}(\Cx)^{-1}\frac{1}{2}\frac{\partial^2 \hat{\Sf}(\Cx)}{\partial \hat{X}^2},\\
  T_{1,0}&= -\frac{\partial^2 \Sf_n(\Cx)}{\partial X_1\partial \hat{X}},\\
  T_{0,1}&=-\frac{1}{2}\frac{\partial^2 \Sf_n(\Cx)}{\partial \hat{X}^2}.
  \end{align*}
 By the definition of $\hat{\gamma}_3(\Sf,\Cx)$ and $\gamma_{3,n}$, we also have the following facts:
 \begin{align*}
    \|\Delta_3(\Sf_n)^{-1}C_{2,1}\|\leq& 3\gamma_{3,n}(\Sf,\Cx)^2+2\gamma_{3,n}(\Sf,\Cx)\cdot2\hat{\gamma}_3(\Sf,\Cx)+\gamma_{3,n}(\Sf,\Cx)\cdot \hat{\gamma}_3(\Sf,\Cx)\\
    \leq& 8\gamma_3(\Sf,\Cx)^2,\\
    \|\Delta_3(\Sf_n)^{-1}C_{1,2}\|\leq& 3\gamma_{3,n}(\Sf,\Cx)^2+2\gamma_3(\Sf,\Cx)\cdot \hat{\gamma}_3(\Sf,\Cx)+\gamma_{3,n}(\Sf,\Cx)\cdot2\hat{\gamma}_3(\Sf,\Cx)\\
    \leq& 7\gamma_3(\Sf,\Cx)^2,\\
    \|\Delta_3(\Sf_n)^{-1}C_{0,3}\|\leq& \gamma_{3,n}(\Sf,\Cx)^2+\gamma_{3,n}(\Sf,\Cx)\cdot \hat{\gamma}_3(\Sf,\Cx)\\
    \leq& 2\gamma_3(\Sf,\Cx)^2,\\
    \|\Delta_3(\Sf_n)^{-1}T_{1,0}\|\leq& 2\gamma_{3,n}(\Sf,\Cx),\\
    \|\Delta_3(\Sf_n)^{-1}T_{0,1}\|\leq& \gamma_{3,n}(\Sf,\Cx).
  \end{align*}

  Let $\{1,\bar{\Lambda}_{1},\bar{\Lambda}_{2}\}$ be the reduced basis of $\mathcal{D}_{\Sf,\Fxi}$ and $\bar{\Lambda}_1=d_1+a_1 d_2$, then
   \[\bar{\Delta}_2=d_1^2+a_1d_1d_2+a_1^2d_2^2,\]
    \[\bar{\Lambda}_2=\bar{\Delta}_2+a_2d_2,\]
    where $a_1=D\hat{\Sf}(\Fxi)^{-1}\frac{\partial \hat{\Sf}(\Fxi)}{\partial X_1},\ a_2=D\hat{\Sf}(\Fxi)^{-1}\bar{\Delta}_2(\hat{\Sf})(\Fxi)$. 

    By  Taylor expansion of $\frac{\partial \hat{\Sf}(\Fxi)}{\partial X_1}$ at $\Cx$
    \[\frac{\partial \hat{\Sf}(\Fxi)}{\partial X_1}=\sum_{k\geq2}\sum_{i=1}^{k}\frac{1}{(i-1)!(k-i)!}\frac{\partial^k \hat{\Sf}(\Cx)}{\partial X_1^i\partial \hat{X}^{k-i}}(\Fxi_1-\Cx_1)^{i-1}(\hat{\Fxi}-\hat{\Cx})^{k-i},\]
  and 
  $\gamma_3(\Sf,\Fxi)=\gamma_3(\Ff,\Exi)$, we have
      \begin{align*}
      \|a_1\| &= \left\|D\hat{\Sf}(\Fxi)^{-1}\frac{\partial \hat{\Sf}(\Fxi)}{\partial X_1}\right\|\\
      &\leq \left\|D\hat{\Sf}(\Fxi)^{-1}D\hat{\Sf}(\Cx)\right\|\cdot \left\|D\hat{\Sf}(\Cx)^{-1}\frac{\partial \hat{\Sf}(\Fxi)}{\partial X_1}\right\|\\
      &\leq \left\|D\hat{\Sf}(\Fxi)^{-1}D\hat{\Sf}(\Cx)\right\|\cdot \sum_{k\geq2}\sum_{i=1}^{k}\frac{k!}{(i-1)!(k-i)!}\hat{\gamma}_3(\Sf,\Cx)^{k-1}\|\Fxi-\Cx\|^{k-1}\\
      &=\left\|D\hat{\Sf}(\Fxi)^{-1}D\hat{\Sf}(\Cx)\right\|\cdot\frac{4\coefhat\gamma_3(\Ff,\Exi)\coeftemp\|\Exi-\Ax\|(1-\coefhat\gamma_3(\Ff,\Exi)\coeftemp\|\Exi-\Ax\|)}{(1-2\coefhat\gamma_3(\Ff,\Exi)\coeftemp\|\Exi-\Ax\|)^2}\\
      &\leq\frac{1}{(1-2\Nb)^2} \cdot \frac{4\coefhat\coeftemp\Nb(1-\coefhat\coeftemp\Nb)}{(1-2\coefhat\coeftemp\Nb)^2}\\
      &=\frac{4\coefhat\coeftemp\Nb(1-\coefhat\coeftemp\Nb)}{(1-2\Nb)^2 (1-2\coefhat\coeftemp\Nb)^2},
          \end{align*}
          and
      \begin{align*}
      \|a_2\| & =\left\|D\hat{\Sf}(\Fxi)^{-1}\bar{\Delta}_2(\hat{\Sf})(\Fxi)\right\|\\
      &\leq \left\|D\hat{\Sf}(\Fxi)^{-1}\frac{1}{2}\frac{\partial^2\hat{\Sf}(\Fxi)}{\partial X_1^2}\right\|+\|a_1\|\cdot\left\|D\hat{\Sf}(\Fxi)^{-1}\frac{\partial^2\hat{\Sf}(\Fxi)}{\partial X_1\partial \hat{X}}\right\|\\
      &\quad+\|a_1\|^2\cdot\left\|D\hat{\Sf}(\Fxi)^{-1}\frac{1}{2}\frac{\partial^2\hat{\Sf}(\Fxi)}{\partial \hat{X}^2}\right\|\\
      &\leq \gamma_3(\Sf,\Fxi)+2\|a_1\|\gamma_3(\Sf,\Fxi)+\|a_1\|^2\gamma_3(\Sf,\Fxi).
    \end{align*}

    \begin{claim}\label{theoremtwelvetemp} We have
    \[
     N_2(\Sf_n,\Cx)-\Fxi_1=\frac{1}{3}\Delta_{3}(\Sf_n)^{-1}\bar{\Lambda}_{2}(B).
    \]
    \end{claim}
    \begin{proof}
    As $\bar{\Lambda}_2=\bar{\Delta}_2+a_2d_2$, 
    applying  $\bar{\Lambda}_2$ on both sides of the equation:
    \[
    \Sf_n(X)=\Nf_n(X)+\Sf_n(\Cx)+\sum_{k=1}^{2}\Delta_k(\Sf_n)(X_1-\Cx_1)^k,
    \]
    we have:
     \begin{align*}
      0=\bar{\Lambda}_{2}(\Sf_n) & =\bar{\Lambda}_{2}(\Nf_n)+\Delta_{2}(\Sf_n) \\
       & =3\Delta_{3}(\Nf_n)(\Fxi_1-\Cx_1)+\Delta_{2}(\Sf_n)+\bar{\Lambda}_{2}(B)\\
       & =3\Delta_{3}(\Sf_n)(\Fxi_1-\Cx_1)+\Delta_{2}(\Sf_n)+\bar{\Lambda}_{2}(B).
    \end{align*}
    Therefore, we have
\[
     N_2(\Sf_n,\Cx)-\Fxi_1=\Cx_1-\frac{1}{3}\Delta_{3}(\Sf_n)^{-1}\Delta_{2}(\Sf_n)-\Fxi_1=\frac{1}{3}\Delta_{3}(\Sf_n)^{-1}\bar{\Lambda}_{2}(B).
    \]
    \end{proof}


    \begin{claim}\label{newone}
 When $\Nb \leq \Nb_3$, we have
 \begin{equation}\label{x1z1}
 |\Dx_1-\Fxi_1| \leq \coeB(\Nb) \| \Ax-\Exi \|.
 \end{equation}
  \end{claim}
  \begin{proof}
   By Claim \ref{theoremtwelvetemp} and above arguments, we have
     \begin{align*}
      &\left| N_2(\Sf_n,\Cx)-\Fxi_1\right|=\left|\frac{1}{3}\Delta_{3}(\Sf_n)^{-1}\bar{\Lambda}_{2}(B)\right|\\
      =& \left|\frac{1}{3}\Delta_{3}(\Sf_n)^{-1}\bar{\Lambda}_{2}\left(T_{1,0}\cdot D\hat{\Nf}(\Cx)^{-1}\hat{\Nf}(X)(X_1-\Cx_1)+T_{0,1} \cdot D\hat{\Nf}(\Cx)^{-1}\hat{\Nf}(X)(\hat{X}-\hat{\Cx})\right)\right|\\
    \notag&\ \ +\left|\frac{1}{3}\Delta_{3}(\Sf_n)^{-1}\bar{\Lambda}_{2}\left(\sum_{i+j=3,j>0}C_{i,j}(X_1-\Cx_1)^i (\hat{X}-\hat{\Cx})^j\right)\right|\\
    \notag&\ \ +\left|\frac{1}{3}\Delta_{3}(\Sf_n)^{-1}\bar{\Lambda}_{2}\left(\sum_{k\geq 4}\frac{D^k \Nf_n(\Cx)(X-\Cx)^k}{k!}\right)\right|\\
    \notag&\ \ +\left|\frac{1}{3}\Delta_{3}(\Sf_n)^{-1}\bar{\Lambda}_{2}\left(\sum_{i+j=1}T_{i,j}\cdot \sum_{k\geq 3}D\hat{\Nf}(\Cx)^{-1}\frac{D^k\hat{\Nf}(\Cx)(X-\Cx)^k}{k!}(X_1-\Cx_1)^i (\hat{X}-\hat{\Cx})^j\right)\right|\\
      =& \frac{1}{3}\left\|\Delta_{3}(\Sf_n)^{-1}T_{0,1}\right\|\cdot \left\|D\hat{\Nf}(\Cx)^{-1}\hat{\Nf}(\Fxi)\right\|\|a_2\|\\
    \notag&\ \ +\frac{1}{3}\left\|\Delta_{3}(\Sf_n)^{-1}\cdot \left(C_{2,1}+C_{1,2}\left(a_1+a_2(\Fxi_1-\Cx_1)\right)+C_{0,3}\left(a_1^2+2a_2(\hat{\Fxi}-\hat{\Cx})\right)\right)\right\|\|\hat{\Fxi}-\hat{\Cx}\|\\
    \notag&\ \ +\frac{1}{3}\cdot\sum_{i+j=3,j>0}\left\|\Delta_{3}(\Sf_n)^{-1}C_{i,j}\right\||\Fxi_1-\Cx_1|^i \|\hat{\Fxi}-\hat{\Cx}\|^{j-1}\|a_2\|\\
    \notag&\ \ +\frac{1}{3}\cdot\frac{1}{2}\sum_{k\geq 4}\sum_{i=2}^{k}\frac{i(i-1)}{i!(k-i)!}\left\|\Delta_{3}(\Sf_n)^{-1}\frac{\partial^k \Nf_n(\Cx)}{\partial X_1^i \hat{X}^{k-i}}\right\||\Fxi_1-\Cx_1|^{i-2}\|\hat{\Fxi}-\hat{\Cx}\|^{k-i}\\
    \notag&\ \ +\frac{1}{3}\cdot \|a_1\|\sum_{k\geq 4}\sum_{i=1}^{k-1}\frac{i(k-i)}{i!(k-i)!}\left\|\Delta_{3}(\Sf_n)^{-1}\frac{\partial^k \Nf_n(\Cx)}{\partial X_1^i \hat{X}^{k-i}}\right\||\Fxi_1-\Cx_1|^{i-1}\|\hat{\Fxi}-\hat{\Cx}\|^{k-i-1}\\
    \notag&\ \ +\frac{1}{3}\cdot \frac{1}{2}\|a_1\|^2\sum_{k\geq 4}\sum_{i=0}^{k-2}\frac{(k-i)(k-i-1)}{i!(k-i)!}\left\|\Delta_{3}(\Sf_n)^{-1}\frac{\partial^k \Nf_n(\Cx)}{\partial X_1^i \hat{X}^{k-i}}\right\||\Fxi_1-\Cx_1|^{i}\|\hat{\Fxi}-\hat{\Cx}\|^{k-i-2}\\
    \notag&\ \ +\frac{1}{3}\cdot \|a_2\|\sum_{k\geq 4}\sum_{i=0}^{k-1}\frac{(k-i)}{i!(k-i)!}\left\|\Delta_{3}(\Sf_n)^{-1}\frac{\partial^k \Nf_n(\Cx)}{\partial X_1^i \hat{X}^{k-i}}\right\||\Fxi_1-\Cx_1|^{i}\|\hat{\Fxi}-\hat{\Cx}\|^{k-i-1}\\
    \notag&\ \ +\frac{1}{3}\sum_{i+j=1}\left\|\Delta_{3}(\Sf_n)^{-1} T_{i,j}\right\|\cdot\sum_{k\geq 3}\left( \frac{1}{2}\sum_{l=2}^{k}\frac{l(l-1)}{l!(k-l)!}\left\|D\hat{\Nf}(\Cx)^{-1} \frac{\partial^k \hat{\Nf}(\Cx)}{\partial X_1^l \hat{X}^{k-l}}\right\||\Fxi_1-\Cx_1|^{l-2}\|\hat{\Fxi}-\hat{\Cx}\|^{k-l}\right.\\
    \notag&\ \ +\|a_1\|\sum_{l=1}^{k-1}\frac{l(k-l)}{l!(k-l)!}\left\|D\hat{\Nf}(\Cx)^{-1} \frac{\partial^k \hat{\Nf}(\Cx)}{\partial X_1^l \hat{X}^{k-l}}\right\||\Fxi_1-\Cx_1|^{l-1}\|\hat{\Fxi}-\hat{\Cx}\|^{k-l-1}\\
    \notag&\ \ +\frac{1}{2}\|a_1\|^2\sum_{l=0}^{k-2}\frac{(k-l)(k-l-1)}{l!(k-l)!}\left\|D\hat{\Nf}(\Cx)^{-1} \frac{\partial^k \hat{\Nf}(\Cx)}{\partial X_1^l \hat{X}^{k-l}}\right\||\Fxi_1-\Cx_1|^{l}\|\hat{\Fxi}-\hat{\Cx}\|^{k-l-2}\\
    \notag&\ \ +\left.\|a_2\|\sum_{l=0}^{k-1}\frac{(k-l)}{l!(k-l)!}\left\|D\hat{\Nf}(\Cx)^{-1} \frac{\partial^k \hat{\Nf}(\Cx)}{\partial X_1^l \hat{X}^{k-l}}\right\||\Fxi_1-\Cx_1|^{l}\|\hat{\Fxi}-\hat{\Cx}\|^{k-l-1}\right)|\Fxi_1-\Cx_1|^i \|\hat{\Fxi}-\hat{\Cx}\|^j\\
      \leq& \frac{1}{3}\gamma_3(\Sf,\Cx) \left(\|\hat{\Fxi}-\hat{\Cx}\|+ \frac{1}{1-\hat{\gamma}_3(\Sf,\Cx)\|\Fxi-\Cx\|}\hat{\gamma}_3(\Sf,\Cx)\|\Fxi-\Cx\|^2\right)\|a_2\|\\
      &\ \ +\frac{1}{3}\left(8\gamma_3(\Sf,\Cx)^2+7\gamma_3(\Sf,\Cx)^2\|a_1\|+ 2\gamma_3(\Sf,\Cx)^2\|a_1\|^2\right)\|\hat{\Fxi}-\hat{\Cx}\|\\
      &\ \ +\frac{7}{3}\gamma_3(\Sf,\Cx)^2\|a_2\||\Fxi_1-\Cx_1|\|\hat{\Fxi}-\hat{\Cx}\|+\frac{4}{3}\gamma_3(\Sf,\Cx)^2\|a_2\|\|\hat{\Fxi}-\hat{\Cx}\|^2\\
      &\ \ +\frac{17}{3}\gamma_3(\Sf,\Cx)\|\Fxi-\Cx\|^2\|a_2\|\\
      &\ \ +\frac{1}{6}\sum_{k\geq 4}\sum_{i=2}^{k}\frac{i(i-1)k!}{i!(k-i)!}\gamma_3(\Sf,\Cx)^{k-1}|\Fxi_1-\Cx_1|^{i-2}\|\hat{\Fxi}-\hat{\Cx}\|^{k-i}\\
      &\ \ +\frac{1}{3}\|a_1\|\sum_{k\geq 4}\sum_{i=1}^{k-1}\frac{i(k-i)k!}{i!(k-i)!}\gamma_3(\Sf,\Cx)^{k-1}|\Fxi_1-\Cx_1|^{i-1}\|\hat{\Fxi}-\hat{\Cx}\|^{k-i-1}\\
      &\ \ +\frac{1}{6}\|a_1\|^2\sum_{k\geq 4}\sum_{i=0}^{k-2}\frac{(k-i)(k-i-1)k!}{i!(k-i)!}\gamma_3(\Sf,\Cx)^{k-1}|\Fxi_1-\Cx_1|^{i}\|\hat{\Fxi}-\hat{\Cx}\|^{k-i-2}\\
      &\ \ +\frac{1}{3}\|a_2\|\sum_{k\geq 4}\sum_{i=0}^{k-1}\frac{(k-i)k!}{i!(k-i)!}\gamma_3(\Sf,\Cx)^{k-1}|\Fxi_1-\Cx_1|^{i}\|\hat{\Fxi}-\hat{\Cx}\|^{k-i-1}\\
      &\ \ +\gamma_3(\Sf,\Cx)\cdot \sum_{k\geq 3}\left(
      \frac{1}{2}\sum_{l=2}^{k}\frac{l(l-1)k!}{l!(k-l)!}\hat{\gamma}_3(\Sf,\Cx)^{k-1}\|\Fxi-\Cx\|^{k-2}\right.\\
      &\ \ \ \  +\|a_1\|\sum_{l=1}^{k-1}\frac{l(k-l)k!}{l!(k-l)!}\hat{\gamma}_3(\Sf,\Cx)^{k-1}\|\Fxi-\Cx\|^{k-2}\\
      &\ \ \ \ +\frac{1}{2}\|a_1\|^2\sum_{l=0}^{k-2}\frac{(k-l)(k-l-1)k!}{l!(k-l)!}\hat{\gamma}_3(\Sf,\Cx)^{k-1}\|\Fxi-\Cx\|^{k-2}\\
      &\ \ \ \ \left.+\|a_2\|\sum_{l=0}^{k-1}\frac{(k-l)k!}{l!(k-l)!}\hat{\gamma}_3(\Sf,\Cx)^{k-1}\|\Fxi-\Cx\|^{k-1}\right) \|\Fxi-\Cx\|\\
    =& \frac{1}{3}\gamma_3(\Sf,\Cx) \left(\|\hat{\Fxi}-\hat{\Cx}\|+ \frac{1}{1-\hat{\gamma}_3(\Sf,\Cx)\|\Fxi-\Cx\|}\hat{\gamma}_3(\Sf,\Cx)\|\Fxi-\Cx\|^2\right)\|a_2\|\\
      &\ \ +\frac{1}{3}\left(8\gamma_3(\Sf,\Cx)^2+7\gamma_3(\Sf,\Cx)^2\|a_1\|+ 2\gamma_3(\Sf,\Cx)^2\|a_1\|^2\right)\|\hat{\Fxi}-\hat{\Cx}\|\\
      &\ \ +\frac{7}{3}\gamma_3(\Sf,\Cx)^2\|a_2\||\Fxi_1-\Cx_1|\|\hat{\Fxi}-\hat{\Cx}\|+\frac{4}{3}\gamma_3(\Sf,\Cx)^2\|a_2\|\|\hat{\Fxi}-\hat{\Cx}\|^2\\
      &\ \ +\frac{17}{3}\gamma_3(\Sf,\Cx)\|\Fxi-\Cx\|^2\|a_2\|\\
      &\ \ +\frac{1}{6}\sum_{k\geq 4}2^{k-2}k(k-1)\gamma_3(\Sf,\Cx)^{k-1}\|\Fxi-\Cx\|^{k-2}\\
      &\ \ +\frac{1}{3}\|a_1\|\sum_{k\geq 4}2^{k-2}k(k-1)\gamma_3(\Sf,\Cx)^{k-1}\|\Fxi-\Cx\|^{k-2}\\
      &\ \ +\frac{1}{6}\|a_1\|^2\sum_{k\geq 4}2^{k-2}k(k-1)\gamma_3(\Sf,\Cx)^{k-1}\|\Fxi-\Cx\|^{k-2}\\
      &\ \ +\frac{1}{3}\|a_2\|\sum_{k\geq 4}2^{k-1}k\gamma_3(\Sf,\Cx)^{k-1}\|\Fxi-\Cx\|^{k-1}\\
      &\ \ +\frac{1}{2}\gamma_3(\Sf,\Cx)\sum_{k\geq 3}(1+2\|a_1\|+\|a_1\|^2)2^{k-2}k(k-1)\hat{\gamma}_3(\Sf,\Cx)^{k-1}\|\Fxi-\Cx\|^{k-1}\\
      &\ \ +\gamma_3(\Sf,\Cx)\sum_{k\geq 3}\|a_2\|2^{k-1}k\hat{\gamma}_3(\Sf,\Cx)^{k-1}\|\Fxi-\Cx\|^{k}\\
    \leq& \frac{\coefn}{3}(1+2\|a_1\|+\|a_1\|^2)\left(\Nb+\frac{\coefhat\Nb}{1-\coefhat\Nb}\right)\|\Exi-\Ax\|\\
      &\ \ +\frac{\coefn}{3}(1+2\|a_1\|+\|a_1\|^2)\cdot\frac{\coefhat\coeftemp^2\Nb}{1-\coefhat\coeftemp\Nb}\|\Exi-\Ax\|\\
      &\ \ +\frac{\coefn^2}{3}(8+7\|a_1\|+2\|a_1\|^2)\left(\Nb+\frac{\coefhat\Nb}{1-\coefhat\Nb}\right)\|\Exi-\Ax\|\\
      &\ \ +\frac{7\coefn^2}{3}(1+2\|a_1\|+\|a_1\|^2)\left(\Nb^2+\frac{\coefhat\Nb^2}{1-\coefhat\Nb}\right)\|\Exi-\Ax\|\\
      &\ \ +\frac{4\coefn^2}{3}(1+2\|a_1\|+\|a_1\|^2)\Nb\left(\Nb+\frac{\coefhat\Nb}{1-\coefhat\Nb}\right)^2\|\Exi-\Ax\|\\
      &\ \ +\frac{17\coefn}{3}(1+2\|a_1\|+\|a_1\|^2)\coeftemp^2\Nb\|\Exi-\Ax\|\\
      &\ \ +\frac{1}{6}(1+2\|a_1\|+\|a_1\|^2)\cdot\frac{8\coefn^3\coeftemp^2\Nb(12\coefn^2\coeftemp^2\Nb^2-16\coefn\coeftemp\Nb+6)}{(1-2\coefn\coeftemp\Nb)^3}\|\Exi-\Ax\|\\
      &\ \ +\frac{1}{3}(1+2\|a_1\|+\|a_1\|^2)\Nb\cdot\frac{8\coefn^3\coeftemp^3\Nb(4-6\coefn\coeftemp\Nb)}{(1-2\coefn\coeftemp\Nb)^2}\|\Exi-\Ax\|\\
      &\ \ +\frac{\coefn}{2}(1+2\|a_1\|+\|a_1\|^2)\cdot\frac{4\coefhat^2\coeftemp^2\Nb(4\coefhat^2\coeftemp^2\Nb^2-6\coefhat\coeftemp\Nb+3)}{(1-2\coefhat\coeftemp\Nb)^3}\|\Exi-\Ax\|\\
      &\ \ +(1+2\|a_1\|+\|a_1\|^2)\coefn\Nb\cdot\frac{4\coefhat^2\coeftemp^3\Nb(3-4\coefhat\coeftemp\Nb)}{(1-2\coefhat\coeftemp\Nb)^2}\|\Exi-\Ax\|\\
      = & \coeB(\Nb) \| \Ax-\Exi \|.
    \end{align*}
\end{proof}

 When $\Nb < \Nb_3\approx0.0137$,  by Claim \ref{normalisation_2},  \ref{newone} and Definition \ref{defkappa3},  we have:
    \begin{align*}
    \|\hat{\Dx}-\hat{\Fxi}\| &\leq L\|\Exi-\Ax\|^2+\frac{1}{1-\kappatwo}\hat{\gamma}_3(\Ff,\Ax)\|\Exi-\Ax\|^2\\
    & \leq \left(\Nb+\frac{\coefhat\Nb}{1-\coefhat\Nb}\right)\| \Ax-\Exi\| \\
    & = \coeA(\Nb) \left \| \Ax-\Exi \right \|.
    \end{align*}
    Now we can complete the proof for Theorem \ref{newtonmainthm}.
    \begin{proof}
     \begin{enumerate}[(1)]
    \item  When $\Nb < \Nb_3\approx0.0137$,  $\coeA (\Nb)^2+\coeB (\Nb)^2<1$, by  Claim \ref{newone}, we have
    \begin{align*}
 \| N_{\Ff}(\Ax) -\Exi \|=&\|W_{\dag}\cdot W_{\ddag}\cdot( \Dx-\Fxi)\|\\
 =& \| \Dx-\Fxi \| \\
 =&  \sqrt{| \Dx_1-\Fxi_1 |^2 + \| \hat{\Dx}-\hat{\Fxi} \|^2} \\
                       \leq & \sqrt{ \coeA (\Nb)^2+\coeB (\Nb)^2} \left \| \Ax-\Exi \right \| \\
                       <& \left \| \Ax-\Exi \right \|.
    \end{align*}
    \item  When $\Nb < \Nb_3^{\prime}\approx0.0098$,  $\coeA (\Nb)^2+\coeB (\Nb)^2<\frac{1}{4}$,  by Claim \ref{newone}, we have
    \begin{align*}
     \| N_{\Ff}(\Ax) -\Exi \|=&\|W_{\dag}\cdot W_{\ddag}\cdot( \Dx-\Fxi)\|\\
 =& \| \Dx-\Fxi \| \\
    =&  \sqrt{| \Dx_1-\Fxi_1 |^2 + \| \hat{\Dx}-\hat{\Fxi} \|^2} \\
                       \leq&  \sqrt{ \coeA (\Nb)^2+\coeB (\Nb)^2} \left \| \Ax-\Exi \right \| \\
                        < &\frac{1}{2}\left \| \Ax-\Exi \right \|.
    \end{align*}
    Hence, the following inequality is true for $k=1$:
    \[
     \left \| N_{\Ff}^{k}(\Ax)-\Exi \right \| < \left( \frac{1}{2} \right)^{2^k-1}  \left \| \Ax-\Exi \right \|.
    \]
    For $k \geq 2$, assume by induction that
    \[
     \left \| N_{\Ff}^{k-1}(\Ax)-\Exi \right \| < \left( \frac{1}{2} \right)^{2^{k-1}-1}  \left \| \Ax-\Exi \right \|.
    \]
    Let $\Nb^{(k-1)}=\gamma_3(\Ff,\Exi)^3\|N_{\Ff}^{k-1}(\Ax)-\Exi\|$. For $0<\Nb<\Nb_3^{\prime}$, we have $\Nb^{(k-1)}<\Nb$ and $\frac{\sqrt{ \coeA (\Nb)^2+\coeB (\Nb)^2}}{\Nb}$ is increasing. Therefore, by  induction we have
    \begin{align*}
    & \left \| N_{\Ff}^{k}(\Ax)-\Exi \right \| \\
    < &\frac{\sqrt{ \coeA (\Nb^{(k-1)})^2+\coeB (\Nb^{(k-1)})^2}\gamma_3(\Ff,\Exi)^3}{\Nb^{(k-1)}} \left \| N_{\Ff}^{k-1}(\Ax)-\Exi \right \| ^2 \\
    < &\frac{\sqrt{ \coeA (\Nb)^2+\coeB (\Nb)^2}\gamma_3(\Ff,\Exi)^3}{\Nb} \left \| N_{\Ff}^{k-1}(\Ax)-\Exi \right \| ^2 \\
    < & \frac{\sqrt{ \coeA (\Nb)^2+\coeB (\Nb)^2}\gamma_3(\Ff,\Exi)^3}{\Nb} \left( \frac{1}{2} \right)^{2^k-2}  \left \| \Ax-\Exi \right \|^2 \\
    \leq &\left( \frac{1}{2} \right)^{2^k-1}  \left \| \Ax-\Exi \right \|.
    \end{align*}
    \end{enumerate}
    \end{proof}

\begin{remark}
The unitary transformations in Algorithm \ref{MultiStructure} may convert a sparse polynomial system into a dense polynomial system, therefore,  the computations of the modified Newton iterations {may} become more {costly}.  Hence, in our implementation, we  use the chain rule to avoid storing  or {computing with}  dense polynomial systems obtained after performing unitary transformations.
For example, suppose
$g(X)=U^{\ast}\cdot \Ff(W_{\dag}\cdot X)$,      
then we have
{\begin{equation}\label{chain0}
D g(X)=U^{\ast}\cdot D\Ff(W_{\dag}\cdot X)\cdot W_{\dag}.
\end{equation}}
Let $y=W_{\dag}^{\ast}\simplex$,   we have
\begin{equation}\label{chain1}
Dg(y)=U^{\ast} \cdot Df(z)\cdot W_{\dag}.
\end{equation}
Instead of evaluating $Dg(X)$ at $y$, we evaluate $Df(X)$ at $z$ and then perform matrix multiplications,  which avoids storing and computing the dense system $Dg(X)$.

Similarly,  as we have already demonstrated  in \cite[Example 3.1]{LZ:2012},  instead of computing and storing dense differential functionals $\Delta_k$ and $\Lambda_k$, we compute polynomials $L_k(g)$ and $P_k(g)$ as
\begin{equation}\label{chain2}
P_k(g)=\sum_{j=1}^{k-1}\frac{j}{k}\cdot D(L_{k-j}(g))\cdot\mathbf{a}_j\mbox{ and }L_k(g)=P_k(g)+Dg\cdot\mathbf{a}_k,
\end{equation}
where $L_k$ and $P_k$ are differential operators corresponding to $\Delta_k$ and $\Lambda_k$ respectively, $\mathbf{a}_1=(1,0,\ldots,0)^T$ and $\mathbf{a}_k=(0, a_{k,2},\ldots,a_{k,n})^T$.
 The polynomial system $P_k(g)$ and the value $\Delta_{k}(\Sf)$ can be obtained by applying the chain rule (\ref{chain0}) and (\ref{chain1})  recursively to (\ref{chain2}).

\end{remark}

\begin{remark}
The Maple codes of three Algorithms and test results are available via request.

 Although the algorithms and  proofs of quadratic convergence  given in the paper are for polynomial systems with exact simple multiple zeros,  examples are given to demonstrate  that our algorithms are also applicable to  analytic systems and polynomial systems with  a cluster of simple roots.
\end{remark}

\bibliographystyle{plain}
\bibliography{linan,wuxiaoli,zhi}

\end{document}